\documentclass[11pt]{article}
\usepackage{amsmath,amssymb,amsthm,bm,graphicx,color,epsfig,enumerate,caption}
\usepackage{mathtools}

\usepackage{tikz}
\usetikzlibrary{shapes,arrows,matrix,patterns,positioning}
\usepackage[norelsize,boxed,linesnumbered,vlined,algo2e]{algorithm2e} 
\usepackage{algorithm}
\usepackage{algorithmic}
\usepackage{amssymb}
\usepackage{epstopdf}
\usepackage{float}
\usepackage{hyperref}
\usepackage{booktabs}
\usepackage{multirow}
\usepackage{array}
\usepackage{bm}
\usepackage{cases}
\usepackage{comment}

\newtheorem{theorem}{Theorem}[section]

\newtheorem{definition}[theorem]{Definition}

\newcommand{\R}{\mathbb{R}}
\newcommand{\N}{\mathcal{N}}
\newcommand{\T}{\mathcal{T}}
\newcommand{\LL}{\mathcal{L}^C}
\newcommand{\WD}{\mathcal{WD}}
\newcommand{\WS}{\mathcal{WS}}

\newcommand{\defeq}{\vcentcolon=}
\newcommand{\E}{\mathrm{E}}

\newcommand{\Var}{\mathrm{Var}}

\newcommand{\SNR}{\text{SNR}}

\usepackage{mathtools}

\makeatletter
\def\CatchFBT@Fin@l#1[#2]{%
   \begingroup
      \makeatletter #2%
      \scantokens\expandafter{%
         \expandafter\CatchFBT@tok\expandafter{\the\CatchFBT@tok}}%
      \CatchFBT@IsAToken{#1}
         {\global#1\expandafter{\the\CatchFBT@tok}}
         {\xdef#1{\the\CatchFBT@tok}}%
      \ifx\CatchFBT@tok#1\else\global\CatchFBT@tok{}\fi
   \endgroup
}
\makeatother

\begin{document}

\title{Multiresolution Mode Decomposition\\ for Adaptive Time Series Analysis}

\author{Haizhao Yang\\
  \vspace{0.1in}\\
  Department of Mathematics, Purdue University, US\\
}

\date{August 2019}
\maketitle

\begin{abstract}
This paper proposes the \emph{multiresolution mode decomposition} (MMD) as a novel model for adaptive time series analysis. The main conceptual innovation is the introduction of the \emph{multiresolution intrinsic mode function} (MIMF) of the form 
\[
\sum_{n=-N/2}^{N/2-1} a_n\cos(2\pi n\phi(t))s_{cn}(2\pi N\phi(t))+\sum_{n=-N/2}^{N/2-1}b_n \sin(2\pi n\phi(t))s_{sn}(2\pi N\phi(t))\] to model nonlinear and non-stationary data with time-dependent amplitudes, frequencies, and waveforms. 
The multiresolution expansion coefficients $\{a_n\}$, $\{b_n\}$, and the shape function series $\{s_{cn}(t)\}$ and $\{s_{sn}(t)\}$ provide innovative features for adaptive time series analysis. For complex signals that are a superposition of several MIMFs with well-differentiated phase functions $\phi(t)$, a new recursive scheme based on Gauss-Seidel  iteration and diffeomorphisms is proposed to identify these MIMFs, their multiresolution expansion coefficients, and shape function series. Numerical examples from synthetic data and natural phenomena are given to demonstrate the power of this new method.

\end{abstract}



\section{Introduction}
\label{sec:intro}

Extracting useful information from large amounts of oscillatory data is important for a considerate number of real world applications such as medical electrocardiography (ECG) reading \cite{HauBio2,Pinheiro2012175,7042968}, atomic crystal images in physics \cite{Crystal,LuWirthYang:2016}, mechanical engineering \cite{Eng2,ME}, art investigation \cite{Canvas,Canvas2}, geology \cite{GeoReview,SSCT,977903}, imaging \cite{4685952}, etc.  In order to extract certain features and analyze adaptive components of oscillatory data, it is typical to assume that the signal $f(t)$ consists of several oscillatory modes like
\begin{equation}
\label{P1}
f(t)=\sum_{k=1}^K \alpha_k(t) e^{2\pi i N_k \phi_k(t)}+r(t),
\end{equation}
for $t\in[0,1]$, where $\alpha_k(t)$ is the instantaneous amplitude, $N_k \phi_k(t)$ is the instantaneous phase, $N_k\phi_k'(t)$ is the instantaneous frequency, and $r(t)$ is the residual signal. Many methods have been developed to decompose the signal $f(t)$ into several modes $\alpha_k(t) e^{2\pi iN_k \phi_k(t)}$ and to estimate the instantaneous information (the amplitude and phase functions) including the empirical mode decomposition (EMD) approach \cite{Huang1998,doi:10.1142/S1793536909000047}, synchrosqueezed transforms \cite{Daubechies2011,behera}, time-frequency reassignment methods \cite{Auger1995,Chassande-Mottin2003}, adaptive optimization \cite{VMD,Hou2012}, iterative filters \cite{YangWang,Cicone2016384}, etc.

Although modeling oscillation in the form of model \eqref{P1} is effective in many applications, sinusoidal oscillatory patterns may loses some important physical information when the data contain more complicated features. {This observation was rasied among EMD approaches and named as the intra-wave phenomenon (see, e.g., \cite{PhysicalAnal}). This phenomenon was also modeled with ``wave-shape functions" in \cite{Hau-Tieng2013}. Let us use the second terminology with $\{s_k(t)\}_{1\leq k\leq K}$ as the shape functions to build a probably better mathematical model for oscillatory data analysis in a form of a superposition of generalized intrinsic mode functions (GIMFs): }
\begin{equation}
\label{P2}
f(t)= \sum_{k=1}^K \alpha_k(t) s_k(2 \pi N_k \phi_k(t))+r(t)=\sum_{k=1}^K \sum_{n=-\infty}^{\infty} \widehat{s_k}(n)\alpha_k(t) e^{2\pi i n N_k \phi_k(t)}+r(t),
\end{equation}
for $t\in[0,1]$, where $\{s_k(t)\}_{1\leq k\leq K}$ are $2\pi$-periodic shape functions with a unit norm in $L^2$ and $\widehat{s}(0)=0$. {The mode decomposition problem with a model in \eqref{P2}  may have different names in the literature; but it is called the generalized mode decomposition (GMD) in this paper.} The shape function can reflect complicated evolution patterns of the signal $f(t)$. The photoplethysmogram (PPG) signal (see Figure \ref{fig:ppg0}) in medical study is one of such complex examples. The PPG signal contains two essential evolution patterns corresponding to the cardiac and respiratory cycles. The shape of the PPG waveform differs from subject to subject and contains valuable information for monitoring the health condition of patients \cite{doi:10.1097/ALN.0b013e31816c89e1}. For more examples, the reader is referred to a detailed survey in \cite{ceptrum}. {The introduction of shape functions makes it more difficult to solve the decomposition problem and it has been an active research direction to seek its numerical solutions \cite{1DSSWPT,Hou2016,ChuiLinWu2016,HZYregression}.}

\begin{figure}
  \begin{center}
   \includegraphics[width=4in]{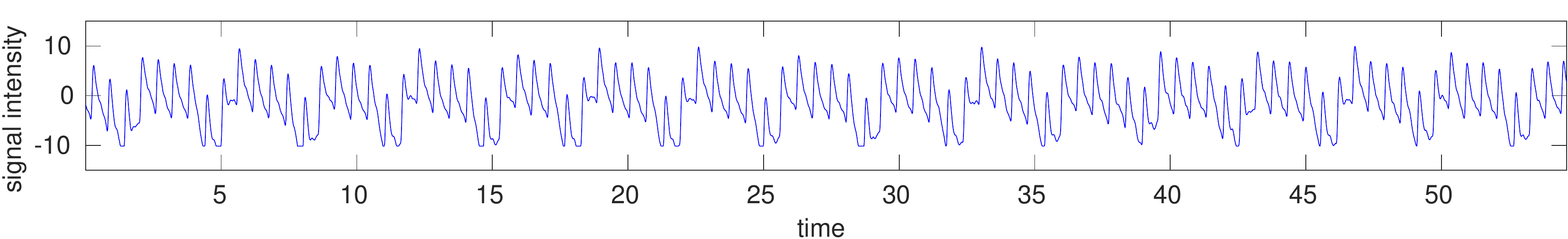} \\
   The raw PPG signal $f(t)$.\\
         \includegraphics[width=4in]{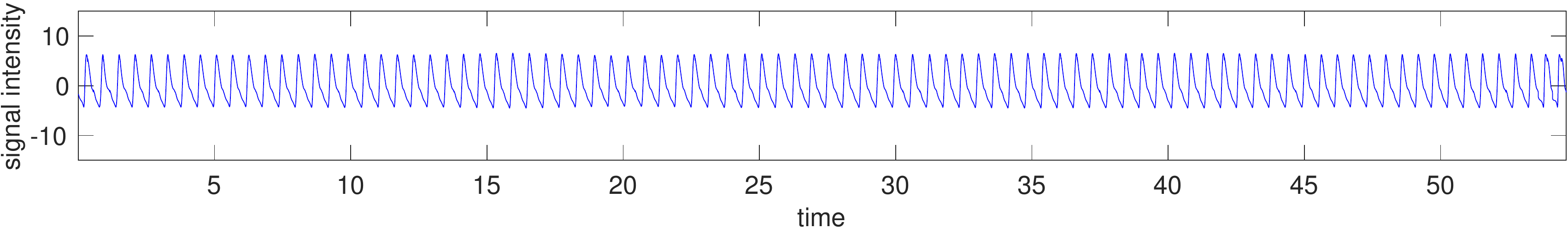}  \\
      The cardiac mode $f_{1}(t)$ by the generalized mode decomposition in \eqref{P2}.\\
      \includegraphics[width=4in]{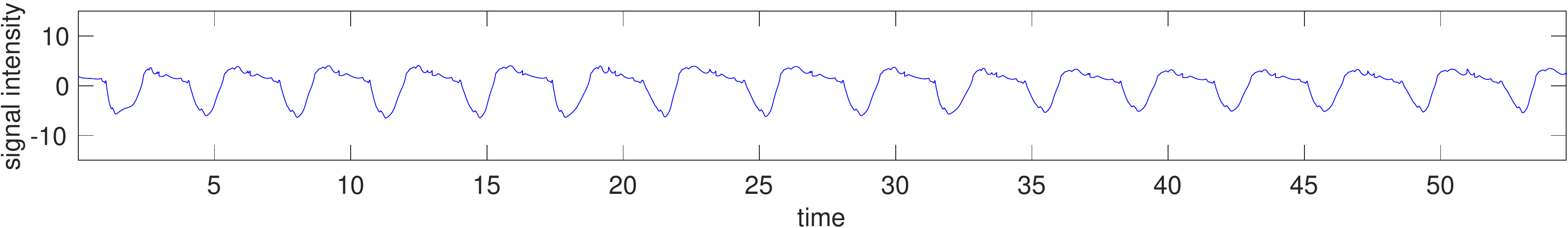}\\
      The respiratory mode $f_{2}(t)$ by the generalized mode decomposition in \eqref{P2}.\\      
      \includegraphics[width=4in]{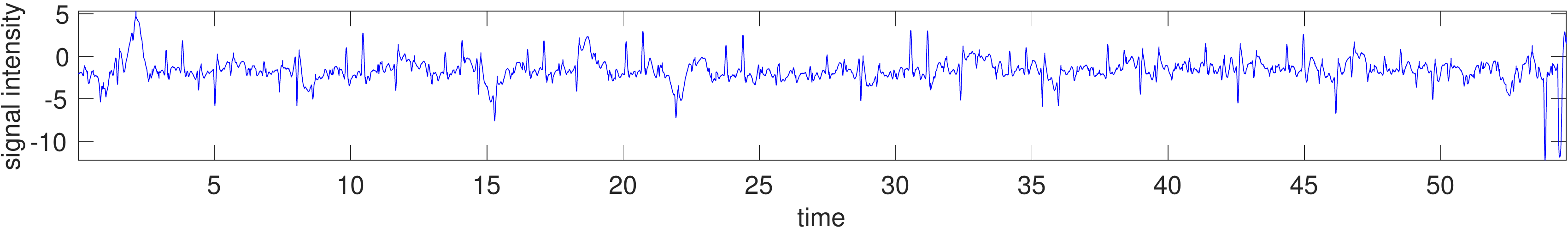}  \\
      The residual error $f(t)-f_{1}(t)-f_{2}(t)$  of the generalized mode decomposition in \eqref{P2}.\\
      \includegraphics[width=4in]{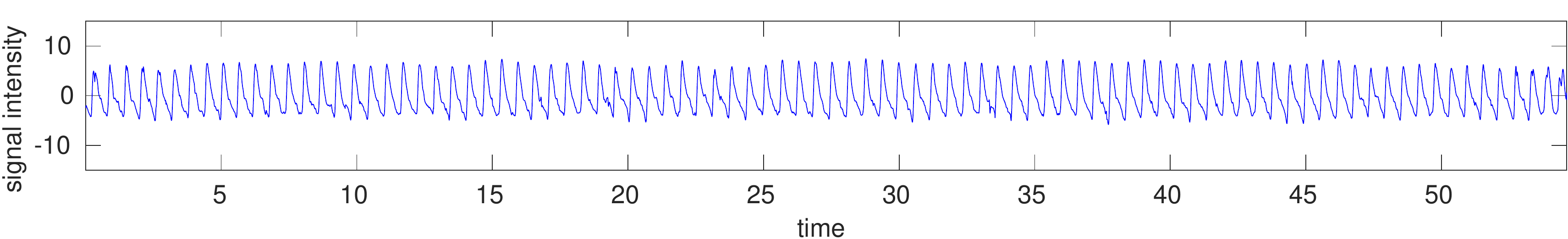}  \\
      The cardiac mode $f_1(t)$ by the multiresolution mode decomposition in \eqref{eqn:m_mmd}.\\
      \includegraphics[width=4in]{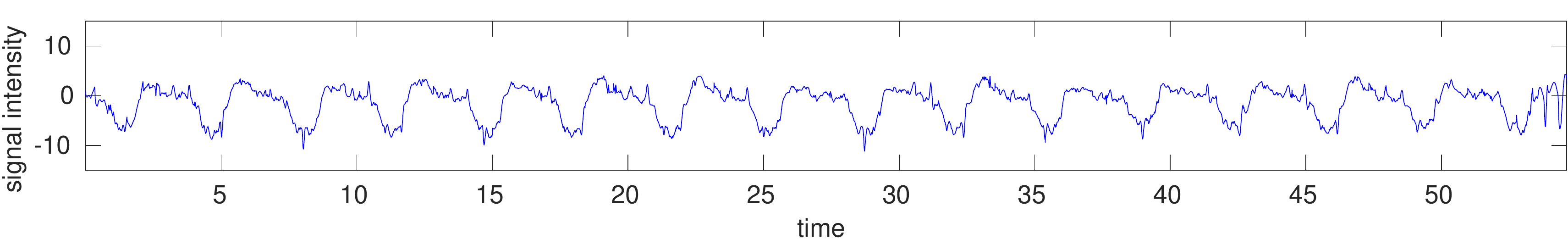}\\
      The respiratory mode $f_1(t)$ by the multiresolution mode decomposition in \eqref{eqn:m_mmd}.\\ 
      \includegraphics[width=4in]{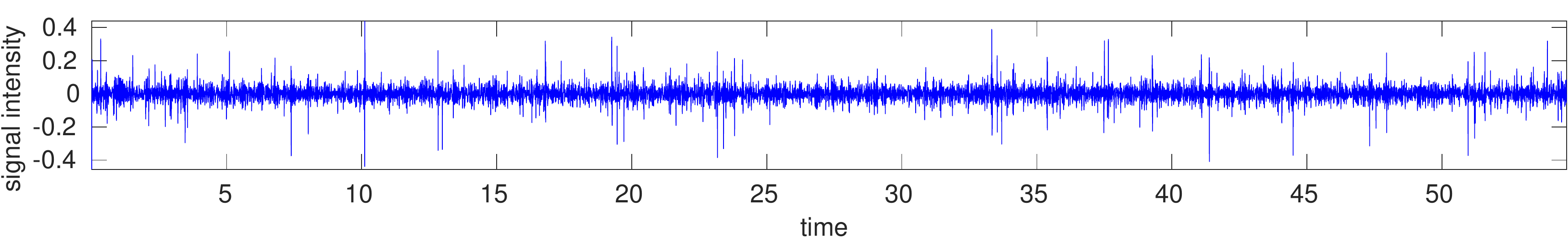}  \\
      The residual error $f(t)-f_{1}(t)-f_{2}(t)$ of the multiresolution mode decomposition in \eqref{eqn:m_mmd}.\\ 
  \end{center}
    \caption{Comparison of the GMD \eqref{P2} and MMD \eqref{eqn:m_mmd} for a photoplethysmogram (PPG) signal. The residual data  of the GMD model still contain obvious oscillatory patterns with significant signal intensity, while the residual data of the MMD model in \eqref{eqn:m_mmd} is much weaker and close to i.i.d random noise {(see Figure \ref{fig:ppg01} below for a quantitative analysis). }}
  \label{fig:ppg0}
\end{figure}

\begin{figure}
  \vspace{-0.5cm}
  \begin{center}
\begin{tabular}{ccc}
   \includegraphics[width=1in]{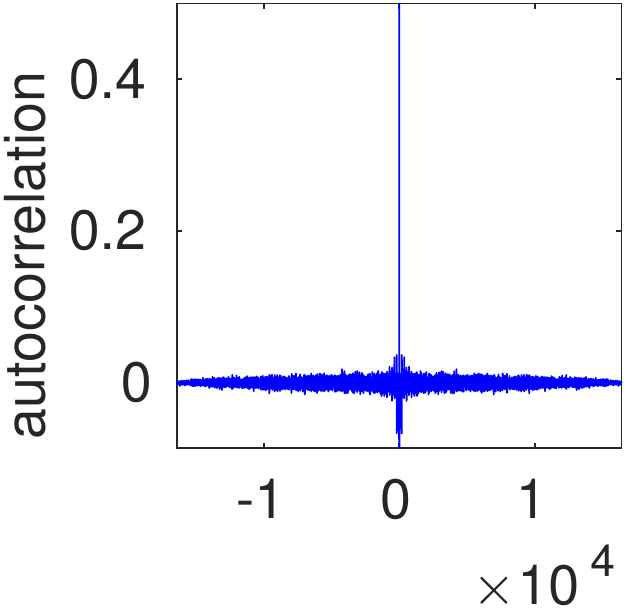} & \includegraphics[width=1in]{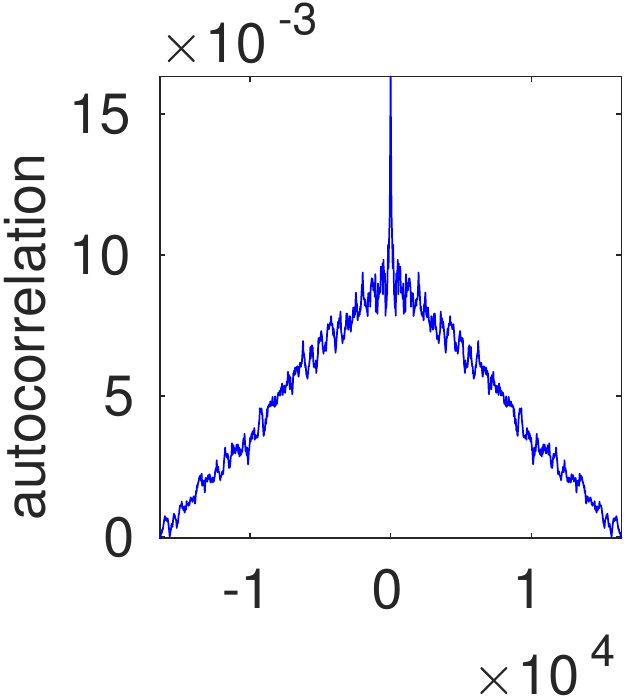}&  \includegraphics[width=1in]{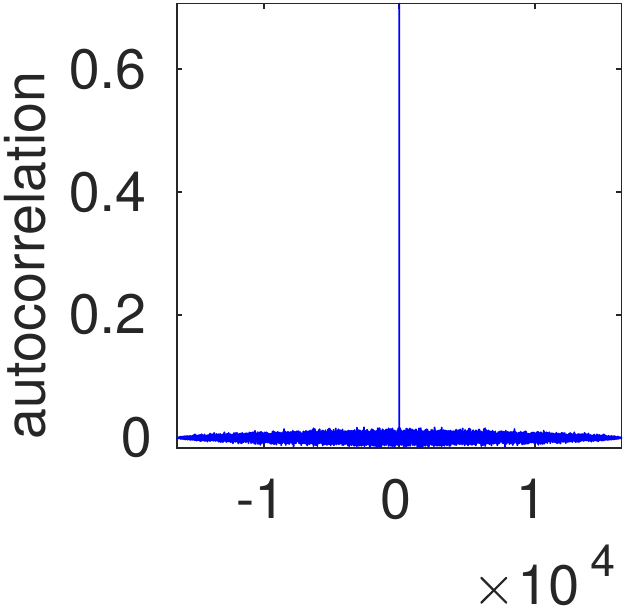} 
   \end{tabular}
  \end{center}
  \vspace{-0.5cm}
    \caption{{Comparison of the whiteness of the residual signal generated by GMD \eqref{P2} and MMD \eqref{eqn:m_mmd} for the PPG signal in Figure \ref{fig:ppg0}. The autocorrelation of the residual signal by GMD, the residual signal by MMD, and a vector of Gaussian random noise is plotted in the left, middle, and right figures, respectively. Theoretically, the autocorrelation of white noise is an impulse at lag 0. Hence, the results here show that the residual signal by MMD is close to white noise, while the one by GMD still contains correlated oscillation}.}
  \vspace{-0.5cm}
  \label{fig:ppg01}
\end{figure}

\begin{figure}
  \vspace{-0.5cm}
  \begin{center}
\begin{tabular}{cc}
   \includegraphics[width=2.2in]{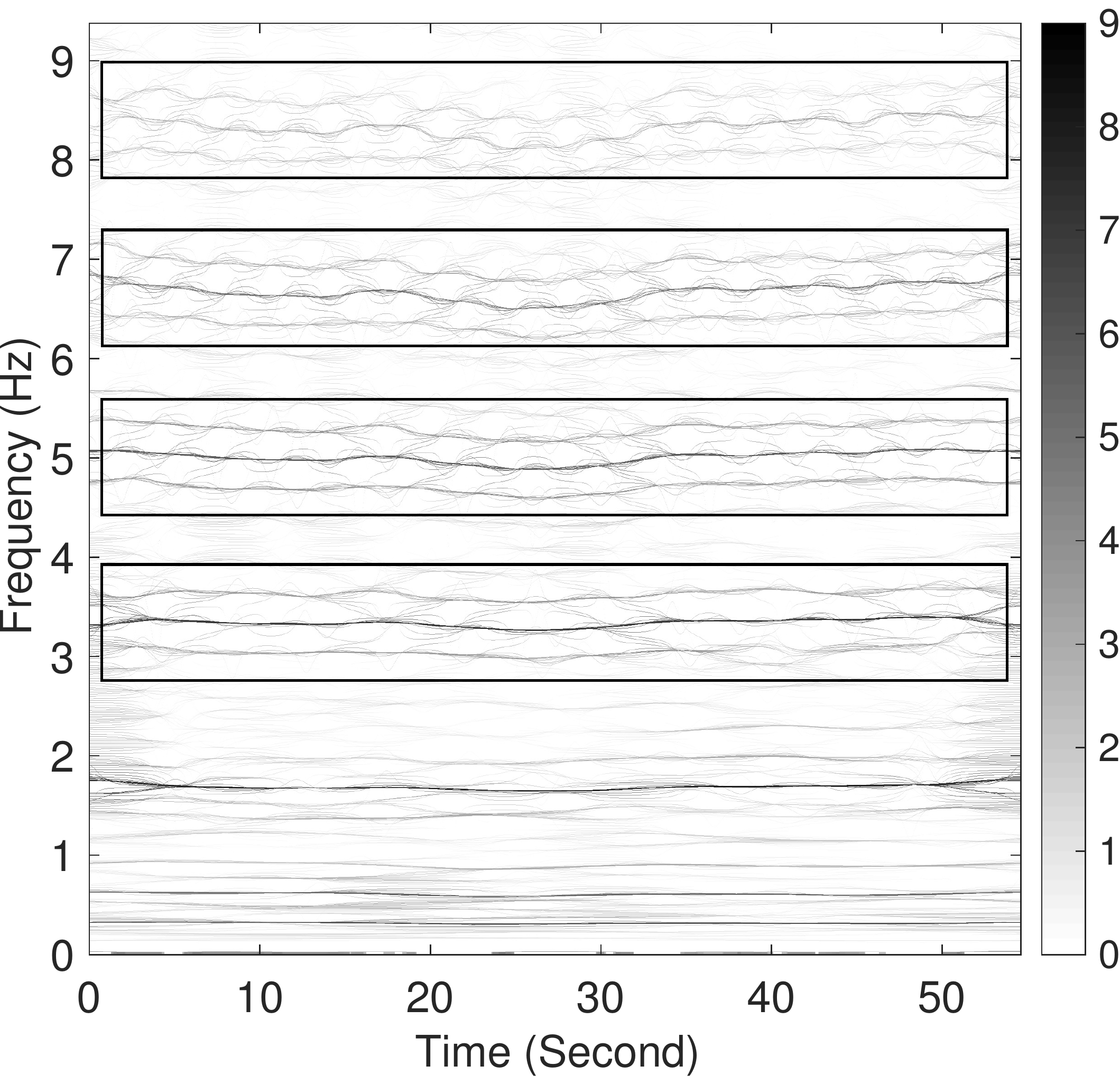} & \includegraphics[width=2.2in]{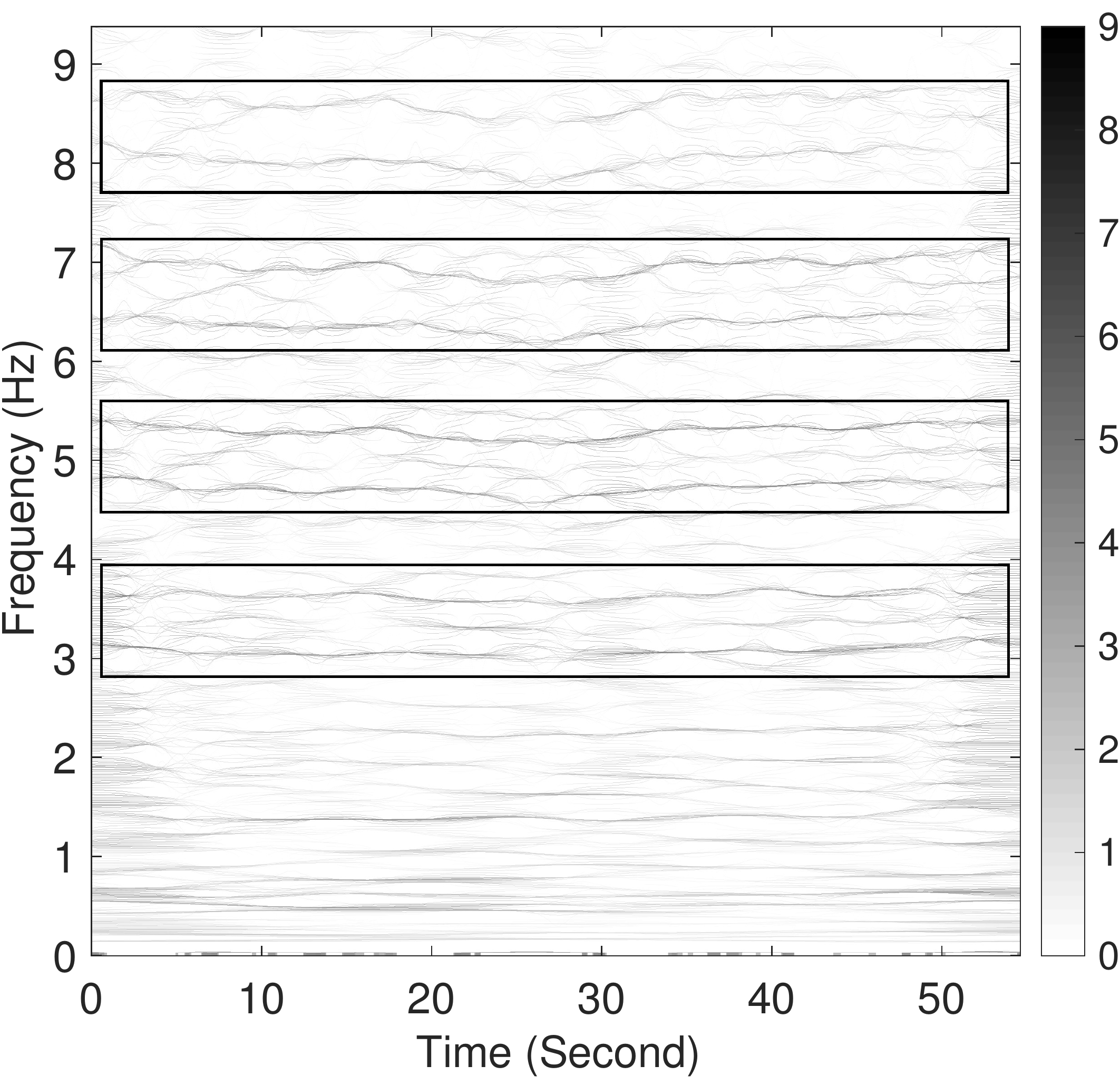}\\ \includegraphics[width=2.2in]{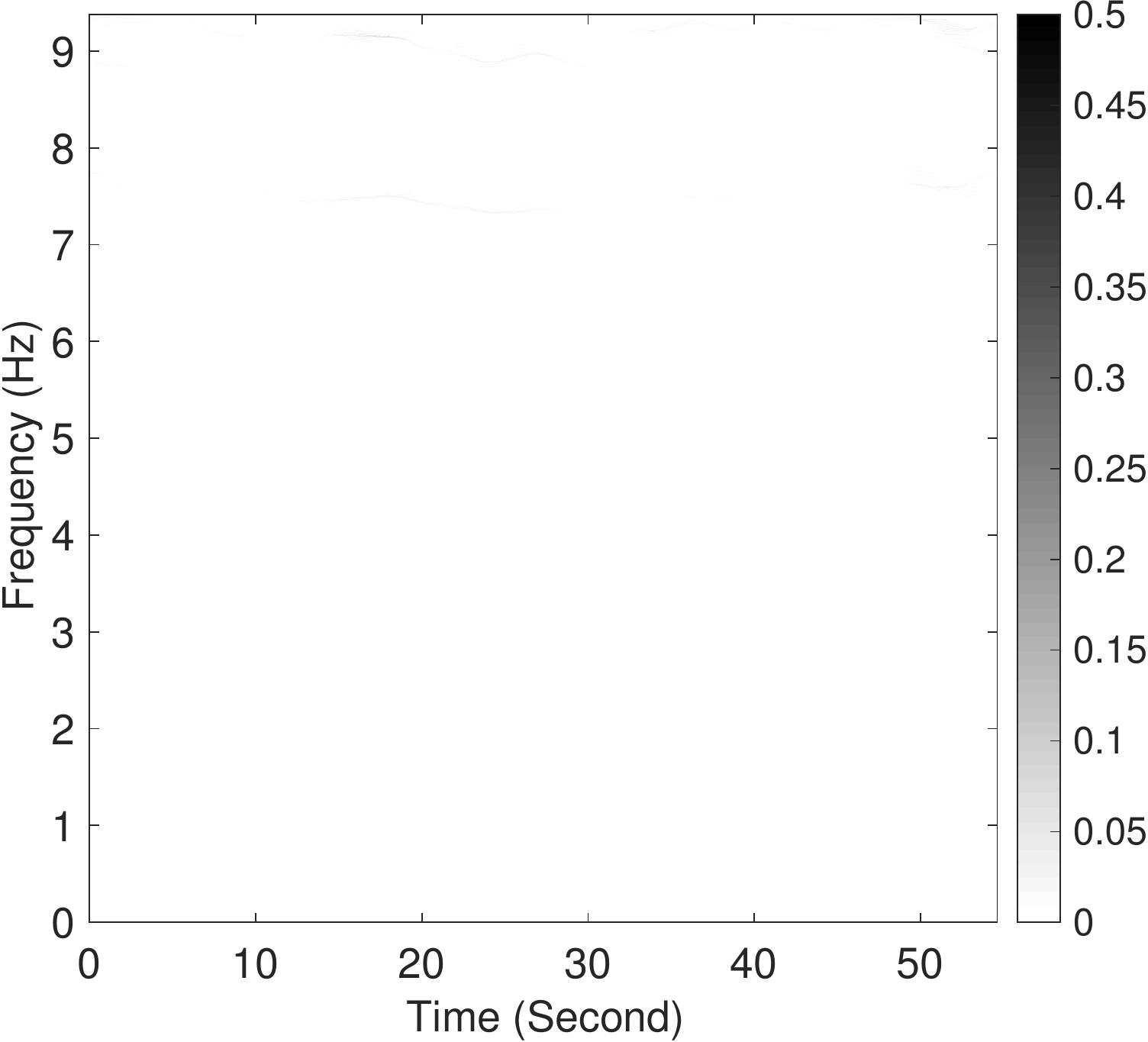} &  \includegraphics[width=2.2in]{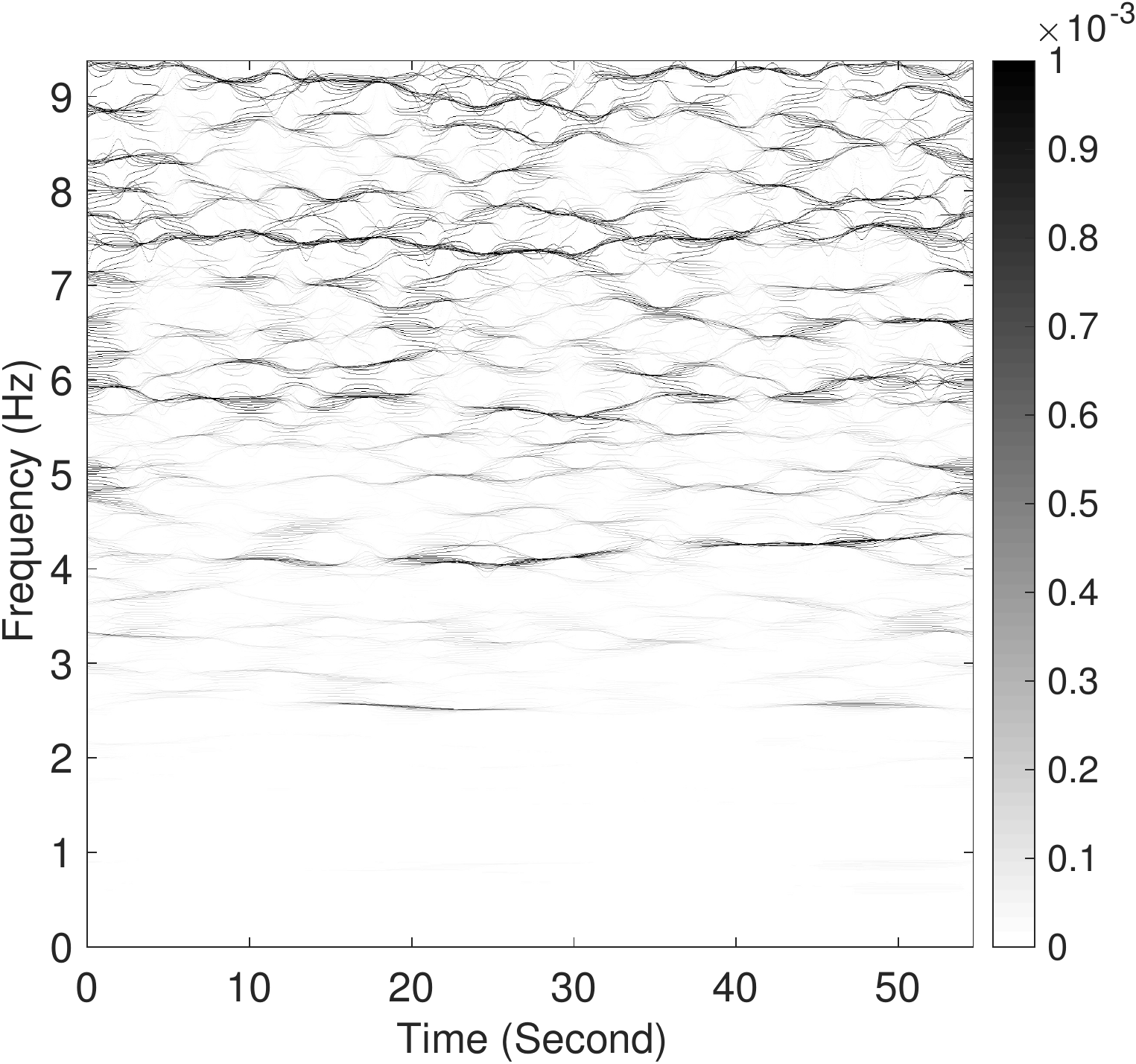} 
   \end{tabular}
  \end{center}
  \vspace{-0.5cm}
    \caption{{Top-left: the logarithm of the synchrosqueezed transform (SST) of the original PPG signal in Figure \ref{fig:ppg0}. Top-right: the logarithm of the SST of the residual signal by GMD \eqref{P2}, i.e., the $4$-th signal in Figure \ref{fig:ppg0}. There are three major instantaneous frequencies in each rectangle on the left, while there are two in the rectangles on the right. Bottom figures: the logarithm of the SST of the residual signal by MMD \eqref{eqn:m_mmd} with different visualization scales, i.e., the last signal in Figure \ref{fig:ppg0}. }}
  \vspace{-0.5cm}
  \label{fig:ppg02}
\end{figure}

In spite of considerable successes of analyzing oscillatory time series in the form of mode decomposition in \eqref{P1} or GMD in \eqref{P2}, these models conflict with the physical intuition that the oscillation pattern of the time series changes in time. For example, the cardiac and respiratory patterns in Figure \ref{fig:ppg0} vary in time. The GMD of the form \eqref{P2} can only extract average evolution patterns (i.e., time-independent shape functions) to describe the cardiac and respiratory time series, leaving the evolution variance of these patterns (i.e., the deviation from the average evolution pattern) in the residual signal $r(t)$ {(see Panel 4 of Figure \ref{fig:ppg0})}. However, the evolution variance is more important than the average evolution patterns for detecting diseases and measuring health risk. For example, the electrocardiogram (ECG) is an important tool to examine the functional status of the heart. The ECG waveform consists of three characteristic events (the P, QRS and T-wave as shown in Figure \ref{fig:ECG0}) associated with each beat. The detection of abnormal ECG waveforms is important to cardiac disease diagnosis \cite{lippincott2011ecg,hyper} and the abnormality is the deviation of an observed ECG waveform to a standard ECG waveform (e.g. tall R peaks caused by possible thickening of heart muscle wall,  wide QRS and wide S waves due to partial or complete right bundle branch block).

\begin{figure}
  \begin{center}
   \includegraphics[width=3in]{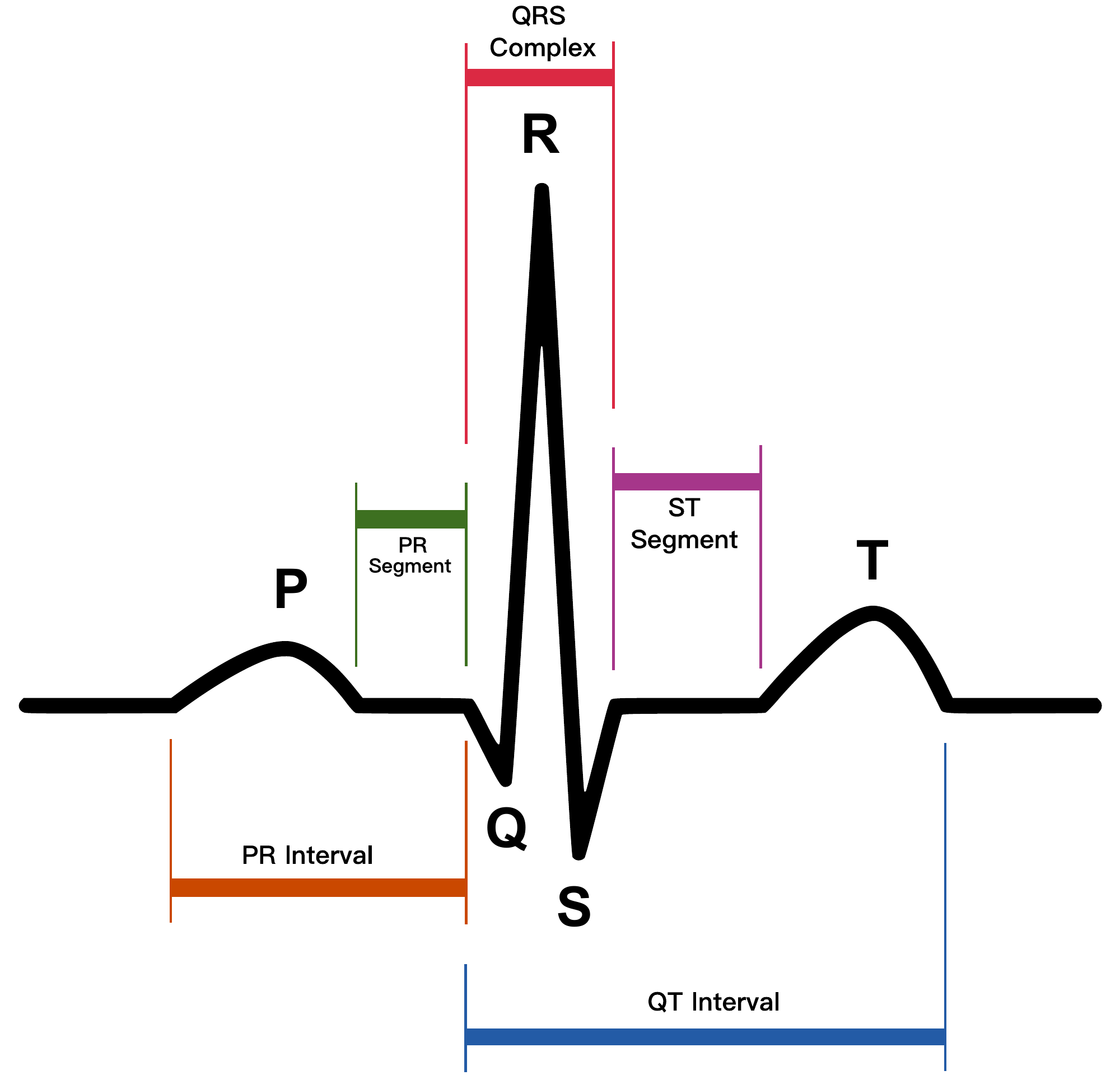}
  \end{center}
    \caption{A schematic diagram of normal sinus rhythm for a human heart as seen on ECG \cite{wikiECG}.}
  \label{fig:ECG0}
\end{figure}

\begin{figure}
  \begin{center}
   \includegraphics[width=6in]{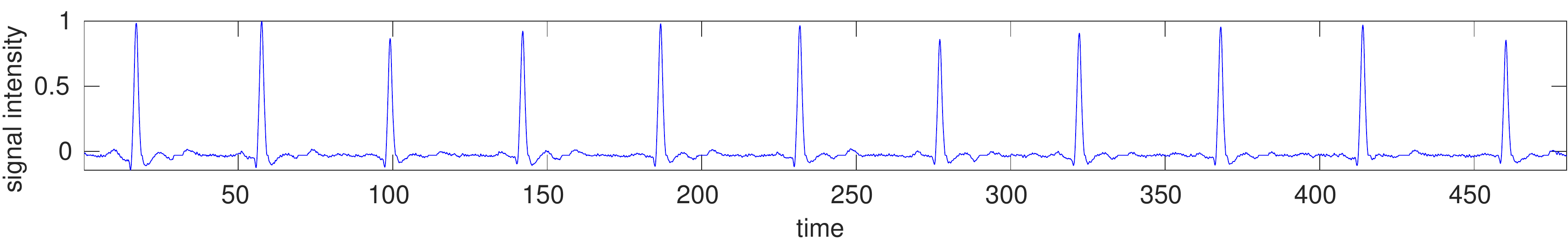} \\
      \includegraphics[width=6in]{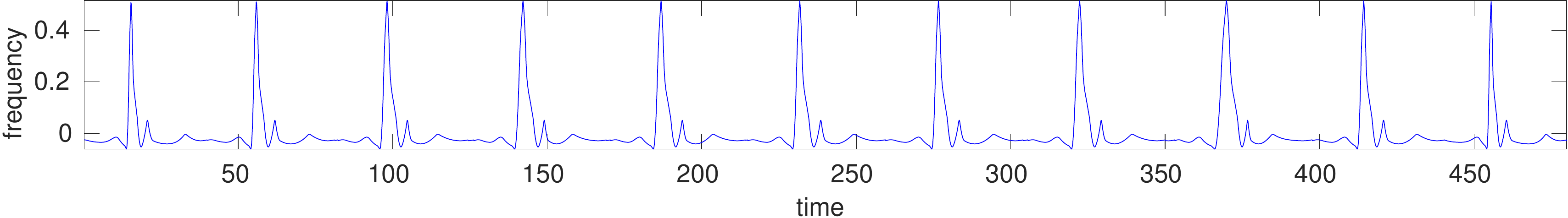}\\
      \includegraphics[width=6in]{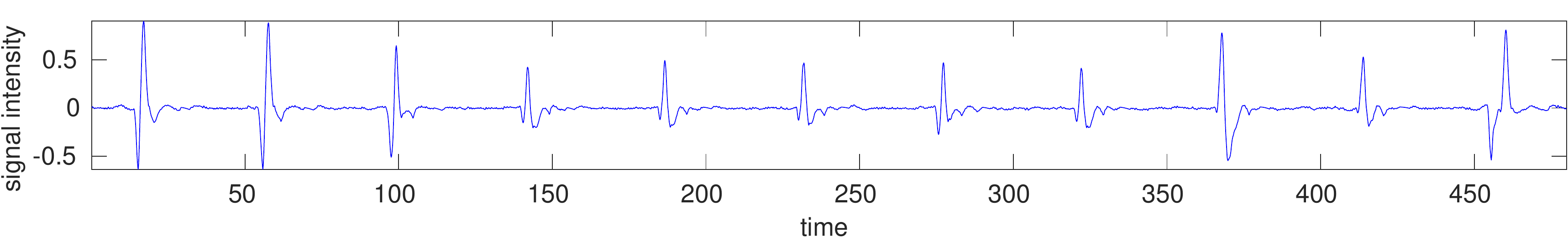}
  \end{center}
    \caption{Top: a motion artifact contaminated electrocardiogram (ECG) signal $f(t)$ modeled by Equation \eqref{eqn:m_IMF}. Middle: the $0$-band multiresolution approximation $\mathcal{M}_0(f)(t)=a_0 s_{c0}(2\pi N\phi(t))$ of $f(t)$. Bottom: $f(t) -\mathcal{M}_0(f)(t)$, the variance of the evolution pattern of $f(t)$.}
  \label{fig:ECG1}
\end{figure}

To analyze the fine features of evolution patterns discussed above, this paper proposes the \emph{multiresolution mode decomposition} as a novel model for adaptive time series analysis. The main conceptual innovation is the introduction of the \emph{multiresolution intrinsic mode function} (MIMF) of the form\footnote{The analytic analog is $f(t)=\sum_{n=-N/2}^{N/2-1}a_n e^{2\pi i n\phi(t)}s_n(2\pi N\phi(t))$, where $\{a_n\}$ are complex numbers and $\{s_n\}$ are analytic shape functions. Without loss of generality, $N$ is assumed to be even throughout this paper.}
\begin{eqnarray}
\label{eqn:m_IMF}
f(t) = \sum_{n=-N/2}^{N/2-1} a_n\cos(2\pi n\phi(t))s_{cn}(2\pi N\phi(t))+\sum_{n=-N/2}^{N/2-1}b_n \sin(2\pi n\phi(t))s_{sn}(2\pi N\phi(t))
\end{eqnarray}
for $t\in[0,1]$, where $\{a_n\}$ and $\{b_n\}$ are real numbers, $\{s_{cn}\}$ and $\{s_{sn}\}$ are real value functions, to model nonlinear and non-stationary data with time-dependent amplitudes, frequencies, and waveforms. The MIMF is a generalization of the model of GIMF,  $\alpha(t)s(2\pi N\phi(t))$, in Equation \eqref{P2} for more accurate data analysis. When $s_{cn}(t)$ and $s_{sn}(t)$ in Equation \eqref{eqn:m_IMF} are equal to the same shape function $s(t)$, the model in Equation \eqref{eqn:m_IMF} is reduced to $\alpha(t)s(2\pi N\phi(t))$ once the amplitude function $\alpha(t)$ is written in the form of its Fourier series expansion. When $s_{cn}(t)$ and $s_{sn}(t)$ are different shape functions, the two summations in Equation \eqref{eqn:m_IMF} lead to time-dependent shape functions to describe the nonlinear and non-stationary time series adaptively. {A recent paper \cite{ceptrum} also tried to address the limitation of GMD in \eqref{P2} by replacing $\widehat{s_k}(n)\alpha_k(t)$ with a time-varying function, denoted as $B_{k,n}(t)$, i.e., introducing more variance to amplitude functions. 
Our model in \eqref{eqn:m_IMF} emphasizes both the time variance of amplitude and shape functions by introducing multiresolution expansion coefficients and shape function series. As far as we understand, instead of estimating time-varying amplitude and shape functions, \cite{ceptrum} proposed an algorithm to eliminate the influence of amplitude and shape functions and estimate phase functions $N_k\phi_k(t)$. This algorithm could be a useful tool complimentary to the algorithm proposed in this paper, since we assume phase functions are known and estimate  time-varying amplitudes and shapes.}

Let $\mathcal{M}_\ell$ be the operator for computing the \emph{$\ell$-banded multiresolution approximation} to a MIMF $f(t)$ in Equation \eqref{eqn:m_IMF}, i.e.,
\begin{eqnarray}
\label{eqn:mm_IMF}
\mathcal{M}_\ell(f)(t) = \sum_{n=-\ell}^{\ell} a_n\cos(2\pi n\phi(t))s_{cn}(2\pi N\phi(t))+\sum_{n=-\ell}^{\ell}b_n \sin(2\pi n\phi(t))s_{sn}(2\pi N\phi(t)),
\end{eqnarray}  
and $\mathcal{R}_\ell$ be the operator for the computing the residual sum
\begin{equation}
\label{eqn:rr_IMF}
\mathcal{R}_\ell(f)(t) = f(t)-\mathcal{M}_\ell(f)(t).
\end{equation}
 Then the $0$-banded multiresolution approximation $\mathcal{M}_0(f)(t)=a_0 s_{c0}(2\pi N\phi(t))$ describes the average evolution pattern of the signal, while the rest describe the evolution variance. As shown in Figure \ref{fig:ECG1}, if $f(t)$ is an ECG signal\footnote{From the PhysiNet \url{https://physionet.org/}. }, then it is more obvious to observe the change of the evolution pattern from $\mathcal{R}_0(f)(t)$ than from $f(t)$, e.g., the change of the height of R peaks and the width of QRS and S waves.
 
In many applications \cite{HauBio2,Pinheiro2012175,7042968,Crystal,LuWirthYang:2016,Eng2,ME,Canvas,Canvas2,GeoReview,SSCT,977903,4685952}, a signal would be a superposition of several MIMFs, for example, a complex signal 
\begin{equation}\label{eqn:m_mmd}
f(t) = \sum_{k=1}^{K} f_k(t),
\end{equation}
where 
\begin{eqnarray*}
f_k(t)=\sum_{n=-N/2}^{N/2-1} a_{n,k}\cos(2\pi n\phi_k(t))s_{cn,k}(2\pi N_k\phi_k(t))+\sum_{n=-N/2}^{N/2-1}b_{n,k} \sin(2\pi n\phi_k(t))s_{sn,k}(2\pi N_k\phi_k(t)).
\end{eqnarray*}
The multiresolution mode decomposition (MMD) problem aims at extracting each MIMF $f_k(t)$, estimating its corresponding multiresolution expansion coefficients $\{a_{n,k}\}$, $\{b_{n,k}\}$, and the shape function series $\{s_{cn,k}(t)\}$ and $\{s_{sn,k}(t)\}$. As we can see in Figure \ref{fig:ppg0}, if MMD is applied to analyze the PPG signal, we can obtain the cardiac and respiratory patterns (Panel $5$ and $6$ in Figure \ref{fig:ppg0}, respectively) with more accurate evolution variance than the model of GMD in \eqref{P2} {in two aspects. In terms of statistical testing, the residual signal of MMD in Panel $7$ of Figure \ref{fig:ppg0} contains information close to i.i.d random noise while the residual by \eqref{P2} still contains correlated oscillation patterns as demonstrated by Figure \ref{fig:ppg01} quantitatively.  From the point of view of time-frequency analysis visualized in Figure \ref{fig:ppg02}, the synchrosqueezed transform (SST) of the residual by \eqref{P2} in the top-right panel shares almost the same spectrogram with the SST of the original PPG signal in the top-left panel. Note that the model in \eqref{P2} seems to capture only partial cardiac and respiratory patterns in the PPG signal and the residual signal still contains significant information resembling the cardiac and respiratory patterns. In fact, there are three major instantaneous frequencies with almost the same geometry in each rectangle on the top-left of Figure \ref{fig:ppg02}, while there are two in the rectangles on the top-right, indicating that the model in \eqref{P2} misses some instantaneous frequencies with similar oscillation patterns as those considered in \eqref{P2}. As a comparison, the SST of the residual by MMD has no obvious spectrogram even if in a much smaller visualization scale $(0,0.5)$ as shown in the bottom-left panel of Figure \ref{fig:ppg02}. In a very small visualization scale like $(0,0.001)$ in the bottom-right panel, we see that the SST of the residual by MMD indicates no meaningful oscillation pattern. The oscillation patterns missed by \eqref{P2} and visualized on the top-right panel of Figure \ref{fig:ppg02}  have been considered in the MMD model using the summation over different $n$'s in \eqref{eqn:m_IMF}. The geometry of the instantaneous frequencies of different terms with different $n$'s is similar to that of the term when $n=0$.}

 Although there have been a few algorithms for the GMD in \eqref{P2} \cite{1DSSWPT,ChuiLinWu2016,HZYregression}, these methods are incapable of either identifying the shape function series $\{s_{cn,k}(t)\}$ and $\{s_{sn,k}(t)\}$, or the multiresolution expansion coefficients $\{a_{n,k}\}$ and $\{b_{n,k}\}$ in the multiresolution mode decomposition. Motivated by the recursive diffeomorphism-based regression (RDBR) in \cite{HZYregression}, this paper proposes a Gauss-Seidel style recursive scheme to solve the multiresolution mode decomposition problem. As we shall see later, the novel recursive scheme has a faster convergence rate than the existing scheme in \cite{HZYregression}, and more importantly, it is robust to the estimated number of components (i.e., it still returns a meaningful mode decomposition even if the input number of modes is wrong).

To make the presentation of the Gauss-Seidel RDBR for the multiresolution mode decomposition more accessible, we will first introduce the Gauss-Seidel  RDBR for the GMD in Section \ref{sec:GMD}. The Gauss-Seidel RDBR for the multiresolution mode decomposition will be introduced in Section \ref{sec:MMD}. In Section \ref{sec:NumEx}, we present some numerical examples to demonstrate the efficiency of the proposed RDBR. Finally, we conclude this paper in Section \ref{sec:con}.
        
\section{Gauss-Seidel recursive scheme for the GMD}
\label{sec:GMD}

In what follows, we introduce the new Gauss-Seidel  recursive diffeomorphism-based regression (RDBR) for the GMD. Existing methods \cite{1DSSWPT,ChuiLinWu2016,HZYregression} for the GMD problem 
\begin{equation}
\label{P3}
f(t)= \sum_{k=1}^K \alpha_k(t) s_k(2 \pi N_k \phi_k(t))=\sum_{k=1}^K \sum_{n=-\infty}^{\infty} \widehat{s_k}(n)\alpha_k(t) e^{2\pi i n N_k \phi_k(t)}
\end{equation}
generally assume that the instantaneous properties (such as $\alpha_k(t)$ and $N_k \phi_k(t)$ in Equation \eqref{P3}) are available and focus on the estimation of shape functions $s_k(t)$. This assumption is based on the observation that: after filtering the signal with a low-pass filter, the GMD problem in Equation \eqref{P3} becomes the standard mode decomposition problem in \eqref{P1}; afterwards, instantaneous properties can be estimated by well-developed algorithms for the mode decomposition problem \cite{Huang1998,doi:10.1142/S1793536909000047,Daubechies2011,Auger1995,Chassande-Mottin2003,VMD,Hou2012,YangWang,Cicone2016384}, {especially the deShape SST \cite{ceptrum,Antonio}}. Hence, we assume that the instantaneous amplitudes and phases are known in this section. 

\subsection{Algorithm description}
\label{sub:aglo}
As in the Jacobi style RDBR algorithm \cite{HZYregression}, we assume $L$ points of measurement $\{f(t_\ell)\}_{\ell=1,\dots,L}$ with independent and identically distributed (i.i.d.) grid points $\{t_\ell\}_{\ell=1,\dots,L}$ from a uniform distribution in $[0,1]$. Usually, the grid is deterministic and uniform in $[0,1]$, but an i.i.d. grid enables a better estimation of the shape function with a smaller $L$, since it allows the access of the $2$-$\pi$ periodic shape function $s(t)$ for every $t\in[0,2\pi]$ with a certain probability.

\begin{figure}
  \begin{center}
    \begin{tabular}{cc}
      \includegraphics[height=0.5in]{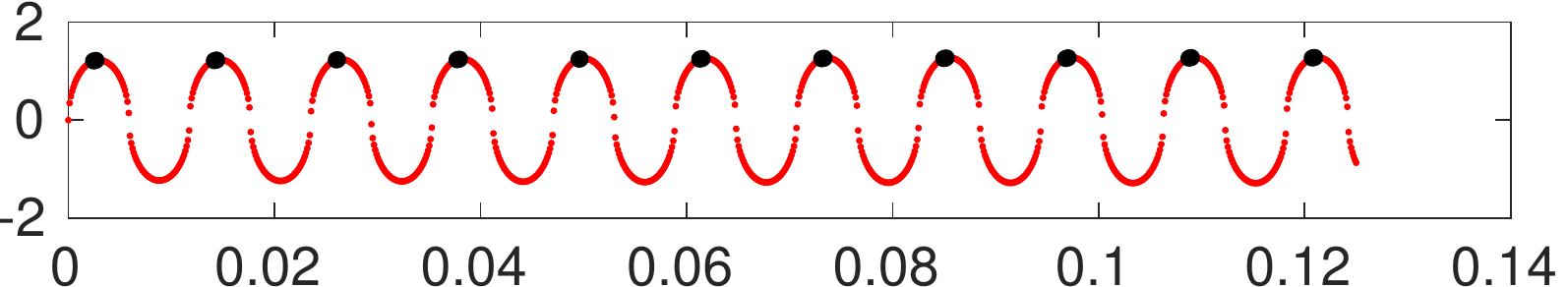} & 
      \includegraphics[height=0.5in]{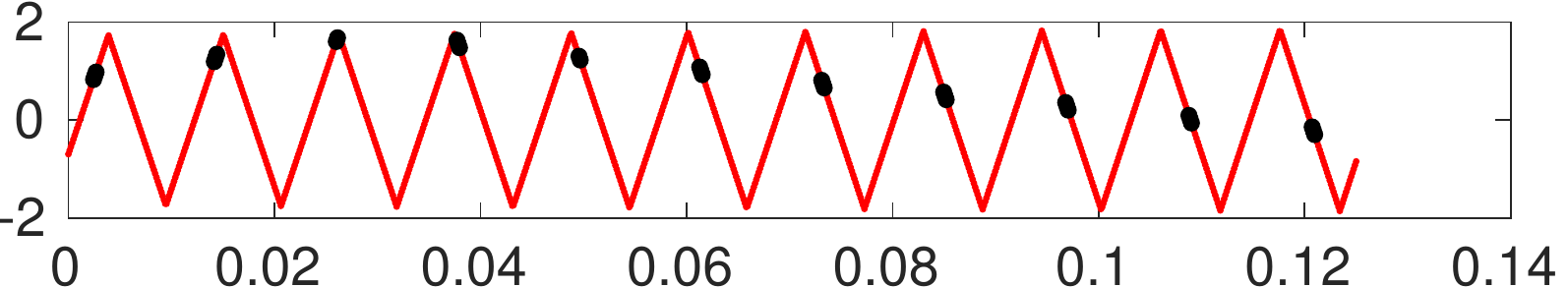}\\
      $s_1(2\pi v)$ &
    $\kappa_1(2\pi v)= s_2(2\pi   p_2\circ p_1^{-1}(v))$\\
    \end{tabular}
  \end{center}
  \caption{An illustration of the inverse-warping idea in Equation \eqref{eqn:ks} when $f(t)= s_1(2\pi N_1\phi_1(t)) + s_2(2\pi N_2\phi_2(t))$. Let $p_j=N_j\phi_j(t)$ for all $j$ and inverse-warp $f(t)$ with $v=p_1(t)$ into $h_1(v)$ as in Equation \eqref{eqn:ks}, then  $h_1(v)$ contains two main parts: a periodic function, $s_1(2\pi v)$ with samples $\{s_1(2\pi v_\ell)\}_{\ell=1,\dots,L}$ (see the left figure), and a non-periodic function, $\kappa_1(2\pi v)$ with samples $\{\kappa_1(2\pi v_\ell)\}_{\ell=1,\dots,L}$ (see the right figure).}
\label{fig:ks}
\end{figure}

 {In the RDBR, we define the inverse-warping data by
\begin{eqnarray}
h_k(v)& =& \frac{f\circ p_k^{-1}(v)}{\alpha_k\circ p_k^{-1}(v)}\nonumber\\
&=&s_k(2\pi v)+ \sum_{j\neq k} \frac{\alpha_j\circ p_k^{-1}(v)}{\alpha_k\circ p_k^{-1}(v)} s_j(2\pi   p_j\circ p_k^{-1}(v))\nonumber\\
&\defeq & s_k(2\pi v)+ \kappa_k(2\pi v),
\label{eqn:ks}
\end{eqnarray}
where $v \defeq p_k(t)$, $p_j(t)=N_j \phi_j(t)$ for all $j$, and 
\begin{eqnarray*}
\kappa_k(2\pi v) \defeq \sum_{j\neq k} \frac{\alpha_j\circ p_k^{-1}(v)}{\alpha_k\circ p_k^{-1}(v)} s_j(2\pi   p_j\circ p_k^{-1}(v)).
\end{eqnarray*}
As a consequence, we have a set of measurements of $h_k(v)$, $\{h_k(v_\ell)\}_{\ell=1,\dots,L}$, sampled in $v$ with $v_\ell=p_k(t_\ell)$ (see Figure \ref{fig:ks} for an illustration). }

If there was a single mode in $f(t)$ (e.g. $f(t)=\alpha_k(t)s_k(2\pi N_k \phi_k(t))$ with $k=1$), then $h_k(v)$ is equal to a periodic function $s_k(2\pi v)$ with period $1$. Hence, if we define a folding map $\tau$ that folds the two-dimensional point set $\{(v_\ell,h_k(v_\ell))\}_{\ell=1,\dots,L}$  together
\begin{eqnarray}\label{eqn:fold}
\tau: \ \      \left(v_\ell, h_k(v_\ell) \right)   \mapsto    \left(\text{mod}(v_\ell,1),  h_k(v_\ell) \right),
\end{eqnarray}
then the point set $\{\tau(v_\ell,s_k(2\pi v_\ell))\}_{\ell=1,\dots,L}\subset \mathbb{R}^2$ is a two-dimensional point set located at the curve $(v,s_k(2\pi v))\subset \mathbb{R}^2$ given by the shape function $s_k(2\pi v)$ with $v\in [0,1)$ (see Figure \ref{fig:s} (left) for an example). Using the notations in non-parametric regression, let $X_k$ be an independent random variable in $[0,1)$, $Y_k$ be the response random variable in $\mathbb{R}$, and consider $(x_\ell^{(k)},y_\ell^{(k)})=\tau(v_\ell,s_k(2\pi v_\ell))$ as $L$ i.i.d. samples of the random vector $(X_k,Y_k)$, then a simple regression results in the exact shape function as follows. Define
\begin{equation}
\label{eqn:re}
s^R_k\defeq \underset{s:\mathbb{R}\rightarrow \mathbb{R}}{\arg\min}\quad \E\{\left| s(2\pi X_k)-Y_k\right|^2\},
\end{equation}
where the superscript $^R$ means the ground truth regression function, then $s_k=s^R_k$. If we denote the numerical solution of the above regression problem as $s^P_k$, then $s^P_k\approx s^R_k=s_k$ when $L$ is sufficiently large.

\begin{figure}
  \begin{center}
    \begin{tabular}{cc}
      \includegraphics[height=1.8in]{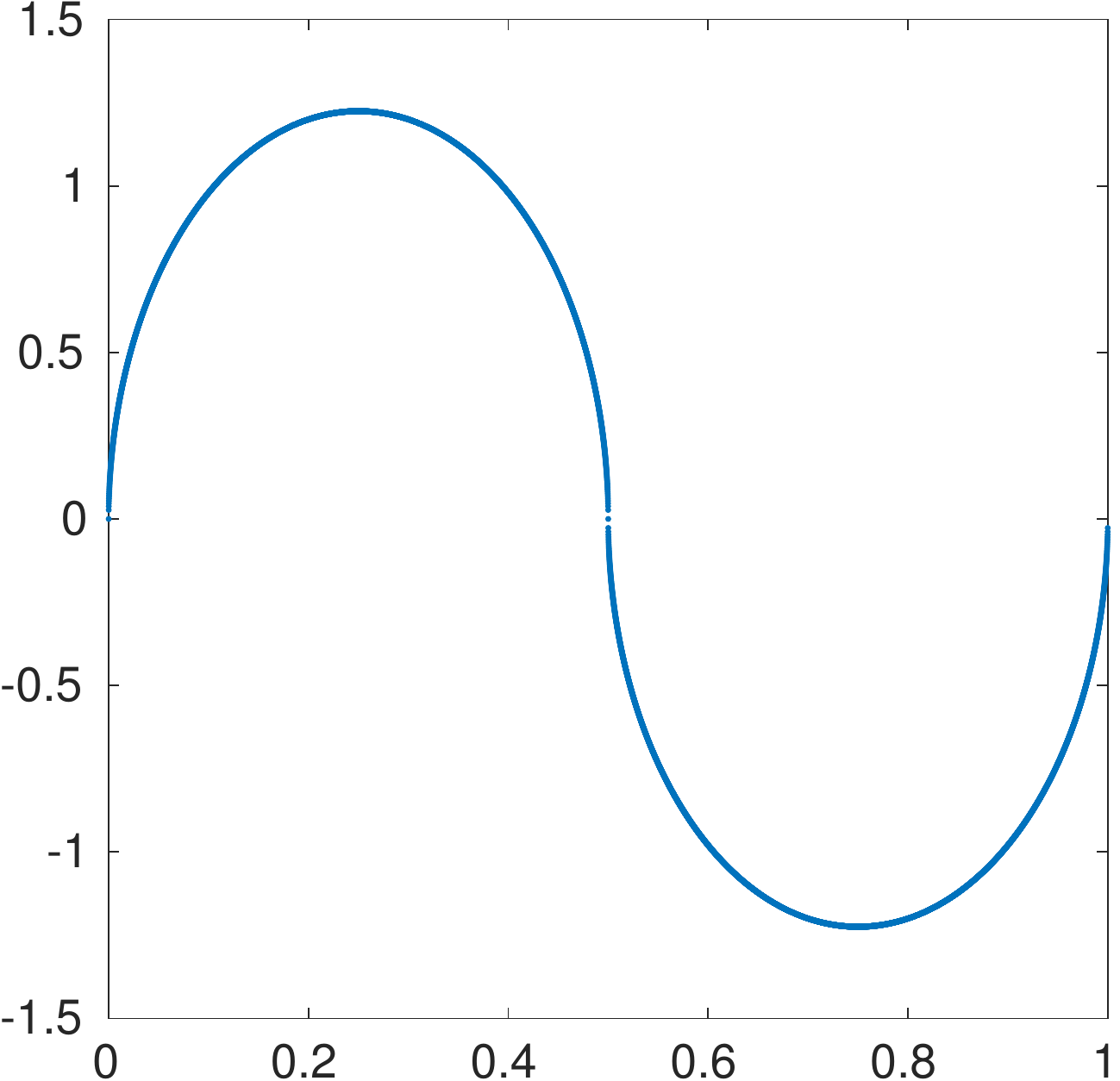} & 
      \includegraphics[height=1.8in]{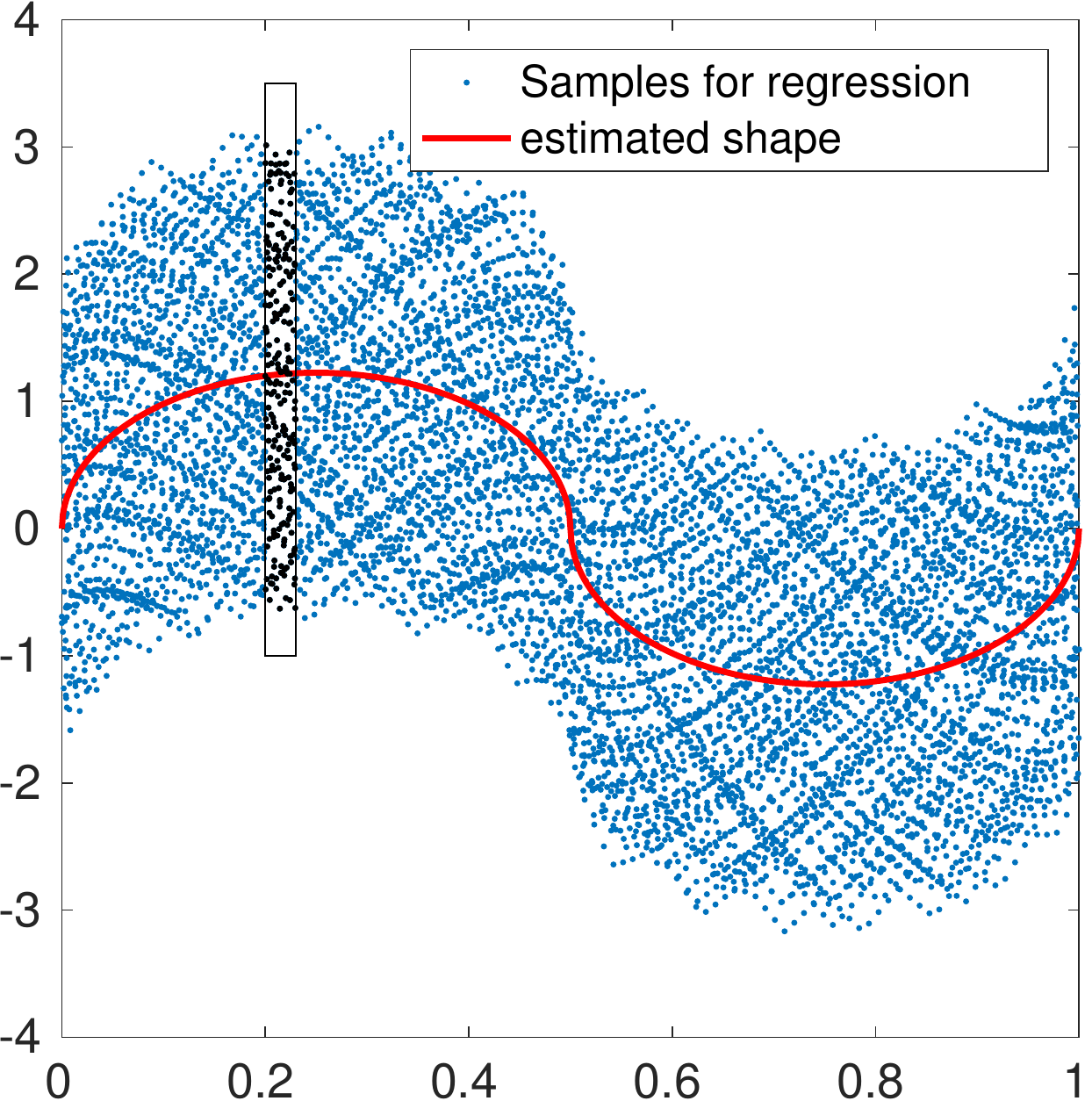}\\
    \end{tabular}
  \end{center}
  \caption{An illustration of the folding idea in Equation \eqref{eqn:fold}. Left: the point set $\{\tau(v_\ell,h_1(v_\ell))\}_{\ell=1,\dots,L}\subset \mathbb{R}^2$ in the case of one mode $f(t)=s_1(2\pi N_1\phi_1(t))$, and the point distribution characterizes the shape function $s_1(2\pi v)$ exactly. Right:  the point set $\{\tau(v_\ell,h_1(v_\ell))\}_{\ell=1,\dots,L}\subset \mathbb{R}^2$ in the case of two modes $f(t)=\sum_{k=1}^2s_k(2\pi N_k\phi_k(t))$, and the point distribution behaves like noisy observation of $s_1(2\pi v)$ with an additive noise $\kappa_1(2\pi v)$. The partition-based regression computes the regression function (in red) at a point $v^{bk}$ as the average height of all observations (black points) with sampling locations in a small neighborhood of $v^{bk}$.}
  \label{fig:s}
\end{figure}

\begin{figure}
  \begin{center}
    \begin{tabular}{cc}
      \includegraphics[height=1.8in]{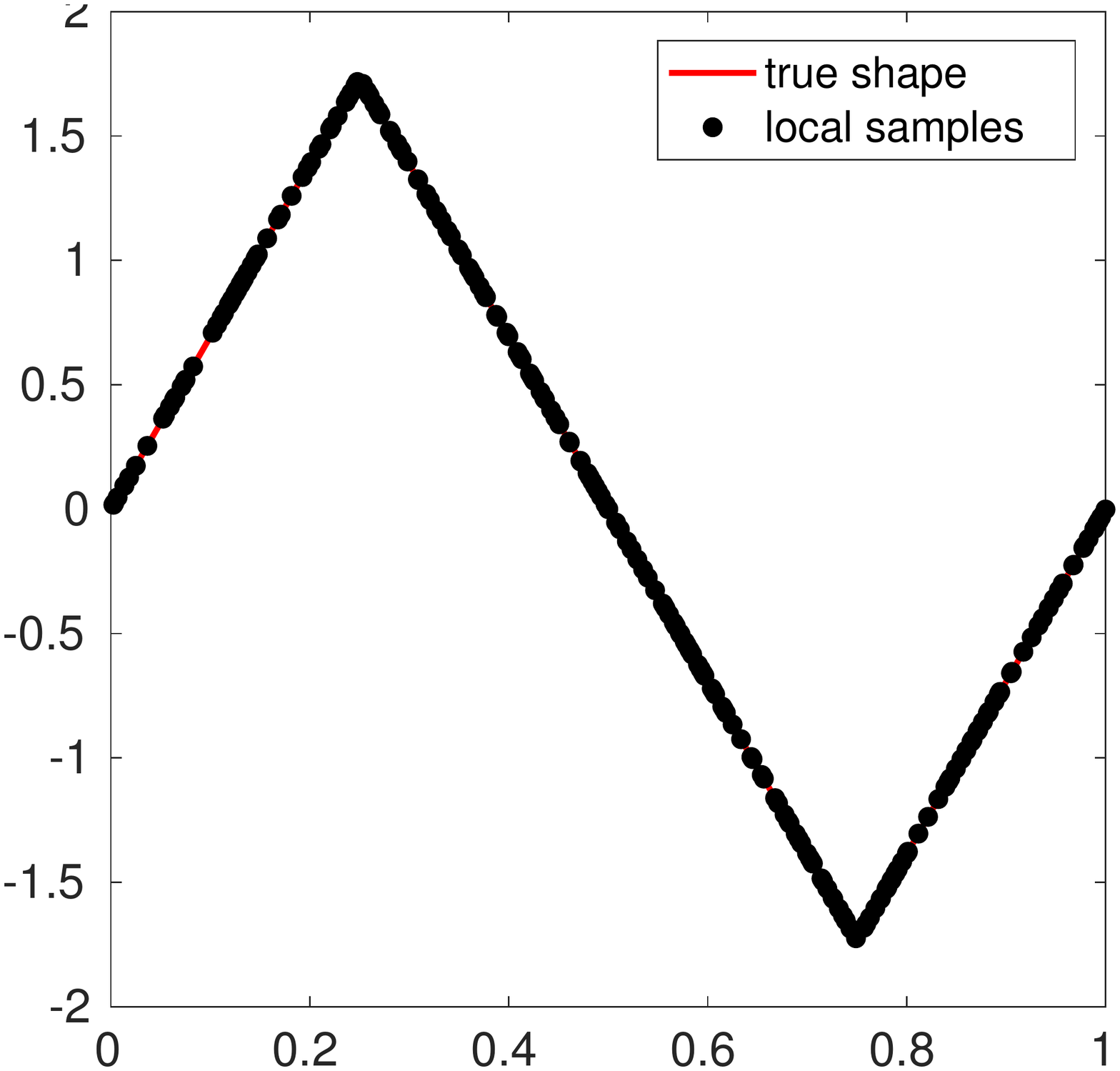} & 
      \includegraphics[height=1.8in]{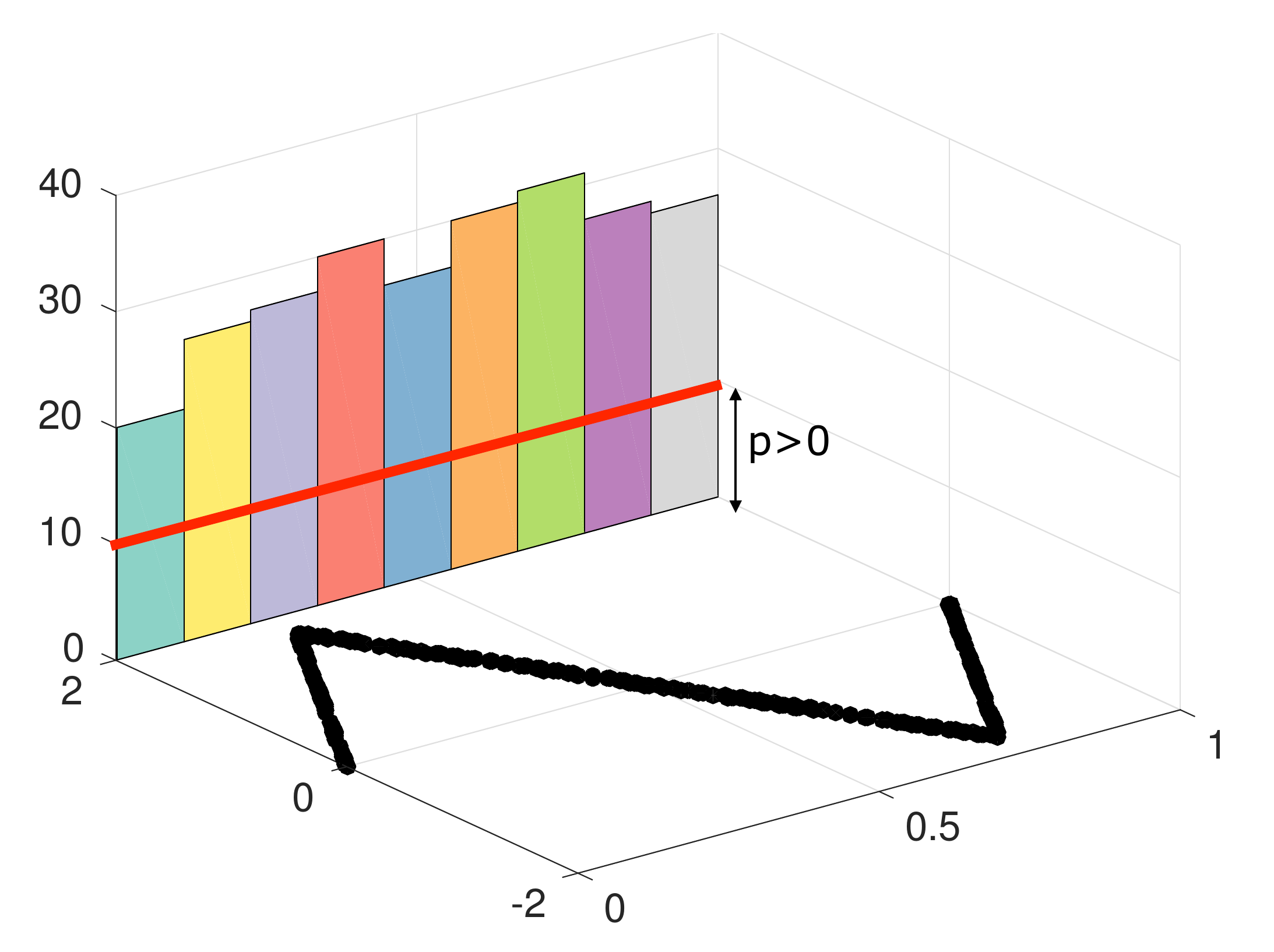}\\
    \end{tabular}
  \end{center}
  \caption{Left: the point set $\{\mod(z_\ell^{bk},1), \kappa_1(2\pi p_1\circ p_2^{-1}(\mod(z^{bk}_\ell,1)))\}_\ell = \{\mod(z_\ell^{bk},1), s_2(2\pi \mod(z^{bk}_\ell,1))\}_\ell \subset \mathbb{R}^2$.  Right: a histogram of the points $\{\mod(z_\ell^{bk},1)\}$ in the left figure.}
  \label{fig:s2}
\end{figure}

  {However, in the case of multiple modes, $\kappa_k(2\pi v)\neq 0$ and $(x_\ell^{(k)},y_\ell^{(k)})=\tau(v_\ell,s_k(2\pi v_\ell)+\kappa_k(2\pi v_\ell))$ for $\ell =1,\dots,L$ can be considered as noise-contaminated i.i.d. samples of a random vector $(X_k,Y_k)$ (see Figure \ref{fig:s} (right) for an illustration). Hence, the regression in \eqref{eqn:re} would not estimate the ground true shape function precisely. Let us use the partition-based regression method (or partitioning estimate) in Chapter 4 of \cite{regressionBook} as an example. Given a small step side $h\ll 1$, the time domain $[0,1]$ is uniformly partitioned into $N^h = \frac{1}{h}$ (assumed to be an integer) parts $\{[t^h_n,t^h_{n+1})\}_{n=0,\dots,N^h-1}$, where $t^h_n=nh$. Let $s_k^P$ denote the estimated regression function by the partition-based regression method with $L$ samples, $(x_\ell^{(k)},y_\ell^{(k)})_{\ell=1,\dots,L}$. Following Chapter 4 of \cite{regressionBook}, we define
\begin{equation}\label{eqn:pb}
 s_k^P(x) \defeq\frac{\sum_{\ell=1}^L\mathcal{X}_{[t^h_n,t^h_{n+1})} (x_\ell^{(k)}) y_\ell^{(k)} }{\sum_{\ell=1}^L\mathcal{X}_{[t^h_n,t^h_{n+1})}(x_\ell^{(k)}) }
\end{equation}
as the partition-based regression function, 
when $x\in [t^h_n,t^h_{n+1})$, where $\mathcal{X}_{[t^h_n,t^h_{n+1})}(x)$ is the indicator function supported on $[t^h_n,t^h_{n+1})$. As illustrated by black points in Figure \ref{fig:s} (right), the partition-based regression essentially computes $ s_k^P(x)$ as the average height of all observations $y_\ell^{(k)}$ with sampling locations $x_\ell^{(k)}$ in a small neighborhood of $x$. Let $s^E_k\defeq s^P_k-s_k$ be the residual shape function, then the residual error of the GMD after one step of regressions on all modes is
 \begin{equation}\label{eqn:rrr}
 r(t)=f(t)- \sum_{k=1}^K \alpha_k(t) s^P_k(2 \pi N_k \phi_k(t))\approx-\sum_{k=1}^K \alpha_k(t) s^E_k(2 \pi N_k \phi_k(t)),
 \end{equation}
  which might be large, since in general $s_k\approx s^P_k$ is not true.}
  
 {The deviation of $s^P_k$ from $s_k$ comes from the influence of $\kappa_k$ in \eqref{eqn:ks}. After the folding map in \eqref{eqn:fold}, we hope that $\{\kappa_k(2\pi \mod(v_\ell,1))\}_{1\leq \ell\leq L}$ behave like i.i.d. samples of a mean-zero random noise so that $s^P_k\approx s_k$. To understand the behavior of $\kappa_k(2\pi \mod(v,1))$, let us take the example of $\kappa_1(2\pi \mod(v,1))$ in Figure \ref{fig:s} (right). If we unfold the black samples $\{h_1(\mod(v^{bk}_\ell,1))=s_1(2\pi \mod(v^{bk}_\ell,1)) +\kappa_1(2\pi \mod(v^{bk}_\ell,1))\}$\footnote{Here $^{bk}$ means samples corresponding to black points in Figure \ref{fig:ks} to \ref{fig:s2}.} (samples with sampling locations $\mod(v^{bk}_\ell,1)$ near the point $v^{bk}$ in Figure \ref{fig:s} (right)) back to $h_1(v^{bk}_\ell)$, we see that the black samples $\kappa_1(2\pi \mod(v^{bk}_\ell,1))$ come from the black samples $\kappa_1(2\pi v^{bk}_\ell)$ in Figure \ref{fig:ks} (right). If we warp $\kappa_1(2\pi v^{bk}_\ell)$ back to $\kappa_1(2\pi p_1(t^{bk}_\ell))$, inverse-warp $\kappa_1(2\pi p_1(t^{bk}_\ell))$ to $\kappa_1(2\pi p_1\circ p_2^{-1}(z^{bk}_\ell))$, where $z:=p_2(t)$, and finally fold $\kappa_1(2\pi p_1\circ p_2^{-1}(z^{bk}_\ell))$ into one period to obtain 
 \begin{equation}\label{eqn:fold2}
 \kappa_1(2\pi p_1\circ p_2^{-1}(\mod(z^{bk}_\ell,1))),
 \end{equation}
 then from Figure \ref{fig:s2} (left) we see that the black samples $\kappa_1(2\pi p_1\circ p_2^{-1}(\mod(z^{bk}_\ell,1)))$ essentially cover the shape function $s_2$. In fact, by the definition of $\kappa_2$, we have
 \begin{equation*}
 \kappa_1(2\pi p_1\circ p_2^{-1}(\mod(z^{bk}_\ell,1)))=s_2(2\pi \mod(z^{bk}_\ell,1)).
 \end{equation*} 
Figure \ref{fig:s2} (right) shows a histogram of the sampling locations $\{\mod(z^{bk}_\ell,1)\}$. Note that $s_2$ has mean zero. Hence, if the point distribution $\{\mod(z^{bk}_\ell,1)\}$ is almost uniform in $[0,1]$, then by the partition-based regression formula, we see
 \begin{eqnarray}
s^P_1(2\pi v^{bk})&= &s_1(2\pi v^{bk}) + \frac{\sum_{\ell} \kappa_1(2\pi \mod(v^{bk}_\ell,1)) }{ \text{ the number of samples of }\{\mod(v^{bk}_\ell,1)\} }\nonumber\\
&= &s_1(2\pi v^{bk}) + \frac{\sum_{\ell} \kappa_1(2\pi v^{bk}_\ell) }{ \text{ the number of samples of }\{v^{bk}_\ell\} }\nonumber\\
&=& s_1(2\pi v^{bk}) + \frac{\sum_{\ell} \kappa_1(2\pi p_1(t^{bk}_\ell)) }{ \text{ the number of samples of }\{t^{bk}_\ell\} }\nonumber\\
&=& s_1(2\pi v^{bk}) + \frac{\sum_{\ell}  \kappa_1(2\pi p_1\circ p_2^{-1}(z^{bk}_\ell)) }{ \text{ the number of samples of }\{z^{bk}_\ell\} }\nonumber\\
&=& s_1(2\pi v^{bk}) + \frac{\sum_{\ell}  \kappa_1(2\pi p_1\circ p_2^{-1}(\mod(z^{bk}_\ell,1))) }{ \text{ the number of samples of }\{\mod(z^{bk}_\ell,1)\} }\nonumber\\
&=& s_1(2\pi v^{bk}) + \frac{\sum_{\ell}  s_2(2\pi \mod(z^{bk}_\ell,1)) }{ \text{ the number of samples of }\{\mod(z^{bk}_\ell,1)\} }\nonumber\\
&\approx & s_1(2\pi v^{bk}) +\int_0^1 s_2(2\pi z)dz\\
&=&  s_1(2\pi v^{bk}).\nonumber
\label{eqn:acc}
\end{eqnarray}
In the case when the point distribution $\{\mod(z^{bk}_\ell,1)\}$ is far from being uniform in $[0,1]$, we see that $s^P_1$ is not close to $s_1$.}
  
{Note that the residual $r(t)$ in \eqref{eqn:rrr} can be viewed as a new superposition of modes with new shape functions $\{-s^E_k\}$. This motivates the Jacobi recursive scheme in \cite{HZYregression} that repeats the same decomposition procedure to decompose the residual $r(t)$, and update the shape function estimation, until the residual is eliminated. It was proved in \cite{HZYregression} that if the empirical distribution of $\{\mod(z^{bk}_\ell,1)\}$ corresponding to $v^{bk}$ for any $v^{bk}$ is uniformly and strictly positive (e.g., larger than $p>0$ as in Figure \ref{fig:s2} (right)), then the recursive scheme is able to eliminate the residual $r(t)$.} 

To improve the convergence rate of the Jacobi style recursive scheme, this paper proposes a Gauss-Seidel recursive scheme in Algorithm \ref{alg:RDBR}. More importantly, the convergence of the Jacobi recursive scheme is sensitive to the prior information as inputs. For example, if the number of modes $K$ is not known exactly, the Jacobi recursive scheme may fail to converge, while the Gauss-Seidel recursive scheme would not. In practice, the input prior information is:
\begin{enumerate}
\item an estimated number of components $\bar{K}$;
\item a set of estimated phase functions
\begin{equation}
\label{eqn:prior1}
p_k(t)=n_k N_{\tau_k} \phi_{\tau_k}(t)-n_k N_{\tau_k} \phi_{\tau_k}(0),
\end{equation}
for $k=1$, $\dots$, $\bar{K}$, where $\tau_k\in\{1,\dots,K\}$ and $n_k\in \mathbb{Z}^+$ are unknown integers;
\item and the corresponding amplitude functions
\begin{equation}
\label{eqn:prior2}
q_k(t)=a_k \alpha_{\tau_k}(t),
\end{equation}
for $k=1$, $\dots$, $\bar{K}$, where $a_k>0$ is unknown.
\end{enumerate}
Note that the curves $\{p_k'(t)=n_k N_{\tau_k} \phi'_{\tau_k}(t)\}$ naturally belong to a few groups, each of which corresponds to the multiple of a fundamental instantaneous frequency $N_k\phi'_k(t)$ for some $k$. Following the curve classification idea in Algorithm 3.7 and Theorem 3.9 in \cite{1DSSWPT}, we are able to classify the curves $\{p_k'(t)=n_k N_{\tau_k} \phi'_{\tau_k}(t)\}$ and identify the corresponding fundamental instantaneous frequencies $\{N_{\tau_k}\phi'_{\tau_k}(t)\}$, which give the fundamental instantaneous phases up to an unknown initial phase $\{N_{\tau_k}\phi_{\tau_k}(t)-N_{\tau_k}\phi_{\tau_k}(0)\}$. Therefore, we can assume that $\{\tau_k\}_{1\leq k\leq\bar{K}}$ are distinct and $n_k=1$ for all $k$ in Equation \eqref{eqn:prior1} and \eqref{eqn:prior2} for clean data (i.e. the prior information only contains the fundamental instantaneous frequencies without their multiples, but the prior information may still miss some instantaneous frequencies). The reader is referred to \cite{1DSSWPT} for more detail. However, in the presence of noise, the instantaneous frequency estimations may have large errors leading the failure of the curve classification. In this case, the prior information may contain instantaneous frequencies that are multiples of the fundamental instantaneous frequencies (i.e., there may be some $n_k\neq 1$). Even if the prior information above misses some fundamental frequencies or contains the multiples of fundamental frequencies, the Gauss-Seidel recursive scheme below can still recover reasonably accurate modes from their superposition.

\begin{algorithm2e}[H]
\label{alg:RDBR}
\caption{Gauss-Seidel recursive diffeomorphism-based regression (RDBR).}
Input: $L$ points of i.i.d. measurement $\{f(t_\ell)\}_{\ell=1,\dots,L}$ with $t_\ell \in [0,1]$, estimated instantaneous phases $\{p_k\}_{k=1,\dots,\bar{K}}$, amplitudes $\{q_k\}_{k=1,\dots,\bar{K}}$, an accuracy parameter $\epsilon<1$, and the maximum iteration number $J$.

Output: the estimated shape functions  $\{\bar{s}_k\}_{k=1,\dots,\bar{K}}$ and the estimated modes $\{\bar{f}_k(t)\}_{k=1,\dots,\bar{K}}=\{{q}_k(t) \bar{s}_k(2\pi p_k(t))\}_{k=1,\dots,\bar{K}}$ at the sampling grid points $\{t_\ell\}_{1\leq \ell\leq L}$. 

Initialize: let $r_1^{(0)}=f$, $\epsilon_1=\epsilon_2=1$, $\epsilon_0=2$, the iteration number $j=0$, $\dot{s}^{(0)}_k=0$, and $\bar{s}_k^{(0)}=0$ for all $k=1,\dots,\bar{K}$.

Compute $N_k$ as the integer nearest to the average of $p'_k(t)$ for $k=1,\dots,\bar{K}$.

Sort $\{N_k\}_{1\leq k\leq \bar{K}}$ in an ascending order and reorder the amplitude and phase functions accordingly.

\While{$j<J$, $\epsilon_1> \epsilon$, $\epsilon_2>\epsilon$, and $|\epsilon_1-\epsilon_0|>\epsilon$}{ 

\For{$k=1,\dots,\bar{K}$}{

Define \[h^{(j)}_k= \frac{r_k^{(j)}\circ p_k^{-1}}{q_k\circ p_k^{-1}},\]  and we know it is sampled on grid points $v_\ell=p_k(t_\ell)$.

Observe that $\left\{\tau(v_\ell,h^{(j)}_k(v_\ell))\right\}_{\ell=1,\dots,L}$ behaves like a sequence of i.i.d. samples of a certain random vector $(X_k,Y^{(j)}_k)$ with $X_k\in [0,1)$.
  
Solve the distribution-free regression problem
\begin{equation}
\label{eqn:reg}
\dot{s}^{(j+1)}_k \approx s^{R,(j+1)}_k=\underset{s:\mathbb{R}\rightarrow \mathbb{R}}{\arg\min}\quad \E\{\left| s(2\pi X_k)-Y^{(j)}_k\right|^2\},
\end{equation}
where $\dot{s}^{(j+1)}_k$ denotes the numerical solution approximating the ground truth solution $s^{R,(j+1)}_k$.

Update $\dot{s}_k^{(j+1)}=\dot{s}_k^{(j+1)}-\frac{1}{2\pi}\int_0^{2\pi} \dot{s}_k^{(j+1)}(t)dt$ for all $k$.

Let $\bar{s}_k^{(j+1)}=\bar{s}_k^{(j)}+\dot{s}_k^{(j+1)}$ for all $k$.

If $k<\bar{K}$, then let $r_{k+1}^{(j)}=r_k^{(j)}-q_k(t) \dot{s}^{(j+1)}_k(2 \pi p_k(t))$; otherwise, let $r_{1}^{(j+1)}=r_k^{(j)}-q_k(t) \dot{s}^{(j+1)}_k(2 \pi p_k(t))$.

    }

Update $\epsilon_0=\epsilon_1$, $\epsilon_1=\|r_1^{(j+1)}\|_{L^2}$, $\epsilon_2=\max_{k}\{\|\dot{s}_k^{(j+1)}\|_{L^2}\}$.

Set $j = j + 1$. 
}
Let $\bar{s}_k=\bar{s}_k^{(j+1)}$ for all $k$.
\end{algorithm2e}

Remark: as pointed out in the Jacobi style RDBR \cite{HZYregression}, the unknown shifts $\{n_k N_{\tau_k}\phi_{\tau_k}(0)\}$ in Equation \eqref{eqn:prior1} and the unknown prefactors $\{a_k\}$ in Equation \eqref{eqn:prior2} have been absorbed in the estimation of shape functions and the reconstruction of modes. Hence, it is not necessary to know them a prior. In the convergence analysis, instead of Equation \eqref{eqn:prior1} and \eqref{eqn:prior2}, we assume that 
\begin{equation}
\label{eqn:prior3}
p_k(t)=n_k N_{\tau_k} \phi_{\tau_k}(t),
\end{equation}
and
\begin{equation}
\label{eqn:prior4}
q_k(t)=\alpha_{\tau_k}(t).
\end{equation}

\subsection{Convergence analysis}
\label{sec:thmRDBR}

In this section, an asymptotic analysis on the convergence of the Gauss-Seidel recursive diffeomorphism-based regression (RDBR) in Algorithm \ref{alg:RDBR} is provided for a class of time series that is a superposition of several generalized intrinsic mode functions as follows. To make the analysis self-contained, a few definitions in \cite{HZYregression} will be repeated.

\begin{definition} Generalized shape functions:
\label{def:GSF}
The generalized shape function class ${\cal S}_M$ consists of $2\pi$-periodic functions $s(t)$ in the Wiener Algebra with a unit $L^2([0,2\pi])$-norm and a $L^\infty$-norm bounded by $M$ satisfying the following spectral conditions:
\begin{enumerate}
\item The Fourier series of $s(t)$ is uniformly convergent;
\item $\sum_{n=-\infty}^{\infty}|\widehat{s}(n)|\leq M$ and $\widehat{s}(0)=0$;
\item Let $\Lambda$ be the set of integers $\{|n|: \widehat{s}(n)\neq 0\}$. The greatest common divisor $\gcd(s)$ of all the elements in $\Lambda$ is $1$.
\end{enumerate}
\end{definition}

\begin{definition}
  \label{def:GIMFF}
  A function $f(t)=\alpha(t)s(2\pi N \phi(t))$ for $t\in[0,1]$ is a generalized intrinsic mode function (GIMF) of type $(M,N)$, if $s(t)\in {\cal S}_M$ and $\alpha(t)$ and $\phi(t)$ satisfy the conditions\footnote{
  In the analysis of the synchrosqueezed transform in \cite{1DSSWPT}, a GIMF of type $(M,N)$ requires stronger conditions as follows:
  \begin{align*}
    \alpha(t)\in C^\infty, \quad |\alpha'|\leq M, \quad 1/M \leq \alpha\leq M \\
    \phi(t)\in C^\infty,  \quad  1/M \leq | \phi'|\leq M, \quad |\phi''|\leq M.
   \end{align*}
   However, the RDBR requires much weaker conditions.
} below:
  \begin{align*}
    \alpha(t)\in C^\infty, \quad 1/M \leq \alpha\leq M, \quad  \phi(t)\in C^\infty,  \quad  1/M \leq | \phi'|\leq M.
   \end{align*}
\end{definition}

In the analysis of the Gauss-Seidel RDBR, the partition-based regression method (or partitioning estimate) in Chapter 4 of \cite{regressionBook} will be adopted. Recall the definition of this regression introduced in \eqref{eqn:pb}:
\[
s^R(x)\approx s^P(x) \defeq\frac{\sum_{\ell=1}^L\mathcal{X}_{[t^h_n,t^h_{n+1})} (x_\ell) y_\ell }{\sum_{\ell=1}^L\mathcal{X}_{[t^h_n,t^h_{n+1})}(x_\ell) },
\]
when $x\in [t^h_n,t^h_{n+1})$. The following theorem given in Chapter 4 in \cite{regressionBook} estimates the $L_2$ risk of the approximation $s^P\approx s^R$ as follows.

\begin{theorem}
\label{thm:reg}
For the uniform partition with a step side $h$ in $[0,1)$ as defined just above, assume that 
\[
\Var(Y|X=x)\leq \sigma^2,\quad x\in\mathbb{R},
\]
\[
|s^R(x)-s^R(z)|\leq C|x-z|,\quad x,z\in\mathbb{R},
\]
$X$ has a compact support $[0,1)$, and there are $L$ i.i.d. samples of $(X,Y)$. Then the partition-based regression method provides an estimated regression function $s^P$ to approximate the ground truth regression function $s^R$, where
\begin{equation*}
s^R = \underset{s:\mathbb{R}\rightarrow \mathbb{R}}{\arg\min}\quad \E\{\left| s(2\pi X)-Y\right|^2\},
\end{equation*}
with an $L^2$ risk bounded by
\[
\E\|s^P-s^R\|^2\leq c_0\frac{\sigma^2+\|s^R\|^2_{L^\infty}}{Lh}+C^2h^2,
\]
where $c_0$ is a constant independent of the number of samples $L$, the regression function $s$,  the step side $h$, and the Lipschitz continuity constant $C$.
\end{theorem}

{As we shall see later in the assumption of Theorem \ref{thm:conv}, when the small step size $h$ in the partition-based regression is fixed, we require the number of samples $L$ to be sufficiently large such that the $L_2$ risk of the approximation $s^P\approx s^R$ is small enough.} To simplify notations, let $\LL$ be the class of functions that are Lipschitz continuous with a constant $C$. Denote the set of sampling grid points $\{t_\ell\}_{\ell=1,\dots,L}$ in Algorithm \ref{alg:RDBR} as $\T$. To estimate the regression function using the partition-based regression method, $\T$ is divided into several subsets as follows. For $i,j=1,\dots,K$, $i\neq j$, $m,n=0,\dots,N^h-1$, let 
\[
\T^{ij}_h(m,n)=\left\{ t\in \T: \mod(p_i(t),1)\in[t^h_m,t^h_m+h),\mod(p_j(t),1)\in[t^h_n,t^h_n+h)\right\},
\]
and
\[
\T^{i}_h(m)=\left\{ t\in \T: \mod(p_i(t),1)\in[t^h_m,t^h_m+h)\right\},
\]
then $\T=\cup_{m=0}^{N^h-1} \T^{i}_h(m)=\cup_{m=0}^{N^h-1}\cup_{n=0}^{N^h-1} \T^{ij}_h(m,n)$. 
Let 
\begin{equation}
\label{eqn:D}
D^{ij}_h(m,n)\quad \text{ and }\quad D^{i}_h(m)
\end{equation}
 denote the number of points in  $\T^{ij}_h(m,n)$ and $\T^{i}_h(m)$, respectively. 

\begin{definition}
  \label{def:wd}
  Suppose phase functions $p_k(t)= N_k \phi_k(t)$ for $t\in[0,1]$, and $k=1,\dots,K$, where $\phi_k(t)$ satisfies\footnote{
  In the analysis of the Jacobi style RDBR in \cite{HZYregression}, stronger conditions for the well-differentiated phase functions were required as follows:  \begin{align*}
    \phi_k(t)\in C^\infty,  \quad  1/M \leq | \phi_k'|\leq M, \quad |\phi_k''|\leq M.
   \end{align*}
  In fact, these conditions can be further weakened when the phase functions are assumed to be known and the synchrosqueezed transform is not involved in the analysis of RDBR.
  } 
  \begin{align*}
    \phi_k(t)\in C^\infty,  \quad  1/M\leq | \phi_k'|\leq M.
   \end{align*}
Then the collection of phase functions $\{p_k(t)\}_{1\leq k\leq K}$ is said to be $(M,N,K,h,\beta,\gamma)$-well-differentiated and denoted as $\{p_k(t)\}_{1\leq k\leq K}\subset\WD(M,N,K,h,\beta,\gamma)$, if the following conditions are satisfied:
\begin{enumerate}
\item $N_k\geq N$ for $k=1,\dots,K$;
\item $  \gamma\defeq  \underset{m,n,i\neq j }{\min}D^{ij}_h(m,n)$ satisfies $\gamma>0$ 
, where $D^{ij}_h(m,n)$ (and $D^{i}_h(m)$  below) is defined in \eqref{eqn:D};
\item Let \[\beta_{i,j} \defeq \left(\sum_{m=0}^{N^h-1} \frac{1}{D^{i}_h(m) }\left( \sum_{n=0}^{N^h-1} ( D^{ij}_h(m,n)-\gamma )^2 \right)\right)^{1/2}\] for all $i\neq j$, then $\beta\defeq \max\{\beta_{i,j}:i\neq j\}$ satisfies $M^2(K-1)\beta<1$.
\end{enumerate}
\end{definition}


In the above definition, $\gamma$ quantifies the dissimilarity between phase functions. The larger $\gamma$ is, the more dissimilarity phase functions have. If two phases are very similar, there might be some nearly empty sets $\T^{ij}_h(m,n)$ and hence $\gamma$ is small.  If $\gamma$ is larger, the numbers $\{D^{ij}_h(m,n)\}_{m,n}$ are closer and $\beta$ would be smaller. To guarantee a large $\gamma$, $N$ and $L$ should be sufficiently large. {To give more intuition of the well-differentiation, let us revisit the example in Figure \ref{fig:ks} to \ref{fig:s2}. Note that: 1) the second condition $\gamma\defeq  \underset{m,n,i\neq j }{\min}D^{ij}_h(m,n)>0$  requires that the empirical distribution of $\{\mod(z^{bk}_\ell,1)\}$ in \eqref{eqn:fold2} is uniformly and strictly positive for all cases in the partition-based regression; 2) the number $\beta$ in the third condition quantifies how uniform the distribution of $\{\mod(z^{bk}_\ell,1)\}$ is in the partition-based regression. The fact that these distributions are close to being uniform is equivalent to the fact that $\beta\approx 0$. As we have seen in the example in \eqref{eqn:acc}, if the distribution of $\{\mod(z^{bk}_\ell,1)\}$ is close to being uniform, which can be guaranteed if $\gamma$ is large and $\beta$ is close to zero, then one-step of diffeomorphism-based regression is already very accurate. Repeatedly applying the diffeomorphism-based regression can quickly eliminate the residual error.} 

With these notations defined, we are ready to present the main analysis of the Gauss-Seidel RDBR. Let's recall that in each iteration of Algorithm \ref{alg:RDBR}, if we denote the target shape function as $s_k^{(j)} $ then the given data is 
\begin{equation}
\label{eqn:r}
r_k^{(j)}(t)=\sum_{\ell=1}^{k-1} \alpha_\ell(t) s^{(j+1)}_\ell(2 \pi p_\ell(t)) + \sum_{\ell=k}^{K} \alpha_\ell(t) s^{(j)}_\ell(2 \pi p_\ell(t)).
\end{equation}
By reformulating the regression problem 
\begin{eqnarray}
s_k^{R,(j+1)} &=& \underset{s:\mathbb{R}\rightarrow \mathbb{R}}{\arg\min}\quad \E\{\left| s(2\pi X_k)-Y_k^{(j)} \right|^2\}\\
&=&  \underset{s:\mathbb{R}\rightarrow \mathbb{R}}{\arg\min}\quad \E\{\left| s(2\pi X_k)-(Y_k^{(j)} -s_k^{(j)}(2\pi X_k))\right|^2\}-s_k^{(j)},
\label{eqn:regthm}
\end{eqnarray}
we see that 
\[
s_k^{R,(j+1)} = s_k^{(j)} + s^{E,(j)}_k,
\]
where
\begin{equation}
\label{eqn:regthm1}
s^{E,(j)}_k(2\pi x)\defeq \E\{Y_k^{(j)} -s_k^{(j)}(2\pi X_k)|X_k=x\}\neq 0
\end{equation}
due to the perturbation caused by other modes. In the next iteration, the target shape function $s^{(j+1)}_k= - s^{E,(j)}_k$. Hence, the key convergence analysis is to show that $s^{E,(j)}_k$ decays as $j\rightarrow \infty$. 


In what follows, we assume that an accuracy parameter $\epsilon$ is fixed. Furthermore, suppose GIMF's $f_k(t)=\alpha_k(t)s_k(2\pi N_k \phi_k(t))$, $k=1,\dots,K$, have phases in $\WD(M,N,K,h,\beta,\gamma)$, all generalized shape functions and amplitude functions are in the space $\LL$. Under these conditions, all regression functions $s^{(j)}_k\in\LL$ and have bounded $L^\infty$ norm linearly depending only on $M$ and $K$. By Line $8$ in Algorithm \ref{alg:RDBR}, we have the nice and key condition that $\int_0^1 s^{(j)}_k(2\pi t)dt=0$ at each iteration for all $k$ and $j$. Note that $\Var(Y_k^{(j)}|X_k=x)$ is bounded by a constant linearly depending only on $M$ and $K$ as well. For the fixed $\epsilon$ and $C$, there exists $h_0(\epsilon,C)$ such that $C^2h^2<\epsilon^2$ if $0<h<h_0$. By the abuse of notation, $O(\epsilon)$ is used instead of $Ch$ later. By Theorem \ref{thm:reg}, for the fixed $\epsilon$, $M$, $K$, $C$, and $h$, there exists $L_0(\epsilon,M,K,C,h)$ such that the $L^2$ error of the partition-based regression is bounded by $\epsilon^2$. In what follows, $h$ is smaller than $h_0$, $L$ is larger than $L_0$, and hence all estimated regression functions approximate the ground truth regression function with an $L^2$ error of order $\epsilon$. Under these conditions and assumptions, as long as $N$ and $L$ are large enough, $s^{E,(j)}_k$ is shown to decay to $O(\epsilon)$ as $j\rightarrow \infty$, and the decay rate will be estimated.

\begin{theorem}
\label{thm:conv} (Convergence of the Gauss-Seidel RDBR when $\bar{K}=K$) Suppose $\bar{K}=K$, $\{1,\dots,K\}=\{\tau_1,\dots,\tau_K\}$, $n_k=1$ for all $k$ in Equation \eqref{eqn:prior3} and \eqref{eqn:prior4}. 
Under the conditions listed in the paragraph immediately preceding this theorem,  we have  
 \[
  \|s^{E,(j)}_k\|_{L^2}\leq O(c_0 \epsilon+( M^2(K-1)\beta)^{c(j,k)} )
 \]
and
\[
\|r_k^{(j+1)}\|_{L^2}\leq O(c_0\epsilon+\left(M^2(K-1)\beta\right)^{c(j,1)})
\]
for all $j\geq 0$ and $1\leq k\leq K$, where $c(j,k)=\left\lceil \frac{jK+k}{K-1}\right\rceil$, the smallest integer larger than or equal to $\frac{jK+k}{K-1}$, $c_0=\frac{1}{1-M^2(K-1)\beta}$ is a constant number, $s^{E,(j)}_k$  is defined in Equation \eqref{eqn:regthm1} and $r_k^{(j)}$ is defined in Equation \eqref{eqn:r}.
\end{theorem}

\begin{proof}
First, we start with the case when $K=2$ and $\alpha_k(t)=1$ for all $t$ and $k$.

Recall that $p_k(t)$ can be considered as a diffeomorphism from $\R$ to $\R$ transforming data in the $t$ domain to the $p_k(t)$ domain. We have introduced the inverse-warping data
\begin{eqnarray*}
h_k^{(j)}(v)& =& r_k^{(j)}\circ p_k^{-1}(v)\\
&=&s_k^{(j)}(2\pi v)+ \sum_{\ell=1}^{k-1}  s_\ell^{(j+1)}(2\pi   p_\ell\circ p_k^{-1}(v))+\sum_{\ell=k+1}^K  s_\ell^{(j)}(2\pi   p_\ell\circ p_k^{-1}(v))\\
&\defeq & s_k^{(j)}(2\pi v)+ \kappa_k^{(j)}(2\pi v),
\end{eqnarray*}
where $v = p_k(t)$. After the folding map
\begin{eqnarray*}
\tau: \ \      \left(v, h_k(v) \right)   \mapsto    \left(\text{mod}(v,1),  h_k^{(j)}(v) \right),
\end{eqnarray*}
we have $(x_\ell,y_\ell)=\tau(v_\ell,s_k^{(j)}(2\pi v_\ell)+\kappa_k^{(j)}(2\pi v_\ell))$ for $\ell =1,\dots,L$ as $L$ i.i.d. samples of a random vector $(X_k,Y_k^{(j)})$, where $X_k\in[0,1]$. We can assume the target shape functions $s^{(j)}_k$ for all $k$ at the $j$th step are known in the analysis, although they are not known in practice. The partition-based regression method is applied (not necessary to know the distribution of the random vector $(X_k,Y_k^{(j)})$) to solve the following regression problem approximately
\begin{equation}\label{eqn:resp}
 \underset{s:\mathbb{R}\rightarrow \mathbb{R}}{\arg\min}\quad \E\{\left| s(2\pi X_k)-(Y_k^{(j)} -s_k^{(j)}(2\pi X_k))\right|^2\},
 \end{equation}
 and the solution is denoted as $s^{P,(j)}_k$. We would like to emphasize that $s^{P,(j)}_k$ is only used in the analysis and it is not computed in Algorithm \ref{alg:RDBR}. Recall notations in Definition \ref{def:wd}. By the partition-based regression method, when $x\in[t^h_m,t^h_m+h)$,
 \[
 s^{P,(j)}_1(2\pi x)=\frac{\sum_{n=0}^{N^h-1}\left( s^{(j)}_2(2\pi t^h_n)+O(\epsilon)\right)D^{12}_h(m,n)}{D^1_h(m)},
 \]
 and 
  \begin{equation}
  \label{eqn:s2}
 s^{P,(j)}_2(2\pi x)=\frac{\sum_{n=0}^{N^h-1}\left( s^{(j+1)}_1(2\pi t^h_n)+O(\epsilon)\right)D^{21}_h(m,n)}{D^2_h(m)},
 \end{equation}
 where $O(\epsilon)$ comes from the approximation of the $\LL$ function $s_i$ using the values on grid points $t^h_n$. Note that in the case of the Jacobi style RDBR in \cite{HZYregression}, the term $s^{(j+1)}_1$ in Equation \eqref{eqn:s2} is replaced with $s^{(j)}_1$, since the Jacobi style iteration doesn't use the latest estimations of shape functions. The following argument is similar to that for Lemma $3.3$ in \cite{HZYregression}. It is easy to check that 
 \[
  |s^{P,(j)}_1(2\pi x)|\leq O(\epsilon)+\frac{\sum_{n=0}^{N^h-1}s^{(j)}_2(2\pi t^h_n) \left(D^{12}_h(m,n)-\gamma\right)}{D^1_h(m)}
 \]
 and
  \[
  |s^{P,(j)}_2(2\pi x)|\leq O(\epsilon)+\frac{\sum_{n=0}^{N^h-1}s^{(j+1)}_1(2\pi t^h_n) \left(D^{21}_h(m,n)-\gamma\right)}{D^2_h(m)},
 \]
which imply that
 \begin{equation*}
 \|s^{P,(j)}_1\|_{L^2}=O(\epsilon) +\left( \sum_{n=0}^{N^h-1}\left(s^{(j)}_2(2\pi t^h_n) \right)^2h \right) ^{1/2}\left( \sum_{m=0}^{N^h-1}\left(\sum_{n=0}^{N^h-1}\left(  \frac{ D^{12}_h(m,n)-\gamma}{D^1_h(m)} \right)^2\right)  \right)^{1/2}
 \end{equation*}
 and
  \begin{equation*}
 \|s^{P,(j)}_2\|_{L^2}=O(\epsilon) +\left( \sum_{n=0}^{N^h-1}\left(s^{(j+1)}_1(2\pi t^h_n) \right)^2h \right) ^{1/2}\left( \sum_{m=0}^{N^h-1}\left(\sum_{n=0}^{N^h-1}\left(  \frac{ D^{21}_h(m,n)-\gamma}{D^2_h(m)} \right)^2\right)  \right)^{1/2}
 \end{equation*}
by the triangle inequality and H{\"o}lder's inequality. Since $s^{(j)}_2$ and $s^{(j+1)}_1$ are in $\LL$, we have
 \[
 \left( \sum_{n=0}^{N^h-1}\left(s^{(j)}_2(2\pi t^h_n) \right)^2h \right) ^{1/2}=\|s^{(j)}_2\|_{L^2}+O(\epsilon)
 \]
 and
  \[
 \left( \sum_{n=0}^{N^h-1}\left(s^{(j+1)}_1(2\pi t^h_n) \right)^2h \right) ^{1/2}=\|s^{(j+1)}_1\|_{L^2}+O(\epsilon).
 \]
 Since phase functions are in $\WD(M,N,h,\beta,\gamma)$, 
 \[
 \left( \sum_{m=0}^{N^h-1}\left(\sum_{n=0}^{N^h-1}\left(  \frac{ D^{ki}_h(m,n)-\gamma}{D^k_h(m)} \right)^2\right)  \right)^{1/2}\leq \beta<1.
 \]
 Hence,
 \begin{equation*}
  \|s^{P,(j)}_1\|_{L^2}\leq O(\epsilon)+\beta \|s^{(j)}_2\|^2_{L^2},
 \end{equation*}
 and
  \[
  \|s^{P,(j)}_2\|_{L^2}\leq O(\epsilon)+\beta \|s^{(j+1)}_1\|^2_{L^2}.
 \]
 By the conditions just listed immediately before Theorem \ref{thm:conv}, and the definition in Equation \eqref{eqn:regthm1}, we have 
  \begin{equation}\label{eqn:sest1}
  \|s^{E,(j)}_1\|_{L^2}\leq O(\epsilon)+  \|s^{P,(j)}_1\|_{L^2} \leq O(\epsilon)+\beta \|s^{(j)}_2\|^2_{L^2},
 \end{equation}
 and
  \[
  \|s^{E,(j)}_2\|_{L^2}\leq O(\epsilon)+  \|s^{P,(j)}_2\|_{L^2}\leq O(\epsilon)+\beta \|s^{(j+1)}_1\|^2_{L^2}.
 \]
Note that $s^{E,(j)}_1=s^{(j+1)}_1$,  it holds that
   \begin{equation}\label{eqn:sest2}
  \|s^{E,(j)}_2\|_{L^2}\leq O(\epsilon)+\beta \|s^{(j+1)}_1\|^2_{L^2}=O(\epsilon)+\beta \|s^{E,(j)}_1\|^2_{L^2}\leq O(\epsilon)+\beta^2 \|s^{(j)}_2\|^2_{L^2}.
 \end{equation}
 By Equation \eqref{eqn:sest1}, \eqref{eqn:sest2}, and the mathematical induction similar to that for Theorem $3.5$ in \cite{HZYregression}, it is easy to show that
 \[
  \|s^{E,(j)}_k\|_{L^2}\leq O(c_0 \epsilon+\beta^{c(j,k)} )
 \]
 where $c(j,k)= 2j+k$ and $c_0=\frac{1}{1-\beta}$ coming from the geometric sequence due to the summation of the $O(\epsilon)$ term for all $j\geq 0$. 
 
To care the general case, we need to extend the argument to $K>2$ and non-constant $\alpha_k$. We shall do this in two steps: first $K>2$ but $\alpha_k\equiv 1$ for all $k$, and then, finally, $K>2$ and varying $\alpha_k$. Rather than repeating the earlier argument in full detail, adapted to these more general situations, we indicate simply, for both steps, what extra estimates need to be taken into account. This may not give the sharpest estimate, but this is not a concern for now.
 
 Next, we prove the case when $K>2$ and $\alpha_k(t)=1$ for all $t$ and $k$. Similarly, by the definition of the partition-based regression and the triangle inequality, we have
 \begin{eqnarray*}
  |s^{P,(j)}_k(2\pi x)|&\leq & O(K\epsilon)+ \sum_{i=1}^{k-1} \frac{\sum_{n=0}^{N^h-1}s^{(j+1)}_i(2\pi t^h_n) \left(D^{ki}_h(m,n)-\gamma\right)}{D^k_h(m)}\\
 & & + \sum_{i=k+1}^K  \frac{\sum_{n=0}^{N^h-1}s^{(j)}_i(2\pi t^h_n) \left(D^{ki}_h(m,n)-\gamma\right)}{D^k_h(m)}.
 \end{eqnarray*}
 Hence, by the triangle inequality and the H{\"o}lder inequality again, it holds that
  \begin{eqnarray*}
 \|s^{P,(j)}_k\|_{L^2}&\leq & O(K\epsilon) +  \sum_{i=1}^{k-1}  \left( \sum_{m=0}^{N^h-1} \left( \frac{\sum_{n=0}^{N^h-1}s^{(j+1)}_i(2\pi t^h_n) \left(D^{ki}_h(m,n)-\gamma\right)}{D^k_h(m)} \right)^2 h \right)^{1/2}\\
&& +  \sum_{i=k+1}^K  \left( \sum_{m=0}^{N^h-1} \left( \frac{\sum_{n=0}^{N^h-1}s^{(j)}_i(2\pi t^h_n) \left(D^{ki}_h(m,n)-\gamma\right)}{D^k_h(m)} \right)^2 h \right)^{1/2}\\
 &\leq &O(K\epsilon) +  \sum_{i=1}^{k-1} \left( \sum_{n=0}^{N^h-1}\left(s^{(j+1)}_i(2\pi t^h_n) \right)^2h \right) ^{1/2}\left( \sum_{m=0}^{N^h-1}\left(\sum_{n=0}^{N^h-1}\left(  \frac{ D^{ki}_h(m,n)-\gamma}{D^k_h(m)} \right)^2\right)  \right)^{1/2}\\
 &&+  \sum_{i=k+1}^K  \left( \sum_{n=0}^{N^h-1}\left(s^{(j)}_i(2\pi t^h_n) \right)^2h \right) ^{1/2}\left( \sum_{m=0}^{N^h-1}\left(\sum_{n=0}^{N^h-1}\left(  \frac{ D^{ki}_h(m,n)-\gamma}{D^k_h(m)} \right)^2\right)  \right)^{1/2}\\
  &\leq &O(\epsilon) + \sum_{i=1}^{k-1}\beta \|s^{(j+1)}_i\|_{L^2} + \sum_{i= k+1}^K\beta \|s^{(j)}_i\|_{L^2}.
 \end{eqnarray*}
Since $s^{E,(j)}_k-s^{P,(j)}_k=O(\epsilon)$, we know
  \begin{equation*}
 \|s^{E,(j)}_k\|_{L^2}\leq O(\epsilon) + \sum_{i=1}^{k-1}\beta \|s^{(j+1)}_i\|_{L^2} + \sum_{i= k+1}^K\beta \|s^{(j)}_i\|_{L^2}.
 \end{equation*}
 Similar to the case of two components, by the equation just above and mathematical induction, one can show that
 \[
  \|s^{E,(j)}_k\|_{L^2}\leq O(c_0 \epsilon+\beta^{c(j,k)} )
 \]
 where $c_0=\frac{1}{1-(K-1)\beta}$ coming from the geometric sequence due to the summation of the $O(\epsilon)$ term for all $j\geq 0$, 
 \[
 c(j,k)=\left\lceil \frac{jK+k}{K-1}\right\rceil,
 \]
 and $\lceil\cdot\rceil$ is the ceiling operator.

 Finally, we prove the case when amplitude functions are smooth functions but not a constant $1$. If the instantaneous frequencies are sufficiently large, depending on $\epsilon$, $M$, $K$, and $C$, amplitude functions are nearly constant up to an approximation error of order $\epsilon$. The time domain $[0,1]$ is divided into sufficiently small intervals such that amplitude functions are nearly constant inside each interval. Accordingly, the samples $(x_\ell,y_\ell)=\tau(v_\ell,s_k(2\pi v_\ell)+\kappa_k(v_\ell))$ for $\ell =1,\dots,L$ of the random vector $(X_k,Y_k^{(j)})$ is divided into groups and the partition-based regression method is applied to estimate the regression function for each group. This is similar to data splitting in nonparametric regression. The bound of $ |s^{P,(j)}_k(x)|$ is a weighted average of the bound given by each group, and the weight comes the number of points in each group over the total number of samples. Note that $\|\alpha_k\|_{L^\infty}\leq M$. By repeating the analysis above, it is simple to show 
 \[
\|s^{E,(j)}_k\|_{L^2}\leq O(\epsilon)+ M^2\sum_{i=1}^{k-1}\beta \|s^{(j+1)}_i\|_{L^2} + M^2\sum_{i= k+1}^K\beta \|s^{(j)}_i\|_{L^2}.
\]
where $M^2$ comes from 
\[
 \frac{\alpha_i\circ p_k^{-1}(v)}{\alpha_k\circ p_k^{-1}(v)}
\]
in $\kappa_k(v)$ after warping. By mathematical induction similar to that for Theorem $3.5$ in \cite{HZYregression} again, it holds that
 \begin{equation}\label{eqn:sf}
  \|s^{E,(j)}_k\|_{L^2}\leq O(c_0 \epsilon+(M^2(K-1)\beta)^{c(j,k)} )
 \end{equation}
 where $c_0=\frac{1}{1-M^2(K-1)\beta}$ coming from the geometric sequence due to the summation of the $O(\epsilon)$ term for all $j\geq 0$, and
 \[
 c(j,k)=\left\lceil \frac{jK+k}{K-1}\right\rceil.
 \]
 This finishes the proof of the first part of Theorem \ref{thm:conv}.
 
 By the definition of $r_k^{(j)}$ in Equation \eqref{eqn:r} and the inequality in Equation \eqref{eqn:sf}, it holds that
 \[
\|r_k^{(j+1)}\|_{L^2}\leq O(c_0\epsilon+(M^2(K-1)\beta)^{c(j,1)}),
\]
which completes the proof of Theorem \ref{thm:conv}.
\end{proof}

Theorem \ref{thm:conv} shows that the regression function in each step of Algorithm \ref{alg:RDBR} decays, if $M^2(K-1)\beta<1$, in the $L^2$ sense up to a fixed accuracy parameter as the iteration number becomes large. Hence, the recovered shape function converges and the residual decays up to a fixed accuracy parameter, if $M^2(K-1)\beta<1$. When the iteration number is sufficiently large, the accuracy of the RDBR in Theorem \ref{thm:conv} is as good as a single step of regression in Theorem \ref{thm:reg}. Compared to the convergence theorem of the Jacobi RDBR in \cite{HZYregression}, both the Jacobi style and the Gauss-Seidel style recursive scheme have linear convergence, but the Gauss-Seidel style has a smaller rate of convergence; especially in the case of two components, the rate of convergence of the Gauss-Seidel style is $\beta^2$ while the one of the Jacobi style is $\beta$.

\begin{theorem}
\label{thm:ns} (Noise robustness of the Gauss-Seidel RDBR) 
Let $f_k(t)=\alpha_k(t)s_k(2\pi N_k \phi_k(t))$, $k=1,\dots,K$, be $K$ GIMF's and $f(t)=\sum_{k=1}^Kf_k(t)+n(t)$, where $n(t)$ is a random noise with a bounded variance $\sigma^2$. Under the other conditions introduced in Theorem \ref{thm:conv}, for the given $\epsilon$,  $\exists L_0(\epsilon,M,K,C,h,\sigma)$, if $L>L_0$, then 
 \[
  \|s^{E,(j)}_k\|_{L^2}\leq O(c_0 \epsilon+( M^2(K-1)\beta)^{c(j,k)} )
 \]
and
\[
\|r_k^{(j+1)}\|_{L^2}\leq O(c_0\epsilon+\left(M^2(K-1)\beta\right)^{c(j,1)}),
\]
for all $j\geq 0$ and $1\leq k\leq K$, where $c(j,k)=\left\lceil \frac{jK+k}{K-1}\right\rceil$, $c_0=\frac{1}{1-M^2(K-1)\beta}$ is a constant number, $s^{E,(j)}_k$  is defined in Equation \eqref{eqn:regthm1} and $r_k^{(j)}$ is defined in Equation \eqref{eqn:r}.
\end{theorem}

Theorem \ref{thm:ns} is an immediate result of Theorem \ref{thm:reg} and \ref{thm:conv}. It shows that as soon as the number of sampling points $L$ is large enough, the noise effect will be negligible.

Next, we discuss the case when the estimated number of components is inexact, i.e., $\bar{K}\neq K$ in Algorithm \ref{alg:RDBR}. Two main situations are concerned here: (1) $\bar{K}>K$ and $\{p_k\}_{1\leq k\leq \bar{K}}$ consists of all the fundamental phase functions and their multiples; (2) $\bar{K}<K$ and $\{p_k\}_{1\leq k\leq \bar{K}}$ are fundamental phase functions. The convergence analysis of other situations can be generalized from these two situations.

In the first situation when $\bar{K}>K$, $\{p_k'(t)=n_k N_{\tau_k} \phi'_{\tau_k}(t)\}_{1\leq k\leq \bar{K}}$ naturally belong to $K$ groups $\{\mathcal{G}_k\}_{1\leq k\leq K}$, each of which corresponds to the multiple of a fundamental instantaneous frequency $p'_k(t)$. For each $p_k$, the RDBR tries to identify a shape function and the result depends on the order of $p_k$ in the set $\{p_k\}$. To make sure that the shape functions corresponding to fundamental phases are approximately the ground truth shape functions, there is a sorting procedure in Line $4$ and $5$ of Algorithm \ref{alg:RDBR}. By the similar analysis in Theorem \ref{thm:conv}, one can show that the summation of the shape function estimations corresponding to the phase functions within each group $\mathcal{G}_k$ converges to the $k$th ground truth shape function $s_k$; i.e., let 
\begin{equation}
\label{eqn:cnv2}
\tilde{s}_k^{(j)}=\sum_{\tau\in\mathcal{G}_k} \bar{s}_\tau^{(j)},
\end{equation}
then $\lim_{j\rightarrow \infty} \tilde{s}_k^{(j)}=O(\epsilon)+s^k$. In particular, Theorem \ref{thm:conv2} characterizes its convergence rate as follows.

\begin{theorem}
\label{thm:conv2} (Convergence of the Gauss-Seidel RDBR when $\bar{K}>K$) Suppose  $\bar{K}>K$ and $\{p_k\}_{1\leq k\leq \bar{K}}$ in Algorithm \ref{alg:RDBR} contains all the fundamental phase functions. Under the conditions listed in the paragraph immediately preceding Theorem \ref{thm:conv}, we have  
 \[
  \| \tilde{s}_k^{(j)}-s_k\|_{L^2}\leq O(c_0 \epsilon+( M^2(K-1)\beta)^{c(j,k)} )
 \]
and
\[
\|\tilde{r}_k^{(j+1)}\|_{L^2}\leq O(c_0\epsilon+\left(M^2(K-1)\beta\right)^{c(j,1)})
\]
for all $j\geq 0$ and $1\leq k\leq K$, where $c(j,k)=\left\lceil \frac{jK+k}{K-1}\right\rceil$, $c_0=\frac{1}{1-M^2(K-1)\beta}$ is a constant number, $\tilde{s}_k^{(j)}$  is defined in Equation \eqref{eqn:cnv2} and $\tilde{r}_k^{(j)}$ is defined as
\[
\tilde{r}_k^{(j)} = r_{\tau_k}^{(j)},
\]
where $\tau_k\in\mathcal{G}_k$ satisfies $n_{\tau_k}\geq n_\nu$ for all $\nu\in\mathcal{G}_k$, and the notation $n_v$ is introduced in Equation \eqref{eqn:prior3}.
\end{theorem}

\begin{proof}
The proof of Theorem \ref{thm:conv2} can be generalized from Theorem \ref{thm:conv}. The phase functions $\{p_\nu\}_{\nu\in\mathcal{G}_k}$ compete with each other to obtain a larger $L^2$-norm in their own shape function estimation $\bar{s}_\nu^{(j)}$. This competition does not have negative effects on the convergence of Algorithm \ref{alg:RDBR}; in fact, the more members in $\mathcal{G}_k$, the faster convergence of the shape function estimation, because the estimation error $\| \tilde{s}_k^{(j)}-s_k\|_{L^2}$ can be reduced by the extra regressions due to more than one member in $\mathcal{G}_k$.
\end{proof}

It is worth pointing out that the Gauss-Seidel RDBR converges in the sense of Theorem \ref{thm:conv2} when $\bar{K}>K$; however, the Jacobi RDBR diverges when $\bar{K}>K$. 

In the second situation when $\bar{K}<K$, it is assumed that $\{p_k\}_{1\leq k\leq \bar{K}}$ are all fundamental phase functions. Since not all the fundamental phase functions are used in Algorithm \ref{alg:RDBR}, the components corresponding to the missing fundamental phase functions always remain in the residual signal $r_k^{(j)}$. Hence, the shape function estimation corresponding to the known phase functions cannot be improved by recursive regression, and Algorithm \ref{alg:RDBR} stops iteration quickly since the stopping criteria $|\epsilon_1-\epsilon_0|\leq \epsilon$ is soon satisfied. After Algorithm \ref{alg:RDBR} stops, by the arguments in Theorem \ref{thm:conv}, it is easy to see that the estimated shape function $\bar{s}_k$ satisfies
\[
\|\bar{s}_k-s_k\|_{L^2}\leq O(\epsilon+M^2\beta \sum_{\ell\neq k} \|s_\ell\|_{L^2})
\]
for $1\leq k\leq \bar{K}$.

%
%
%
%
        
\section{RDBR for the multiresolution mode decomposition}
\label{sec:MMD}

\subsection{Algorithm description}

In what follows, we modify the Gauss-Seidel RDBR in Algorithm \ref{alg:RDBR} to solve the multiresolution mode decomposition. The multiresolution mode decomposition problem aims at extracting each multiresolution intrinsic mode function $f_k(t)$ from their superposition $f(t)=\sum_{k=1}^K f_k(t)$, estimating their corresponding multiresolution expansion coefficients and the shape function series defined as follows.

\begin{definition}
  \label{def:m_IMF}
  A function 
  \begin{equation}
\label{eqn:m_IMF2}
f(t) = \sum_{n=-N/2}^{N/2-1} a_n\cos(2\pi n\phi(t))s_{cn}(2\pi N\phi(t))+\sum_{n=-N/2}^{N/2-1}b_n \sin(2\pi n\phi(t))s_{sn}(2\pi N\phi(t))
\end{equation}
is a \emph{multiresolution intrinsic mode function} (MIMF) of type $(M_0,M,N,\epsilon)$ defined on $[0,1]$, if the conditions below are satisfied:
\begin{itemize}
\item the shape function series $\{s_{cn}(t)\}$ and $\{s_{sn}(t)\}$ are in ${\cal S}_M$;
\item 
the multiresolution expansion coefficients $\{a_{n}\}$ and $\{b_{n}\}$ satisfy
\begin{align*}
    \sum_{n=-N/2}^{N/2-1}|a_n|\leq M, \quad  \quad  \sum_{n=-N/2}^{N/2-1}|a_n|-\sum_{n=-M_0}^{M_0-1}|a_n|\leq \epsilon, \\
  \sum_{n=-N/2}^{N/2-1}|b_n|\leq M,  \quad  \quad \sum_{n=-N/2}^{N/2-1}|b_n|-\sum_{n=-M_0}^{M_0-1}|b_n|\leq \epsilon;
\end{align*}
\item $\phi(t)$ satisfies
\begin{align*}
    \phi(t)\in C^\infty,  \quad  1/M \leq | \phi'|\leq M.
\end{align*}
\end{itemize}
\end{definition}
A MIMF is a generalization of the GIMF in Equation \eqref{P2} for time-dependent shape functions to describe the nonlinear and non-stationary time series adaptively, as we shall see in numerical examples in Section \ref{sec:NumEx}. A single MIMF itself is a superposition of GIMF's that share the same phase function; identifying its multiresolution expansion coefficients and shape function series requires separating these GIMF's.

The only difference to distinguish different GIMF's in a MIMF is the frequency of the oscillation in $\cos(2\pi n\phi(t))$ and $\sin(2\pi n\phi(t))$ (see Equation \eqref{eqn:m_IMF2}). Hence, if $\phi(t)$ and $N$ are known exactly, Fourier analysis in the coordinate of $\phi(t)$ (instead of $t$) can estimate the multiresolution expansion coefficients $\{a_n\}$, $\{b_n\}$, and the shape function series $\{s_{cn}\}$ and $\{s_{sn}\}$. However, in practice $\phi(t)$ and $N$ are only known approximately and the estimation error in $\phi(t)$ can be amplified by a factor $O(Nm)$, where $m$ is the frequency bandwidth of shape functions, if the Fourier analysis in the $\phi(t)$ coordinate is applied. This instability makes the Fourier approach less attractive in analyzing a MIMF.

An immediate question is whether the diffeomorphism-based regression idea in Section \ref{sec:MMD} can estimate the GIMF's in a MIMF. The main concerns are: when regression is applied to estimate one GIMF, other GIMF's acting as additive perturbation seem to cause a large estimation error; when $n\neq 0$, $\cos(2\pi n\phi(t))$ and $\sin(2\pi n\phi(t))$ as the amplitudes of GIMF's are occasionally zero, making it numerically infeasible to apply the diffeomorphism-based regression directly. Fortunately, for the first concern, the oscillatory functions $\cos(2\pi n\phi(t))$ and $\sin(2\pi n\phi(t))$ have zero mean in the coordinate of $\phi(t)$, exactly cancelling out the noise perturbation. As for the second concern, by the formulas of trigonometric functions, 
  \begin{eqnarray*}
\cos(2\pi m\phi(t))f(t) &=& \sum_{n=-N/2}^{N/2-1} \frac{a_n}{2}\left(\cos(2\pi (m+n)\phi(t)) + \cos(2\pi (m-n)\phi(t)) \right)s_{cn}(2\pi N\phi(t))\\
& &+\sum_{n=-N/2}^{N/2-1}\frac{b_n}{2}\left( \sin(2\pi (n+m)\phi(t))+\sin(2\pi (n-m)\phi(t))\right)s_{sn}(2\pi N\phi(t)),
\end{eqnarray*}
where there is only one term with a non-zero-mean amplitude, $\frac{a_m}{2}s_{cm}(2\pi N\phi(t))$, implying that the diffeomorphism-based regression could be able to estimate $a_m$ and $s_{cm}$ from $\cos(2\pi m\phi(t))f(t)$. Similarly, we can estimate $b_m$ and $s_{sm}$ form $\sin(2\pi m\phi(t))f(t)$. Hence, by applying the diffeomorphism-based regression to $\cos(2\pi n\phi(t))f(t)$ and $\sin(2\pi n\phi(t))f(t)$ for $n=-N/2,\dots,N/2-1$, we can estimate all the multiresolution expansion coefficients $\{a_n\}$, $\{b_n\}$, and the shape function series $\{s_{cn}\}$ and $\{s_{sn}\}$.

However, in the presence of a superposition of several MIMFs, a sequence of diffeomorphism-based analysis discussed just above cannot estimate all the multiresolution expansion coefficients and the shape function series. The reason for this inaccuracy is similar to the motivation of the RDBR for GMD in Section \ref{sec:GMD}. Hence, we expect that a modified version of the RDBR can solve the multiresolution mode decomposition problem under the condition of well-differentiated phase functions defined as follows.
  
\begin{definition}
  \label{def:wd}
  Suppose 
  \begin{eqnarray*}
f_k(t)=\sum_{n=-N_k/2}^{N_k/2-1} a_{n,k}\cos(2\pi n\phi_k(t))s_{cn,k}(2\pi N_k\phi_k(t))+\sum_{n=-N_k/2}^{N_k/2-1}b_{n,k} \sin(2\pi n\phi_k(t))s_{sn,k}(2\pi N_k\phi_k(t)).
\end{eqnarray*}
is a MIMF of type $(M_0,M,N_k,\epsilon)$ for $t\in[0,1]$, $k=1,\dots,K$, and \[\{p_k(t)=N_k\phi_k(t)\}_{1\leq k\leq K}\subset\WD(M,N,K,h,\beta,\gamma),\] then  $f(t)=\sum_{k=1}^K f_k(t)$ is said to be a well-differentiated superposition of MIMFs of type $(M_0,M,N,K,h,\beta,\gamma,\epsilon)$. Denote the set of all these functions $f(t)$ as $\WS(M_0,M,N,K,h,\beta,\gamma,\epsilon)$.
\end{definition}

Recall that $\mathcal{M}_\ell$ is the operator for computing the \emph{$\ell$-banded multiresolution approximation} to a MIMF $f(t)$ in Equation \eqref{eqn:m_IMF2}, i.e.,
\begin{equation}
\label{eqn:mmm_IMF}
\mathcal{M}_\ell(f)(t) = \sum_{n=-\ell}^{\ell} a_n\cos(2\pi n\phi(t))s_{cn}(2\pi N\phi(t))+\sum_{n=-\ell}^{\ell}b_n \sin(2\pi n\phi(t))s_{sn}(2\pi N\phi(t)).
\end{equation}  
Ideally, we hope to apply the modified RDBR algorithm in Algorithm \ref{alg:RDBR2} to $f(t) -\mathcal{M}_{n-1}(f)(t)$ for computing the multiresolution expansion coefficients $\{a_{n,k}\}_{1\leq k\leq K}$, $\{b_{n,k}\}_{1\leq k\leq K}$, and the shape functions $\{s_{cn,k}\}_{1\leq k\leq K}$ and $\{s_{sn,k}\}_{1\leq k\leq K}$, where $\mathcal{M}_{n-1}(f)(t)$ is available from previous computation. However, the modified RDBR can only return estimations approximately and $\mathcal{M}_{n-1}(f)(t)$ is only available approximately. This motivates us to repeatedly apply the same idea to refine the estimations. In summary, Algorithm \ref{alg:MMD} below identifies $\mathcal{M}_{M_0}(f_k)(t)$, its multiresolution expansion coefficients $\{a_{n,k}\}$ and $\{b_{n,k}\}$, and its shape function series $\{s_{cn,k}\}$ and $\{s_{sn,k}\}$ from the superposition $f(t)\in \WS(M_0,M,N,K,h,\beta,\gamma,\epsilon)$ based on repeating the modified RDBR in Algorithm \ref{alg:RDBR2}. In the pseudo-code in Algorithm \ref{alg:MMD}, the input and output of Algorithm \ref{alg:RDBR2} is denoted as 
\[
[\{\bar{s}_k\}_{1\leq \bar{K}},\{\bar{f}_k\}_{1\leq \bar{K}}]=RDBR(f,\{p_k\}_{1\leq \bar{K}},n,tp,\epsilon,J).
\]

{The MMD algorithm in Algorithm \ref{alg:MMD} is essentially $J_1J_2M_0$ iterations of recursive regression, while the GMD solved by Algorithm \ref{alg:RDBR} requires $J_1$ iterations when we set $J=J_1$. The main computational cost in each iteration in both algorithms is the same and depends on the cost of regression. Hence, the computational cost of the MMD over the one of GMD is always $O(J_2 M_0)$ independent of data and computational environment. In the current version of code, both GMD and MMD are implemented with a knot-free spline regression in \cite{Regression}, which might not be efficient enough. In a parallel paper concerning the efficient numerical implementation of the MMD (including GMD) algorithm, we have proposed fast algorithms to solve the MMD problem \cite{fMMD}.
}

\vspace{3in}
\begin{algorithm2e}[H]
\label{alg:RDBR2}
\caption{A modified RDBR for the multiresolution mode decomposition.}
Input: $L$ points of i.i.d. measurement $\{f(t_\ell)\}_{\ell=1,\dots,L}$ with $t_\ell \in [0,1]$, estimated instantaneous phases $\{p_k\}_{k=1,\dots,\bar{K}}$, an accuracy parameter $\epsilon<1$, the maximum iteration number $J$, frequency $n$, the type of amplitude $tp$.

Output: the estimated shape functions $\{\bar{s}_k\}_{k=1,\dots,\bar{K}}$, and the estimated modes $\{\bar{f}_k(t)\}_{k=1,\dots,\bar{K}}$ at the sampling grid points $\{t_\ell\}_{1\leq \ell\leq L}$. 

Initialize: let $r_1^{(0)}=f$, $\epsilon_1=\epsilon_2=1$, $\epsilon_0=2$, the iteration number $j=0$, $\dot{s}^{(0)}_k=0$,  $\bar{s}_k^{(0)}=0$, and $\bar{f}_k=0$  for all $k=1,\dots,\bar{K}$.

Compute $N_k$ as the integer nearest to the average of $p'_k(t)$ for $k=1,\dots,\bar{K}$.

Sort $\{N_k\}_{1\leq k\leq \bar{K}}$ in an ascending order and reorder the amplitude and phase functions accordingly.

\While{$j<J$, $\epsilon_1> \epsilon$, $\epsilon_2>\epsilon$, and $|\epsilon_1-\epsilon_0|>\epsilon$}{ 

\For{$k=1,\dots,\bar{K}$}{
\If {$tp=1$}{
Let $g(t) =  \cos(2\pi \frac{n}{N_k} p_k(t))$ and evaluate $h^{(j)}_k(t)=g(p_k^{-1}(t))r_k^{(j)}(p_k^{-1}(t))$.
}
\Else{
Let $g(t) =  \sin(2\pi \frac{n}{N_k} p_k(t))$ and evaluate $h^{(j)}_k(t)= g(p_k^{-1}(t))r_k^{(j)}(p_k^{-1}(t))$.
}

Repeat Line $9$ to $11$ in Algorithm \ref{alg:RDBR} to compute $\dot{s}_k^{(j+1)}$.

\If {$m=0$}{
Let $\dot{f}^{(j+1)}_k(t)= \dot{s}^{(j+1)}_k(2 \pi p_k(t))$.
}
\Else{
Let $\dot{f}^{(j+1)}_k(t)= 2g(t)\dot{s}^{(j+1)}_k(2 \pi p_k(t))$ and $\dot{s}_k^{(j+1)}=2\dot{s}_k^{(j+1)}$.
}

Let $\bar{s}_k^{(j+1)}=\bar{s}_k^{(j)}+\dot{s}_k^{(j+1)}$.

Update $\bar{f}_k(t)\leftarrow \bar{f}_k(t) + \dot{f}^{(j+1)}_k(t)$.

\If{$k<\bar{K}$}{
Let $r_{k+1}^{(j)}=r_k^{(j)}- \dot{f}^{(j+1)}_k(t)$.
}
\Else{
Let $r_{1}^{(j+1)}=r_k^{(j)}- \dot{f}^{(j+1)}_k(t)$.
}

    }

Update $\epsilon_0=\epsilon_1$, $\epsilon_1=\|r_1^{(j+1)}\|_{L^2}$, $\epsilon_2=\max_{k}\{\|\dot{s}_k^{(j+1)}\|_{L^2}\}$.

Set $j = j + 1$. 
}
\end{algorithm2e}

\vspace{2in}

\begin{algorithm2e}[H]
\label{alg:MMD}
\caption{Multiresolution mode decomposition.}
Input: $L$ points of i.i.d. measurement $\{f(t_\ell)\}_{\ell=1,\dots,L}$ with $t_\ell \in [0,1]$, estimated instantaneous phases $\{p_k\}_{k=1,\dots,\bar{K}}$, accuracy parameters $\epsilon_1$ and $\epsilon_2$, the maximum iteration numbers $J_1$ and $J_2$, and the band-width parameter $M_0$.

Output: $\mathcal{M}_{M_0}(f_k)(t)$ at the sampling grid points $\{t_\ell\}_{1\leq \ell\leq L}$, its multiresolution expansion coefficients $\{a_{n,k}\}_{n=-M_0,\dots,M_0}$ and $\{b_{n,k}\}_{n=-M_0,\dots,M_0}$, and its shape function series $\{s_{cn,k}\}_{n=-M_0,\dots,M_0}$ and $\{s_{sn,k}\}_{n=-M_0,\dots,M_0}$ for $1\leq k\leq \bar{K}$. 

Initialize: let $a_{n,k}=0$, $b_{n,k}=0$, $s_{cn,k}=0$, $s_{sn,k}=0$, $\mathcal{M}_{M_0}(f_k)=0$ for all $k$ and $n$; let $c=\|f\|_{L^2}$; let $e=1$; let $r^{(0)}=f$.

\For{$j=1,2,\dots,J_1,$}{

\For{$n=0,1,-1,\dots,M_0,-M_0$}{

$[\{\bar{s}_k\}_{1\leq \bar{K}},\{\bar{f}_k\}_{1\leq \bar{K}}]=RDBR(r^{(j-1)},\{p_k\}_{1\leq \bar{K}},n,1,\epsilon_2,J_2)$.

\For{$k=1,\dots,\bar{K}$}{
$s_{cn,k}\leftarrow s_{cn,k}+\bar{s}_k$.

Update $\mathcal{M}_{M_0}(f_k)(t)\leftarrow \mathcal{M}_{M_0}(f_k)(t) + \bar{f}_k$.

Compute $r^{(j)}= r^{(j-1)}-\bar{f}_k$.
}

\If {$|n|>0$}
{

$[\{\bar{s}_k\}_{1\leq \bar{K}},\{\bar{f}_k\}_{1\leq \bar{K}}]=RDBR(r^{(j)},\{p_k\}_{1\leq \bar{K}},n,0,\epsilon_2,J_2)$.

\For{$k=1,\dots,\bar{K}$}{

$s_{sn,k}\leftarrow s_{sn,k}+\bar{s}_k$.

Update $\mathcal{M}_{M_0}(f_k)(t)\leftarrow \mathcal{M}_{M_0}(f_k)(t) + \bar{f}_k$.

Update $r^{(j)}\leftarrow r^{(j)}-\bar{f}_k$.
}
}
%
%
}

If $\|r^{(j)}\|_{L^2}/c\leq\epsilon_1$, then break the for loop.

\If{$\|r^{(j)}\|_{L^2}/c\geq e-\epsilon_1$}{
Break the for loop.
}
\Else{
$e=\|r^{(j)}\|_{L^2}/c$.
}
}

Let $a_{n,k}=\|s_{cn,k}\|_{L^2}$ and $s_{cn,k}=s_{cn,k}/a_{n,k}$ for all $k$ and $n$.

Let $b_{n,k}=\|s_{sn,k}\|_{L^2}$ and $s_{sn,k}=s_{sn,k}/b_{n,k}$ for all $k$ and $n$.
\end{algorithm2e}

\vspace{3in}

\subsection{Convergence analysis}
\label{sec:thmRDBR2}

In this section, an asymptotic analysis on the convergence of the multiresolution mode decomposition in Algorithm \ref{alg:MMD} is introduced. We assume that the number of MIMFs is known exactly, i.e., $\bar{K}=K$; we also focus on the case when $J_2=1$ in Algorithm \ref{alg:MMD}; the analysis for other cases can be directly generalized. Similar to the theory in Section \ref{sec:GMD}, the main analysis is to prove that the for-loop in Line $5$ in Algorithm \ref{alg:MMD} is able to approximately identify the multiresolution expansion coefficients $a_{n,k}$ and $b_{n,k}$, and estimate the shape function series $s_{cn,k}$ and $s_{sn,k}$ for $1\leq k\leq K$ and $-M_0\leq n\leq M_0$. Though the Gauss-Seidel iteration is used in Algorithm \ref{alg:MMD} for faster convergence, we will show a looser  bound of the convergence by using the Jacobi style iteration for the purpose of simplicity.

In the $(j+1)$th iteration in Line $4$, the residual function is 
  \begin{eqnarray}
  \label{eqn:res1}
r^{(j)}(t)&=&\sum_{k=1}^K  \sum_{n=-M_0}^{M_0} \cos(2\pi n\phi_k(t))s_{cn,k}^{(j)}(2\pi N_k\phi_k(t))\nonumber\\
&&+\sum_{k=1}^K \sum_{n=-M_0}^{M_0} \sin(2\pi n\phi_k(t))s_{sn,k}^{(j)}(2\pi N_k\phi_k(t))+O(\epsilon),
\end{eqnarray}
where the superscript $^{(j)}$ indicates the new target shape function series in the $(j+1)$th iteration of Line $4$ in Algorithm \ref{alg:MMD}, and the multiresolution expansion coefficients have been absorbed in the shape functions. Since we adopt the Jacobi style iteration in the analysis, the residual function in \eqref{eqn:res1} remains the same throughout the $j$th iteration, i.e. Line $9$, $10$, $15$, and $16$ were postponed until the end of the for-loop in Line $5$. For the $k$th component and frequency $n$, the diffeomorphism-based analysis with a constant amplitude function and the phase function $N_k\phi_k(t)$ is applied to $2\cos(2\pi n\phi_k(t))r^{(j)}(t)$ and $2\sin(2\pi n\phi_k(t))r^{(j)}(t)$ to estimate the target shape functions $s_{cn,k}^{(j)}$ and $s_{sn,k}^{(j)}$, respectively. 

In particular, 
\begin{eqnarray}
\label{eqn:r3}
&&2\cos(2\pi n\phi_k(t))r^{(j)}(t)\nonumber\\
&=& \sum_{k=1}^K  \sum_{m=-M_0}^{M_0} 2 \cos(2\pi n\phi_k(t))\cos(2\pi m\phi_k(t))s_{cm,k}^{(j)}(2\pi N_k\phi_k(t))\nonumber\\
&&+\sum_{k=1}^K \sum_{m=-M_0}^{M_0}2 \cos(2\pi n\phi_k(t)) \sin(2\pi m\phi_k(t))s_{sm,k}^{(j)}(2\pi N_k\phi_k(t))+O(\epsilon),
\end{eqnarray}
where the term $2\cos(2\pi n\phi_k(t))\cos(2\pi n\phi_k(t))s_{cn,k}^{(j)}(2\pi N_k\phi_k(t))$ results in
\[
s_{cn,k}^{(j)}(2 \pi N_k\phi_k(t))+\cos(4\pi n\phi_k(t))s_{cn,k}^{(j)}(2 \pi N_k\phi_k(t)).
\]
$s_{cn,k}^{(j)}(2 \pi N_k\phi_k(t))$ is the only term with a non-zero-mean amplitude function in \ref{eqn:r3}. Following the same notations as in Section \ref{sec:thmRDBR}, let
\begin{eqnarray*}
h_{cn,k}^{(j)}(v)& =& 2\left( \cos(2\pi n\phi_k(t))r^{(j)}\right)\circ p_k^{-1}(v)\\
&= & s_{cn,k}^{(j)}(2\pi v)+ \kappa_{cn,k}^{(j)}(2\pi v)+O(\epsilon),
\end{eqnarray*}
where $v = p_k(t)=N_k\phi_k(t)$ and $\kappa_{cn,k}^{(j)}$ comes from other terms in \eqref{eqn:r3}. After the folding map
\begin{eqnarray*}
\tau: \ \      \left(v, h_{cn,k}^{(j)}(v) \right)   \mapsto    \left(\text{mod}(v,1),  h_{cn,k}^{(j)}(v) \right),
\end{eqnarray*}
we have $(x_\ell,y_\ell)=\tau(v_\ell,s_{cn,k}^{(j)}(2\pi v_\ell)+\kappa_{cn,k}^{(j)}(2\pi v_\ell)+O(\epsilon))$ for $\ell =1,\dots,L$ as $L$ i.i.d. samples of a random vector $(X_{cn,k},Y_{cn,k}^{(j)})$, where $X_{cn,k}\in[0,1]$. Hence, the regression problem in \eqref{eqn:r3} can be written as
\begin{eqnarray}
s_{cn,k}^{R,(j+1)} &=& \underset{s:\mathbb{R}\rightarrow \mathbb{R}}{\arg\min}\quad \E\{\left| s(2\pi X_{n,k})-Y_{cn,k}^{(j)} \right|^2\}\nonumber\\
&=&  \underset{s:\mathbb{R}\rightarrow \mathbb{R}}{\arg\min}\quad \E\{\left| s(2\pi X_{n,k})-(Y_{cn,k}^{(j)} -s_{cn,k}^{(j)}(2\pi X_{cn,k}))\right|^2\}-s_{cn,k}^{(j)},
\label{eqn:reg3}
\label{eqn:regthm2}
\end{eqnarray}
we have 
\[
s_{cn,k}^{R,(j+1)} = s_{cn,k}^{(j)} + s^{E,(j)}_{cn,k},
\]
where
\begin{equation}
\label{eqn:regthm12}
s^{E,(j)}_{cn,k}(2\pi x)\defeq \E\{Y_{cn,k}^{(j)} - s_{cn,k}^{(j)}(2\pi X_{cn,k})|X_{cn,k}=x\}\neq 0
\end{equation}
due to the perturbation caused by $\kappa_{cn,k}^{(j)}$. In the next iteration, the target shape function $s^{(j+1)}_{cn,k}= - s^{E,(j)}_{cn,k}$. Hence, the key convergence analysis is to show that $s^{E,(j)}_{cn,k}$ decays as $j\rightarrow \infty$. 

Similarly, when the diffeomorphism-based regression is applied to $s_{sn,k}^{(j)}$, we are able to estimate $s_{sn,k}^{(j)}$ approximately and the estimation error $s^{E,(j)}_{sn,k}$ is the estimation target $s^{(j+1)}_{sn,k}= - s^{E,(j)}_{sn,k}$ in the next iteration.

In what follows, we assume that an accuracy parameter $\epsilon$ is fixed. Furthermore, suppose $f(t)\in\WS(M_0,M,N,K,h,\beta,\gamma,\epsilon)$, and all shape functions are in the space $\LL$. Under these conditions, all regression functions $s^{(j)}_{cn,k}\in\LL$, $s^{(j)}_{sn,k}\in\LL$, and have bounded $L^\infty$ norm depending only on $M$ and $K$. By Algorithm \ref{alg:RDBR2}, we have the nice and key conditions that $\int_0^1 s^{(j)}_{cn,k}(2\pi t)dt=0$ and $\int_0^1 s^{(j)}_{sn,k}(2\pi t)dt=0$ at each iteration for all $n$, $k$ and $j$. Note that $\Var(Y_{cn,k}^{(j)}|X_{cn,k}=x)$ and $\Var(Y_{sn,k}^{(j)}|X_{sn,k}=x)$ are bounded by a constant depending only on $M$ and $K$ as well. For the fixed $\epsilon$ and $C$, there exists $h_0(\epsilon,C)$ such that $C^2h^2<\epsilon^2$ if $0<h<h_0$. By the abuse of notation, $O(\epsilon)$ is used instead of $Ch$ later. By Theorem \ref{thm:reg}, for the fixed $\epsilon$, $M_0$, $M$, $K$, $C$, and $h$, there exists $L_0(\epsilon,M_0,M,K,C,h)$ such that the $L^2$ error of the partition-based regression is bounded by $\epsilon^2$. In what follows, $h$ is smaller than $h_0$, $L$ is larger than $L_0$, and hence all estimated regression functions approximate the ground truth regression function with an $L^2$ error of order $\epsilon$. Under these conditions and assumptions, $s^{E,(j)}_{cn,k}$ and $s^{E,(j)}_{sn,k}$ are shown to decay to $O(\epsilon)$ as $j\rightarrow \infty$, as long as $L$ and $N$ are sufficiently large, and the decay rate will be estimated.

\begin{theorem}
\label{thm:conv3} (Convergence of the multiresolution mode decomposition) Under the conditions listed in the paragraph immediately preceding this theorem, suppose $N_k$ and $\phi_k(t)$ are known for all $k$, and $J_2=1$ in Algorithm \ref{alg:MMD}, as long as $L$ and $N$ are sufficiently large, we have  
 \[
  \|s^{E,(j)}_{cn,k}\|_{L^2}\leq O(c_0 \epsilon+( \beta  (2M_0+1)(K-1))^j ),
 \]
 and 
  \[
  \|s^{E,(j)}_{sn,k}\|_{L^2}\leq O(c_0 \epsilon+(\beta  (2M_0+1)(K-1))^j )
 \]
for all $j\geq 0$ and $1\leq k\leq K$, where $c_0=\frac{1}{1-\beta  (2M_0+1)(K-1)}$ is a constant number, $s^{E,(j)}_{cn,k}$ is defined in Equation \eqref{eqn:regthm12} and $s^{E,(j)}_{sn,k}$ is defined similarly.
\end{theorem}

\begin{proof}
We only prove the case when $K=2$; the proof for a general case follows a similar discussion of the proof of Theorem \ref{thm:conv}.

Recall notations in Definition \ref{def:wd}. When the partition-based regression method is applied to solve the regression problem in \eqref{eqn:reg3} to obtain the approximate regression function $s^{P,(j)}_{cn,1}(2\pi x)$, we have
 \[
 s^{P,(j)}_{cn,1}(2\pi x)=\frac{\sum_{n=0}^{N^h-1}\left( \kappa^{(j)}_{cn,1}( 2\pi t^h_n)+O(\epsilon)\right)D^{12}_h(m,n)}{D^1_h(m)},
 \]
 for $x\in[t^h_m,t^h_m+h)$, where $O(\epsilon)$ comes from the approximation of $\LL$ functions using the values on grid points $t^h_n$, and the $O(\epsilon)$ term in the residual function in \eqref{eqn:res1}. The following argument is similar to that for Lemma $3.3$ in \cite{HZYregression}. It is easy to check that 
 \begin{equation}\label{eqn:inq1}
  |s^{P,(j)}_{cn,1}(2\pi x)|\leq O(\epsilon)+\frac{\sum_{n=0}^{N^h-1} \bar{\kappa}^{(j)}_{cn,1}( 2\pi t^h_n) \left(D^{12}_h(m,n)-\gamma\right)}{D^1_h(m)} + \gamma \frac{\sum_{n=0}^{N^h-1} \bar{\kappa}^{(j)}_{cn,1}(2\pi  t^h_n)}{D^1_h(m)},
 \end{equation}
 where
\[
 \bar{\kappa}^{(j)}_{cn,k}(2\pi v) :=\kappa^{(j)}_{cn,k}(2\pi v) -  \dot{\kappa}^{(j)}_{cn,k}(2\pi v),
 \]
 and
  \begin{eqnarray*}
 \dot{\kappa}^{(j)}_{cn,k}(2\pi v)&:=&\sum_{m=-M_0}^{M_0} 2 \cos(2\pi\frac{n}{N_k}v)\cos(2\pi \frac{m}{N_k}v)s_{cm,k}^{(j)}(2\pi v)\nonumber\\
&&+ \sum_{m=-M_0}^{M_0}2 \cos(2\pi \frac{n}{N_k}v) \sin(2\pi  \frac{m}{N_k}v)s_{sm,k}^{(j)}(2\pi v)-s_{cn,k}^{(j)}(2 \pi v)\\
&=& \sum_{m=-M_0}^{M_0} \left(\cos(2\pi \frac{m+n}{N_k}v) + \cos(2\pi \frac{n-m}{N_k}v) \right)s_{cn,k}^{(j)}(2\pi v)\\
& &+\sum_{m=-M_0}^{M_0}\left( \sin(2\pi \frac{n+m}{N_k}v)+\sin(2\pi \frac{m-n}{N_k}v)\right)s_{sn,k}^{(j)}(2\pi v)-s_{cn,k}^{(j)}(2 \pi v),
 \end{eqnarray*}
 since 
 \begin{equation}\label{eqn:sum}
 \frac{\sum_{n=0}^{N^h-1} \dot{\kappa}^{(j)}_{cn,1}(2\pi  t^h_n) \left(D^{12}_h(m,n)\right)}{D^1_h(m)}=O(\epsilon),
 \end{equation}
 which is due to the fact that the oscillation in the smooth amplitudes ($\cos(2\pi \frac{\pm n\pm m}{N_k}v)$ and $ \sin(2\pi \frac{\pm n\pm m}{N_k}v)$) can cancel out the summation in \ref{eqn:sum}. 
 
The inequalities in \eqref{eqn:inq1} implies that
 \begin{eqnarray}\label{eqn:f0}
 \|s^{P,(j)}_{cn,1}\|_{L^2}&\leq &O(\epsilon) +\left( \sum_{n=0}^{N^h-1}\left(\bar{\kappa}^{(j)}_{cn,1}(2\pi  t^h_n)  \right)^2h \right) ^{1/2}\left( \sum_{m=0}^{N^h-1}\left(\sum_{n=0}^{N^h-1}\left(  \frac{ D^{12}_h(m,n)-\gamma}{D^1_h(m)} \right)^2\right)  \right)^{1/2}\nonumber\\
 &&+ \gamma \frac{\sum_{n=0}^{N^h-1} \bar{\kappa}^{(j)}_{cn,1}(2\pi  t^h_n)}{D^1_h(m)}
 \end{eqnarray}
by the triangle inequality and H{\"o}lder's inequality. Note that
\begin{eqnarray*}
\left( \sum_{n=0}^{N^h-1}\left(\bar{\kappa}^{(j)}_{cn,k}(2\pi t^h_n)  \right)^2h\right)^{1/2}&\leq & \sum_{\ell\neq k}\sum_{m=-M_0}^{M_0}\left( \sum_{n=0}^{N^h-1}\left(s^{(j)}_{cm,\ell} (2\pi t^h_n)  \right)^2h\right)^{1/2}\\
&& +  \sum_{\ell\neq k}\sum_{m=-M_0}^{M_0} \left( \sum_{n=0}^{N^h-1}\left(s^{(j)}_{sm,\ell} (2\pi t^h_n)  \right)^2h\right)^{1/2}
\end{eqnarray*}
by the triangle inequality. Since all shape functions are in $\LL$, we have
 \[
 \left( \sum_{n=0}^{N^h-1}\left(s^{(j)}_{cm,\ell}(2\pi t^h_n) \right)^2h \right) ^{1/2}=\|s^{(j)}_{cm,\ell}\|_{L^2}+O(\epsilon),
 \]
 and
  \[
 \left( \sum_{n=0}^{N^h-1}\left(s^{(j)}_{sm,\ell}(2\pi t^h_n) \right)^2h \right) ^{1/2}=\|s^{(j)}_{sm,\ell}\|_{L^2}+O(\epsilon)
 \]
 for all $\ell$ and $m$. Hence,
 \begin{eqnarray}\label{eqn:f1}
&&\left( \sum_{n=0}^{N^h-1}\left(\bar{\kappa}^{(j)}_{cn,k}(2\pi t^h_n)  \right)^2h\right)^{1/2}\nonumber\\
&\leq & O(\epsilon) + (2M_0+1)(K-1)  \underset{1\leq k\leq K,-M_0\leq m\leq M_0}{\max}\left( \|s^{(j)}_{cm,k}\|_{L^2}+  \|s^{(j)}_{sm,k}\|_{L^2} \right).
 \end{eqnarray}
Since phase functions are in $\WD(M,N,h,\beta,\gamma)$, 
 \begin{equation}\label{eqn:f2}
 \left( \sum_{m=0}^{N^h-1}\left(\sum_{n=0}^{N^h-1}\left(  \frac{ D^{ki}_h(m,n)-\gamma}{D^k_h(m)} \right)^2\right)  \right)^{1/2}\leq \beta<1.
 \end{equation}
 Hence, by \eqref{eqn:f0}, \eqref{eqn:f1}, and \eqref{eqn:f2}, it holds that
  \begin{eqnarray}\label{eqn:f4}
 \|s^{P,(j)}_{cn,1}\|_{L^2}&\leq &O(\epsilon) + \gamma \frac{\sum_{n=0}^{N^h-1} \bar{\kappa}^{(j)}_{cn,1}(2\pi  t^h_n)}{D^1_h(m)}\\
 && + \beta  (2M_0+1)(K-1)  \underset{1\leq k\leq K,-M_0\leq m\leq M_0}{\max}\left( \|s^{(j)}_{cm,k}\|_{L^2}+  \|s^{(j)}_{sm,k}\|_{L^2} \right).\nonumber
 \end{eqnarray}
 Note that $D^k_h(m)\geq N^h\gamma=\gamma/h$. Hence, $\frac{\gamma}{D^k_h(m)}\leq h$ and
 \begin{equation}\label{eqn:f5}
  \gamma \frac{\sum_{n=0}^{N^h-1} \bar{\kappa}^{(j)}_{cn,1}( 2\pi t^h_n)}{D^1_h(m)}\leq  h\sum_{n=0}^{N^h-1} \bar{\kappa}^{(j)}_{cn,1}( 2\pi t^h_n).
 \end{equation}
 Since $s^{(j)}_{cn,k}\in\LL$, $s^{(j)}_{sn,k}\in\LL$,  $\int_0^1 s^{(j)}_{cn,k}(2\pi t)dt=0$, and $\int_0^1 s^{(j)}_{sn,k}(2\pi t)dt=0$  for all $n$ and $k$, by the stationary phase approximation, we have 
 \begin{equation}\label{eqn:f6}
  h\sum_{n=0}^{N^h-1} \bar{\kappa}^{(j)}_{cn,1}( 2\pi t^h_n)\leq O(\epsilon),
 \end{equation}
 as long as the number of grid points, $L$, and the lower bound of the oscillation frequency, $N$, are sufficiently large. In sum, by \eqref{eqn:f4}, \eqref{eqn:f5}, and \eqref{eqn:f6}, we have
   \begin{eqnarray}\label{eqn:f7}
 \|s^{P,(j)}_{cn,1}\|_{L^2}&\leq &O(\epsilon) \\
 && + \beta  (2M_0+1)(K-1)  \underset{1\leq k\leq K,-M_0\leq m\leq M_0}{\max}\left( \|s^{(j)}_{cm,k}\|_{L^2}+  \|s^{(j)}_{sm,k}\|_{L^2} \right).\nonumber
 \end{eqnarray}
 
In fact, following the arguments for \eqref{eqn:f7}, we can bound $ \|s^{P,(j)}_{cn,k}\|_{L^2}$ and $ \|s^{P,(j)}_{sn,k}\|_{L^2}$ for all $k$ with the same bound as in \eqref{eqn:f7}.
 
 By the conditions just listed immediately before Theorem \ref{thm:conv3}, and the definition in Equation \eqref{eqn:regthm12}, we have 
  \begin{eqnarray}\label{eqn:sest12}
 && \|s^{E,(j)}_{cn,k}\|_{L^2}\leq O(\epsilon)+  \|s^{P,(j)}_{cn,k}\|_{L^2} \leq \\
  && O(\epsilon)+ \beta  (2M_0+1)(K-1)  \underset{1\leq k\leq K,-M_0\leq m\leq M_0}{\max}\left( \|s^{(j)}_{cm,k}\|_{L^2}+  \|s^{(j)}_{sm,k}\|_{L^2} \right),\nonumber
 \end{eqnarray}
 and
  \begin{eqnarray}\label{eqn:sest22}
 && \|s^{E,(j)}_{sn,k}\|_{L^2}\leq O(\epsilon)+  \|s^{P,(j)}_{sn,k}\|_{L^2} \leq \\
  && O(\epsilon)+ \beta  (2M_0+1)(K-1)  \underset{1\leq k\leq K,-M_0\leq m\leq M_0}{\max}\left( \|s^{(j)}_{cm,k}\|_{L^2}+  \|s^{(j)}_{sm,k}\|_{L^2} \right).\nonumber
 \end{eqnarray}
 By Equation \eqref{eqn:sest12}, \eqref{eqn:sest22}, and mathematical induction, it is easy to show that
  \[
  \|s^{E,(j)}_{cn,k}\|_{L^2}\leq O(c_0 \epsilon+( \beta  (2M_0+1)(K-1))^j ),
 \]
 and 
  \[
  \|s^{E,(j)}_{sn,k}\|_{L^2}\leq O(c_0 \epsilon+(\beta  (2M_0+1)(K-1))^j )
 \]
for all $j\geq 0$ and $1\leq k\leq K$, where $c_0=\frac{1}{1-\beta  (2M_0+1)(K-1)}$ is a constant number coming from the geometric sequence due to the summation of the $O(\epsilon)$ term for all $j\geq 0$. 
 
The above proof is just for $K=2$. To care the general case for $K>2$, the only difference is the presence of more terms in the inequality estimations. 
\end{proof}

Theorem \ref{thm:conv3} shows that the regression function in each iteration of Line $4$ in Algorithm \ref{alg:MMD} decays, if $J_2=1$ and $\beta  (2M_0+1)(K-1)<1$, in the $L^2$ sense up to a fixed accuracy parameter as the iteration number becomes large. Hence, the recovered shape function converges and the residual decays up to a fixed accuracy parameter. For a general case when $J_2>1$, the convergence of Algorithm \ref{alg:MMD} is obvious following Theorem \ref{thm:conv3}. It is tedious to compare the convergence rate for different $J_2$, since it depends on how fast the multiresolution expansion coefficients $a_{cn,k}$, $a_{sn,k}$, $b_{cn,k}$, and $b_{sn,k}$ decay in $n$. In general, the faster the coefficients decay in $n$, the larger $J_2$ should be used. Another immediate result of Theorem \ref{thm:conv3} is that, the correlation of MIMF's with well-differentiated phase functions is asymptotically zero in the sense of recursive diffeomorphism-based regression. In other words, when the MIMF and MMD model is applied to analyze a time series, the resulting representation by MIMF's is unique.  The robustness of Algorithm \ref{thm:conv3} is an immediate result following the proof of Theorem  \ref{thm:ns} and \ref{thm:conv2}.

\section{Numerical Example}
\label{sec:NumEx}

In this section, some numerical examples of synthetic and real data are provided to support the multiresolution mode decomposition (MMD) model and Algorithm \ref{alg:MMD}. We apply the least squares spline regression method with free knots in \cite{Regression} to solve all the regression problems in this paper. The implementation of the regression method is available online\footnote{Available at https://www.mathworks.com/matlabcentral/fileexchange/25872-free-knot-spline-approximation.}. In all synthetic examples, we assume the fundamental instantaneous phases are known and only focus on verifying the theory in Section \ref{sec:MMD}. In real examples, we apply the one-dimensional synchrosqueezed wave packet transform (SSWPT) to estimate instantaneous phases as inputs of the multiresolution mode decomposition in Algorithm \ref{alg:MMD}. The implementation of the SSWPT is available in SynLab\footnote{Available at https://github.com/HaizhaoYang/SynLab.}, {while the code for the MMD is available online in a MATLAB package named DeCom}\footnote{Available at https://github.com/HaizhaoYang/DeCom.}.

Before presenting results, we would like to summarize the main parameters in the above packages and in Algorithm \ref{alg:MMD}. In the spline regression with free knots, main parameters are
\begin{itemize}
\item $nk$:  the number of free knots;
\item $krf$: the knot removal factor, a number quantifying how likely a free knot would be removed;
\item $ord$: the highest degree of spline polynomials.
\end{itemize}
In SynLab, main parameters are
\begin{itemize}
\item $s$: a geometric scaling parameter;
\item $rad$: the support size of the mother wave packet in the Fourier domain;
\item $red$: a redundancy parameter, the number of frames in the wave packet transform;
\item $\epsilon_{sst}$: a threshold for the wave packet coefficients.
\end{itemize}
In Algorithm \ref{alg:MMD}, main parameters are
\begin{itemize}
\item $J_1$: the maximum number of iterations allowed in Algorithm \ref{alg:MMD};
\item $J_2$: the maximum number of iterations allowed in Algorithm \ref{alg:RDBR2};
\item $M_0$: the bandwidth parameter;
\item $\epsilon_1=\epsilon_2=\epsilon$: the accuracy parameter.
\end{itemize}
For the purpose of convenience, the synthetic data is defined in $[0,1]$ and sampled on a uniform grid. All these parameters in different examples are summarized in Table \ref{tab:1}.

\begin{table}[htp]
\centering
\begin{tabular}{rcccccccccccc}
\toprule
  figure &  $nk$ & $krf$ & $ord$ & $s$
                             & $rad$ & $red$ & $\epsilon_{sst}$ & $J_1$ & $J_2$ & $M_0$ & $\epsilon$ & $L$ \\
\toprule
 7, 8, 9  & 20 & 1.0001 & 3 & 0.5 & 1.5 & 8 & 1e-3 & 200 & 10 & 20 & 1e-6 & $2^{16}$ \\
 10, 11, 12  & 20 & 1.0001 & 3 & 0.5 & 1 & 8 & 1e-3 & 200 & 10 & 20 & 1e-6 & $4000$ \\
14, 15  & 20 & 1.0001 & 3 & -- & -- & -- & -- & 200 & 10 & 10 & 1e-6 & $2^{15}$ \\
16, 17, 18  & 20 & 1.0001 & 3 & 0.5 & 1.5 & 8 & 1e-3 & 200 & 10 & 40 & 1e-6 & $2^{16}$ \\
\bottomrule
\end{tabular}
\caption{Parameters in the spline regression, SynLab, and Algorithm \ref{alg:MMD}. The notation ``--" means the corresponding parameter is not used.}
\label{tab:1}
\end{table}

In the noisy synthetic examples, Gaussian random noise with a distribution $\N(0,\sigma^2)$ is used. The signal-to-noise ration ($\SNR$) of $f(t)$ is defined as
\begin{equation}
\label{eqn:SNR}
\SNR [dB]=10\log_{10}\left(\frac{\|f\|_{L^2}}{\sigma^2}\right)
\end{equation}
where $\sigma^2$ is the variation of the noise.

\subsection{Synthetic multiresolution mode decomposition}

In this section, clean and noisy synthetic time series are provided to demonstrate the effectiveness and the robustness of the multiresolution mode decomposition in Algorithm \ref{alg:MMD}. 

{\bf The first synthetic example.} We consider a simple case when the signal has two MIMFs with ECG shape functions. In particular, {we generate MIMFs such that $s_{sn,k}=s_{cn,k}=s_k$ for all $n\geq 0$ and $k$, and $s_{sn,k}=s_{cn,k}=0$ for all $n<0$ and $k$, resulting in an example that can be either considered as a GMD or an MMD problem. Through this example we see that Algorithm \ref{alg:MMD} for MMD can also be applied to solve the GMD problem.} For example, we consider a signal of the form
\begin{equation}
\label{eqn:ex2}
f(t) = f_1(t)+f_2(t) +ns,
\end{equation}
where $ns$ denotes Gaussian random noise with a distribution $\mathcal{N}(0,\sigma^2)$, 
\begin{equation}\label{eq:f1_4}
f_1(t) = \alpha_1(\phi_1(t))s_1(300\pi \phi_1(t)),
\end{equation}
\begin{equation}\label{eq:f2_4}
f_2(t) = \alpha_2(\phi_2(t))s_2(440\pi\phi_2(t)),
\end{equation}
\[
\alpha_1(t) =  1+0.2\cos(2\pi t)+0.1\sin(2\pi t), 
\]
\[
\alpha_2(t) =  1+0.1\cos(2\pi t)+0.2\sin(2\pi t),
\]
\[
\phi_1(t) = t+0.006\sin(2\pi t),
\]
and 
\[
\phi_2(t) = t+0.006\cos(2\pi t).
\]
$s_1(2\pi t)$ and $s_2(2\pi t)$ are generalized shape functions defined in $[0,1]$ as shown in Figure \ref{fig:4_0}. In the noiseless example when $ns=0$, we apply the MMD in Algorithm \ref{alg:MMD} with the known instantaneous phases mentioned just above to estimate the multiresolution expansion coefficients and the shape functions series. Note that it is sufficient to show the estimation accuracy of the product of the multiresolution expansion coefficient and its corresponding shape function; the products of the first five leading expansion coefficients and shape functions are shown in Figure \ref{fig:4}. The estimation errors are very small; the estimated results and the ground truth are almost indistinguishable. The numerical results of a noisy signal with $ns=\mathcal{N}(0,2.25)$ are shown in Figure \ref{fig:4ns}. Even if the SNR for the leading terms $\mathcal{M}_0(f_1)(t)$ and $\mathcal{M}_0(f_2)(t)$ are near $-10$, the estimation $a_{0,1}s_{c0,1}(t)$ is almost equal to the ground truth $s_1(t)$, and the estimation $a_{0,2}s_{c0,2}(t)$ is almost equal to the ground truth $s_2(t)$. The SNR for the other terms are even much smaller (less than $-20$). However, as shown in Figure \ref{fig:4ns}, the estimation of other terms are still reasonable, capturing the main trend of the shape functions. Therefore, Algorithm \ref{alg:MMD} is quite robust against noise perturbation. 

\begin{figure}[ht!]
  \begin{center}
    \begin{tabular}{cc}
      \includegraphics[height=1in]{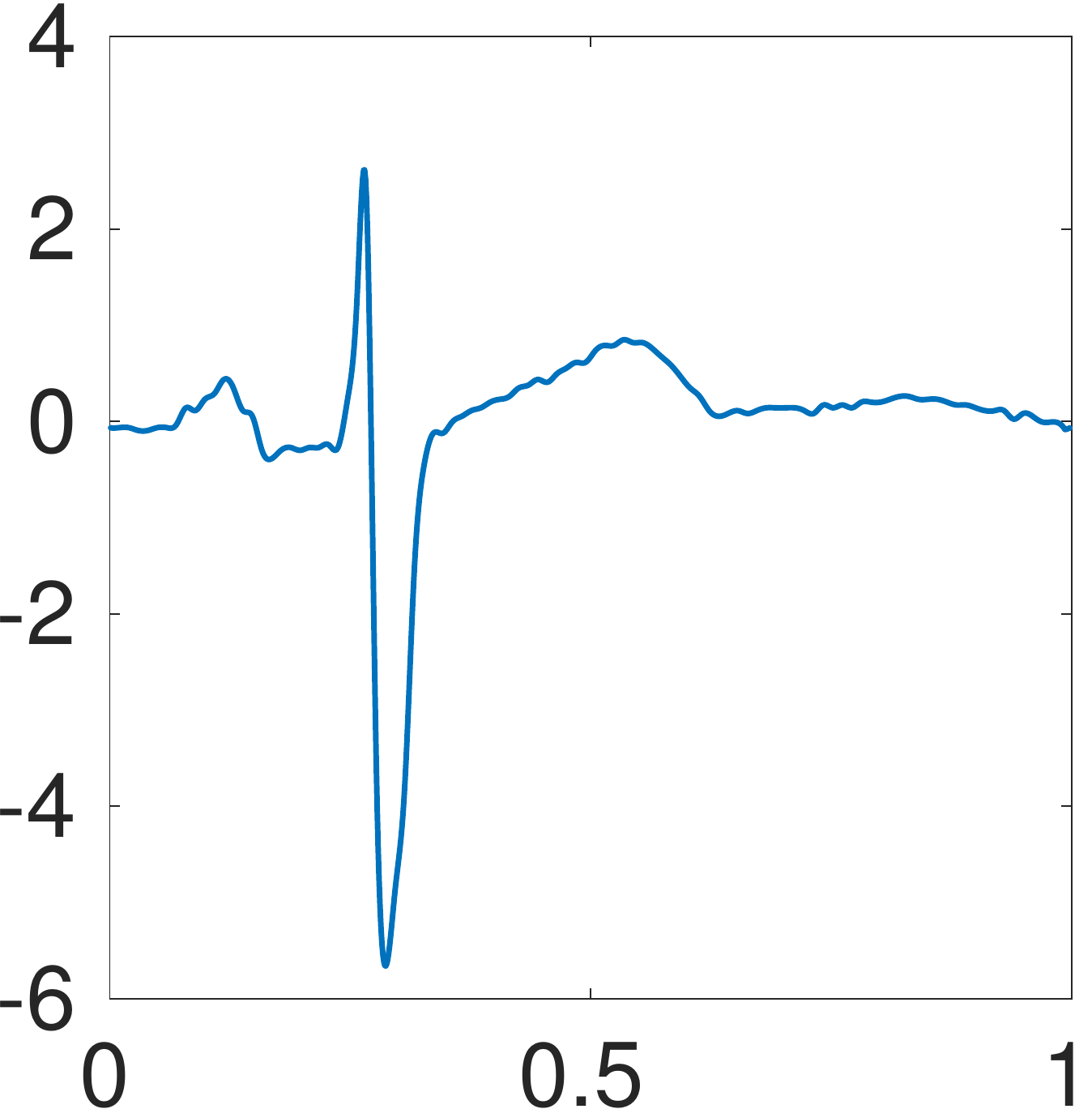}  &
   \includegraphics[height=1in]{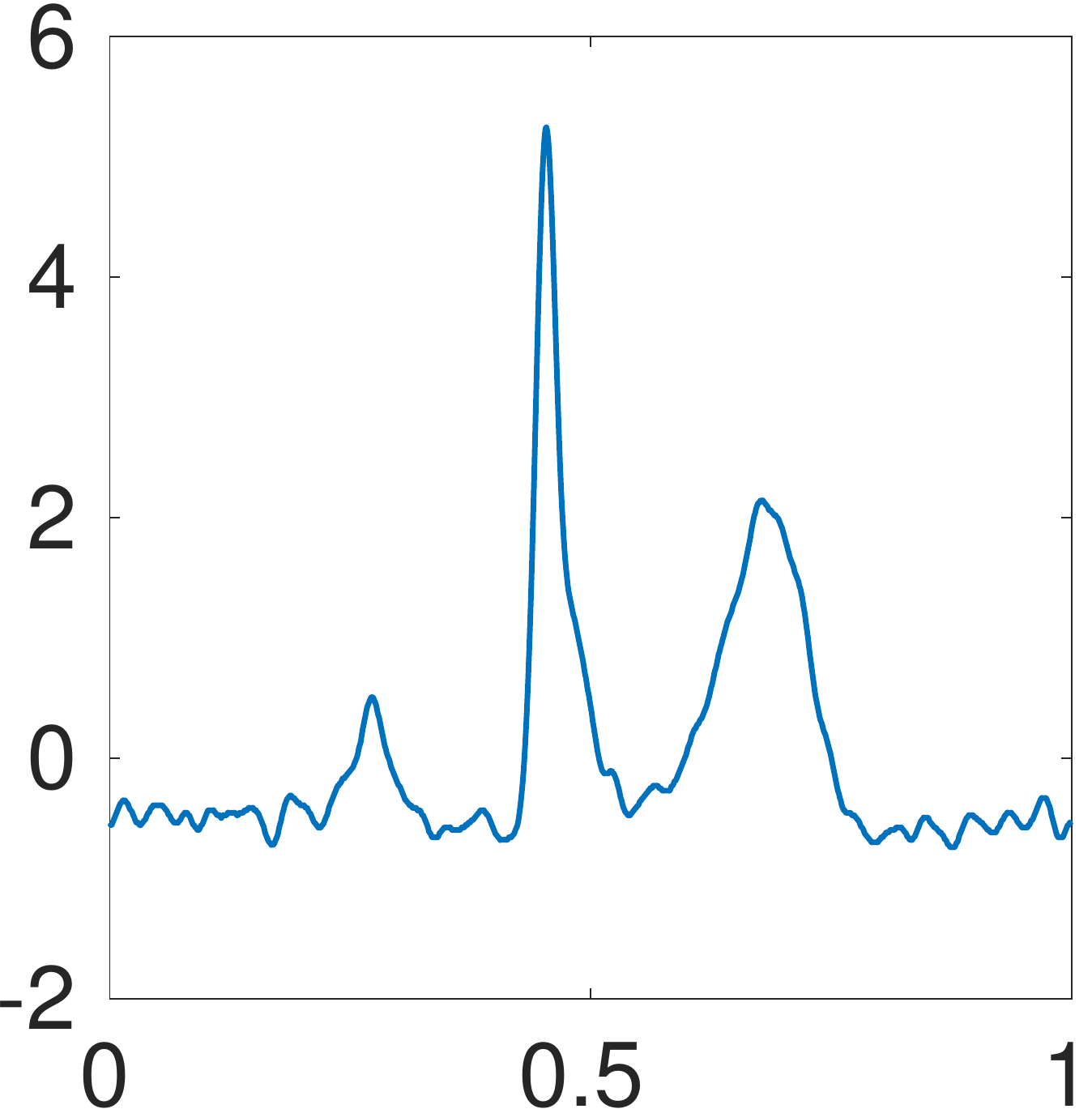}  
    \end{tabular}
  \end{center}
  \caption{Shape function $s_1(2\pi t)$ in \eqref{eq:f1_4} and $s_2(2\pi t)$ in \eqref{eq:f2_4}.}
\label{fig:4_0}
\end{figure}

\begin{figure}[ht!]
  \begin{center}
    \begin{tabular}{ccccc}
      \includegraphics[width=1in]{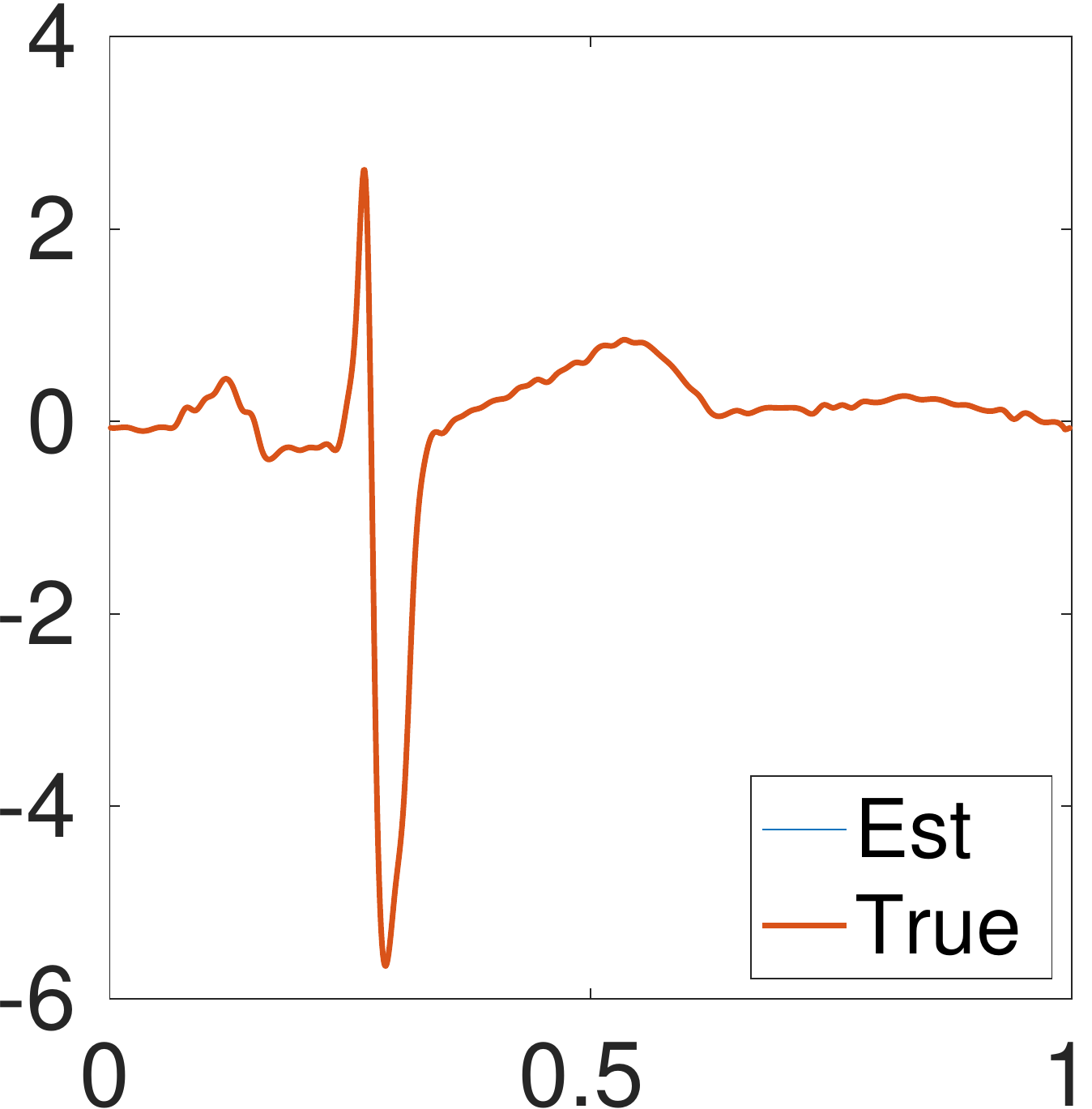}   &
      \includegraphics[width=1.1in]{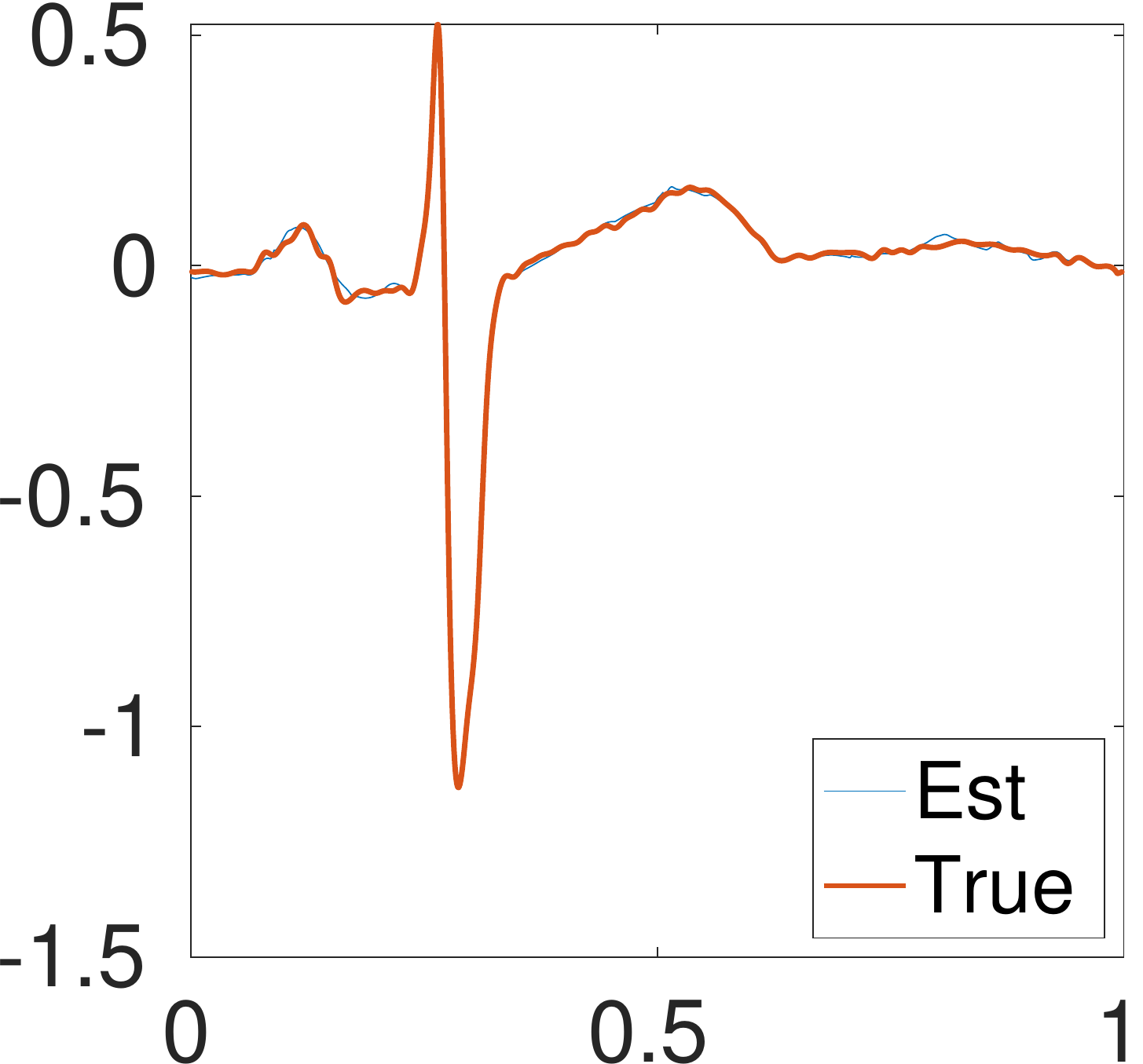}   &
      \includegraphics[width=1.15in]{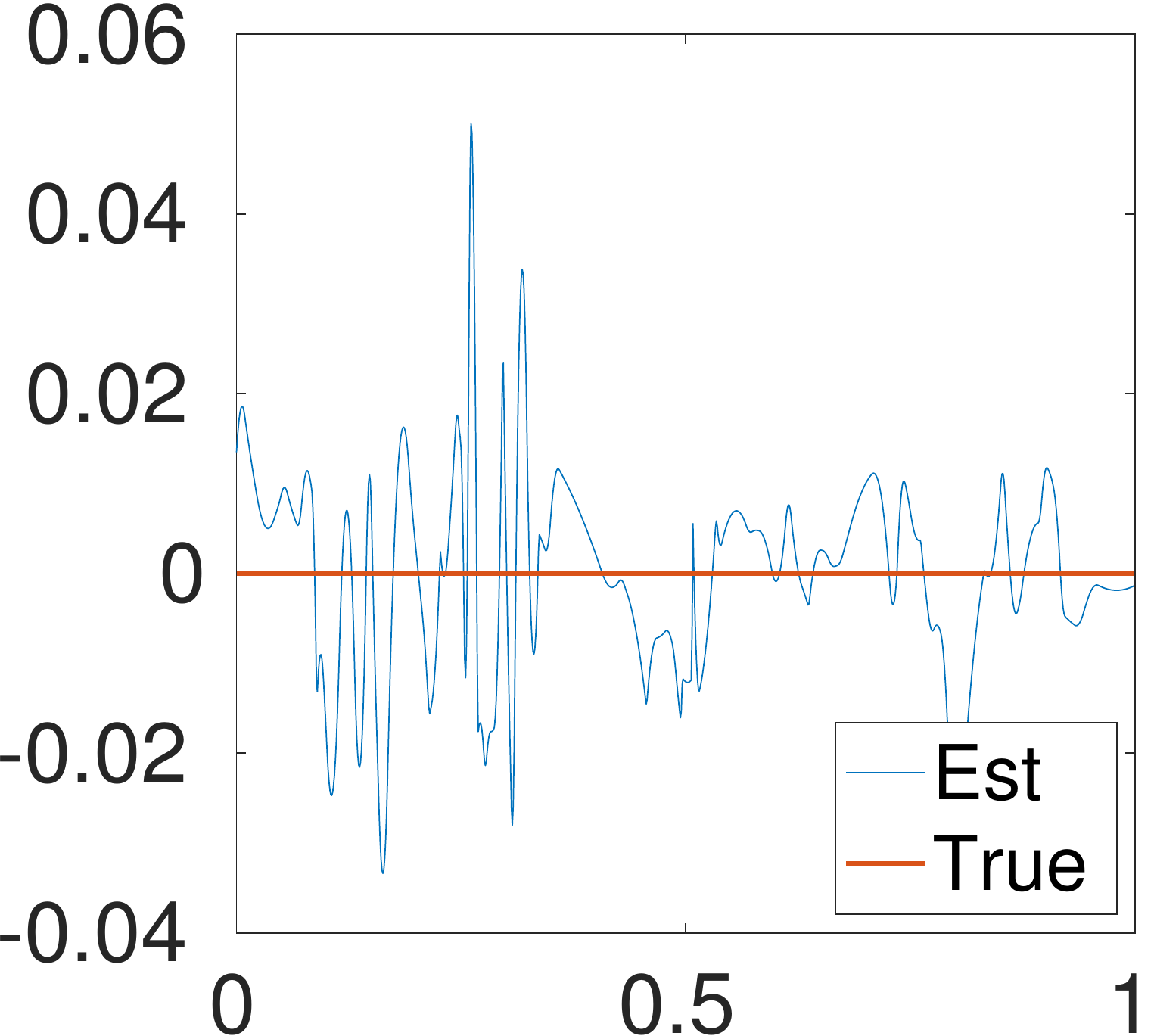}   &
      \includegraphics[width=1.1in]{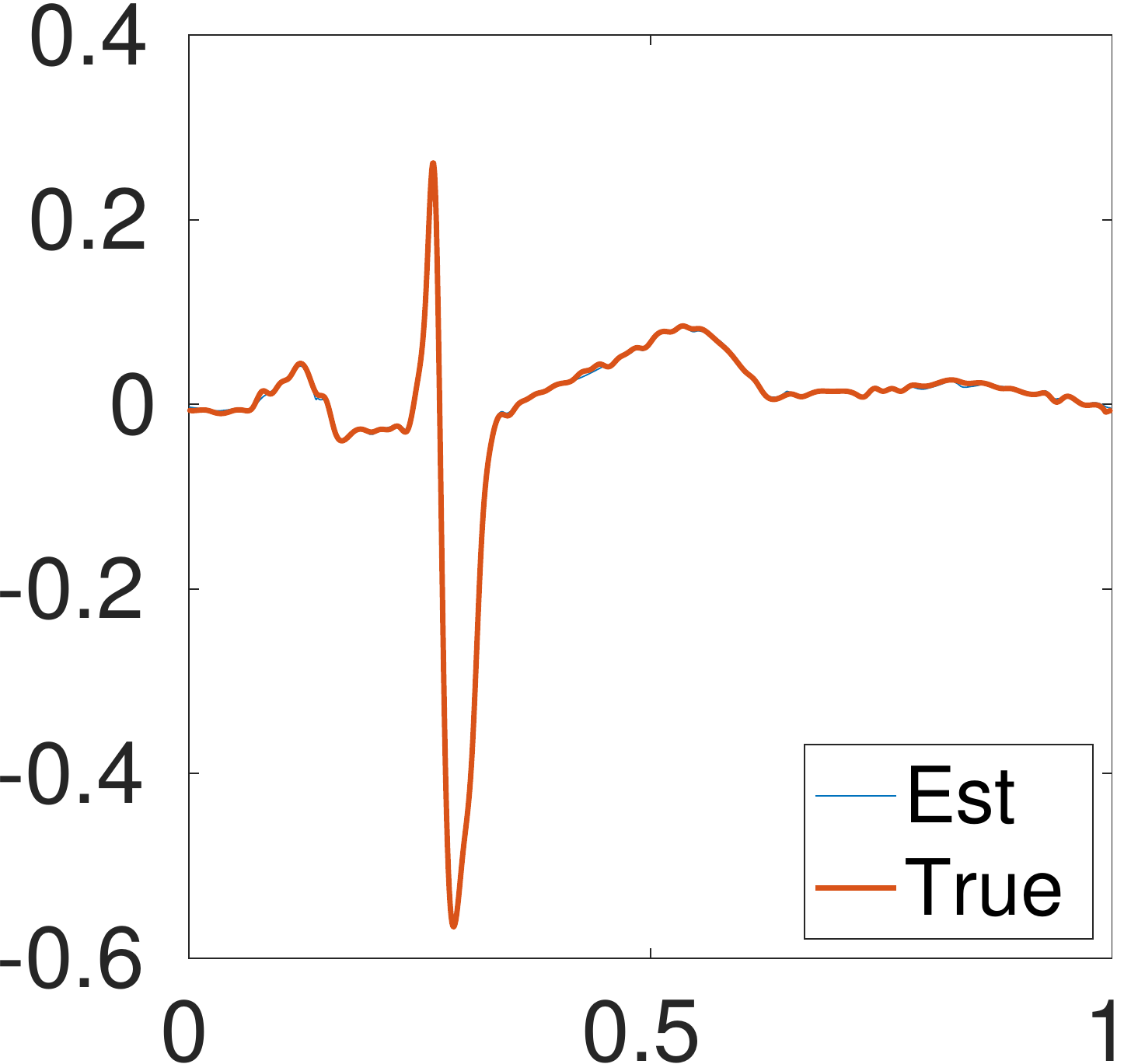}   &
      \includegraphics[width=1in]{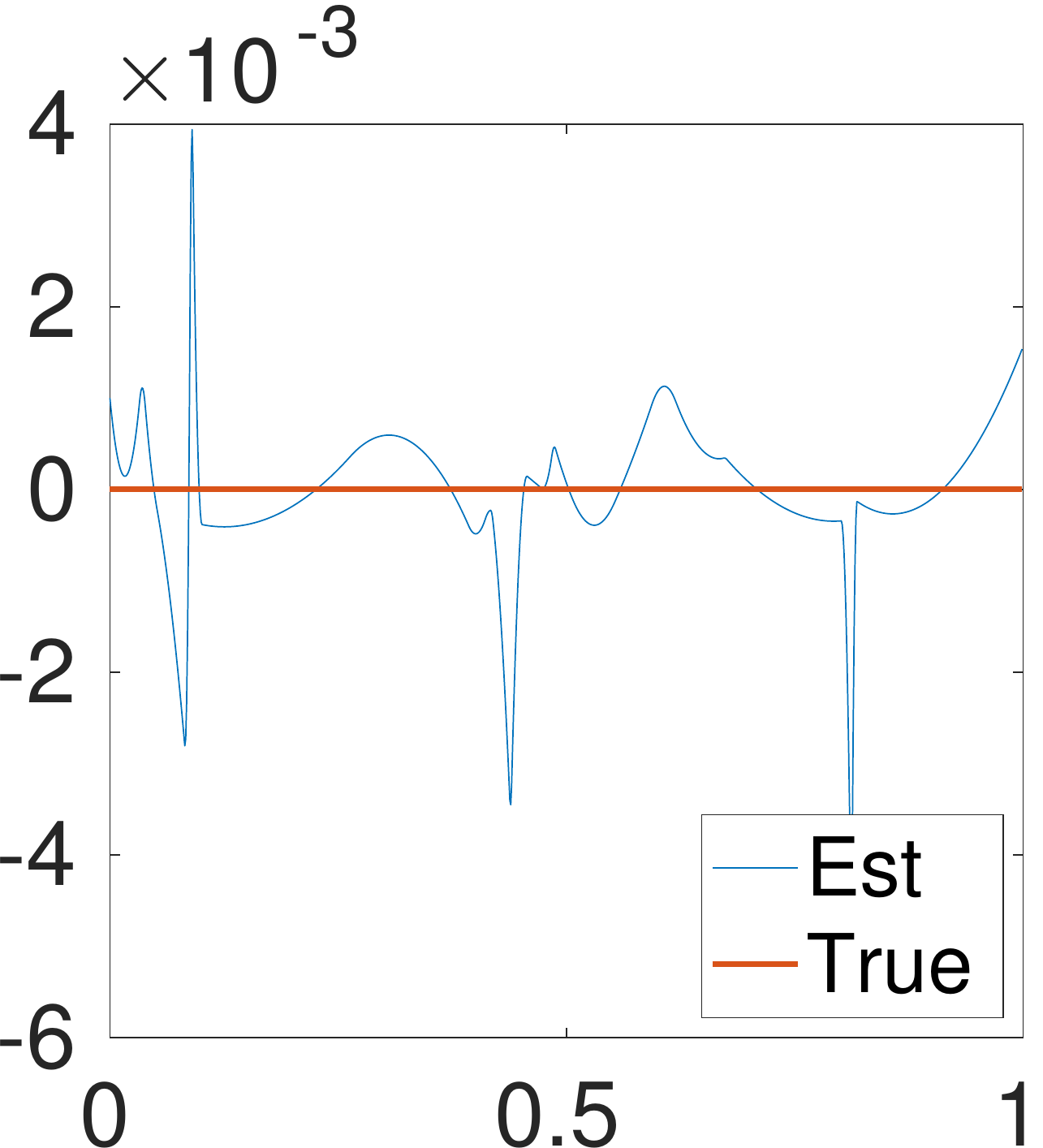}   \\
      \includegraphics[width=1in]{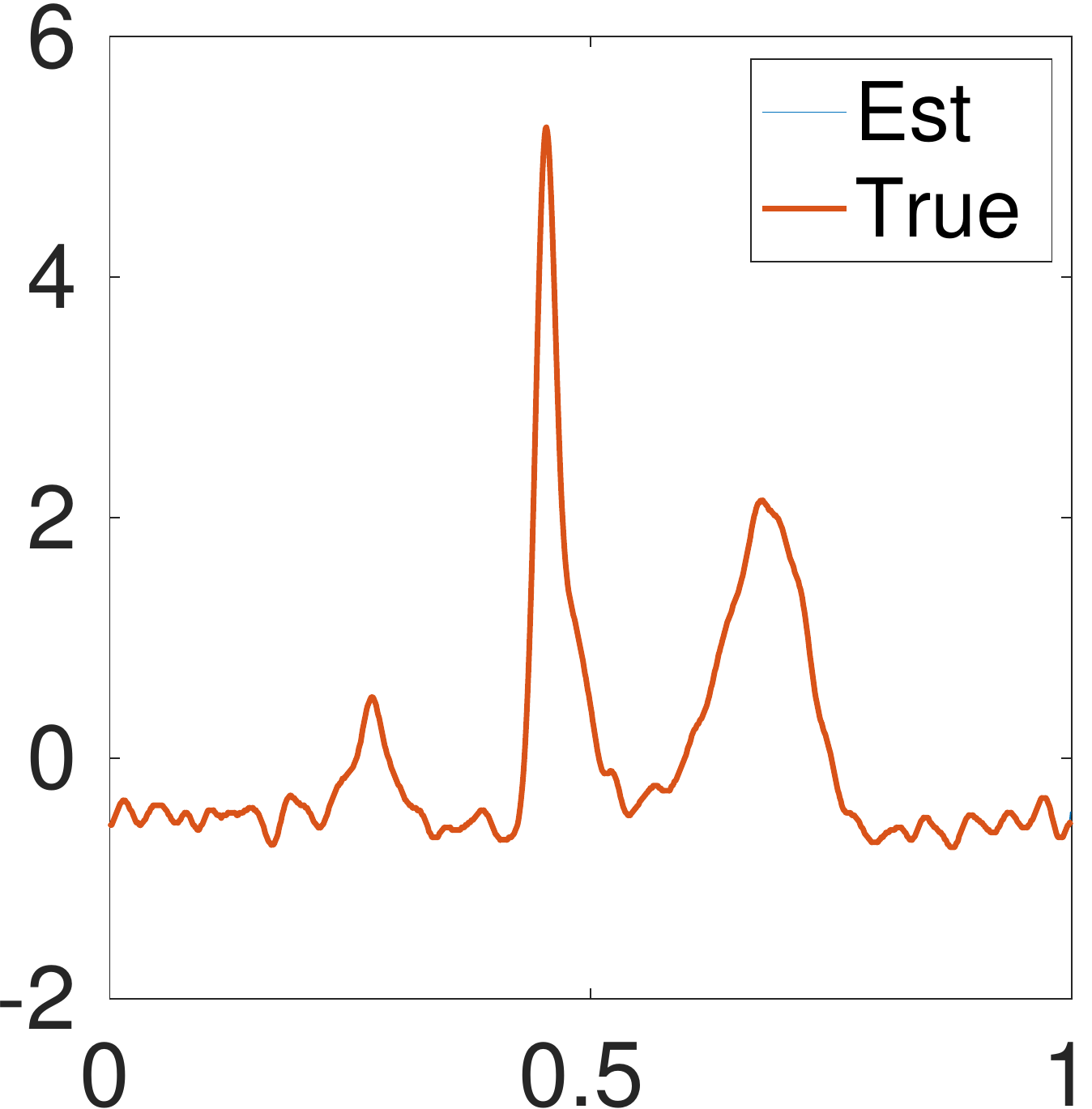}   &
      \includegraphics[width=1.1in]{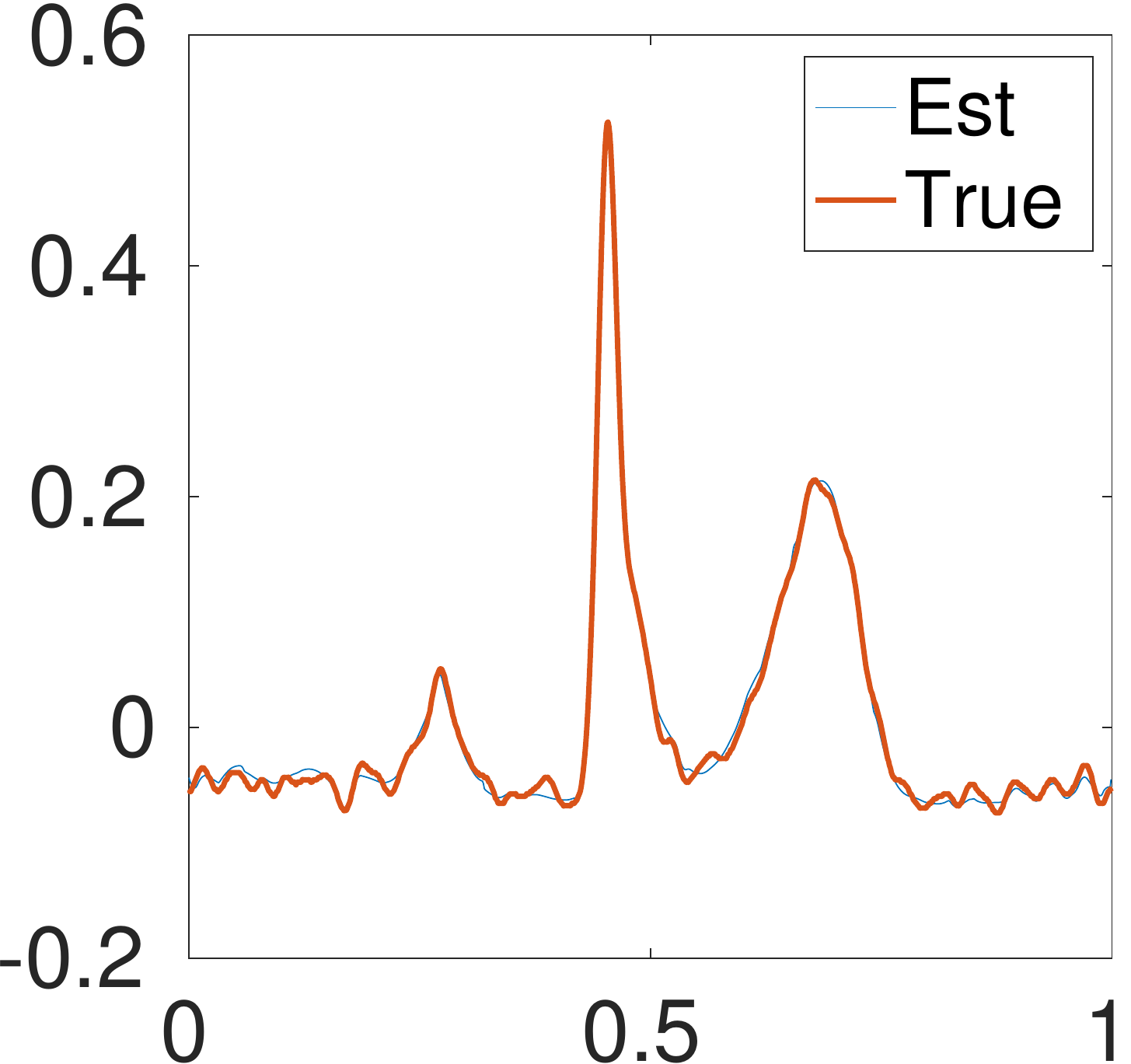}   &
      \includegraphics[width=1.15in]{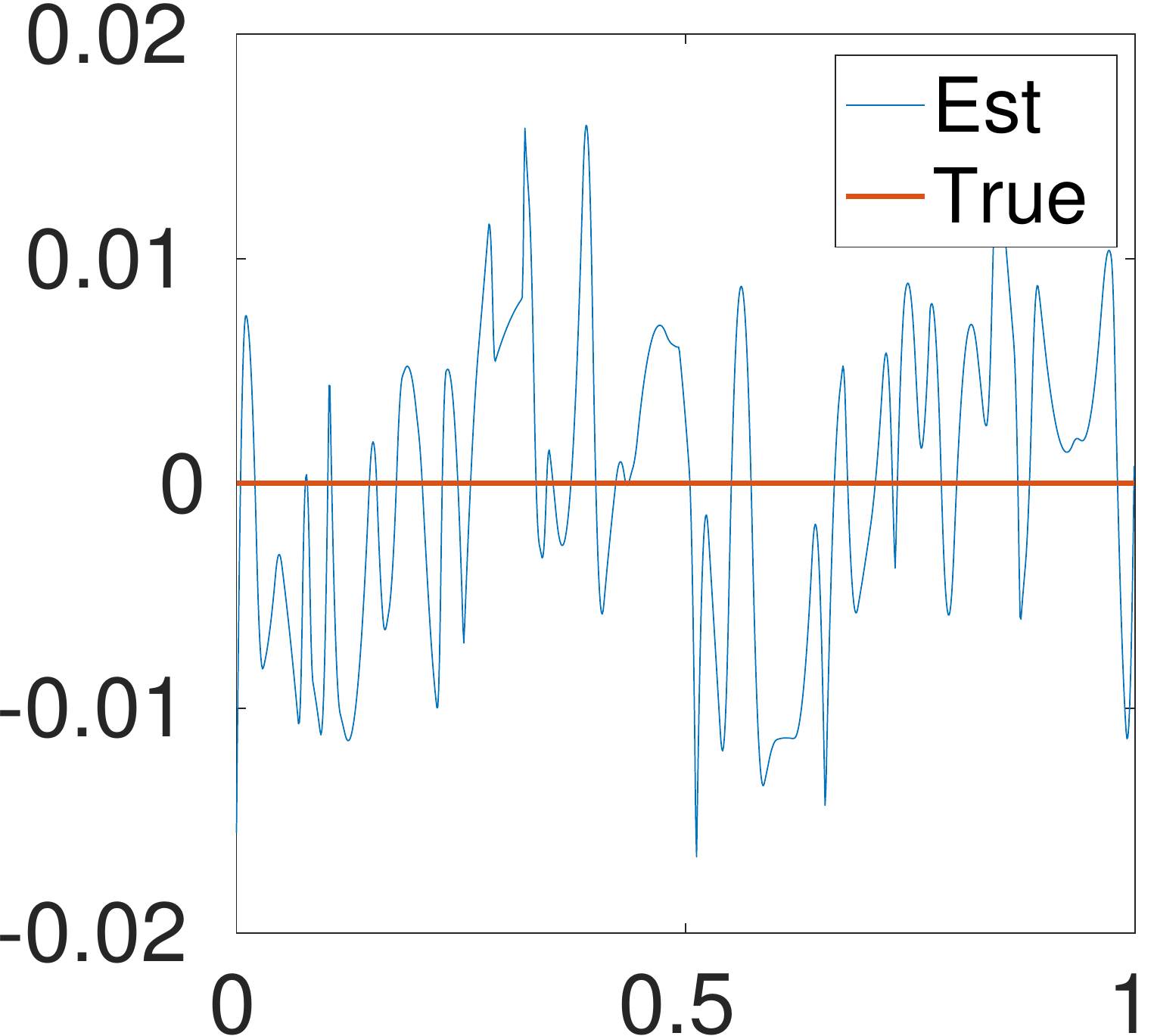}   &
      \includegraphics[width=1.05in]{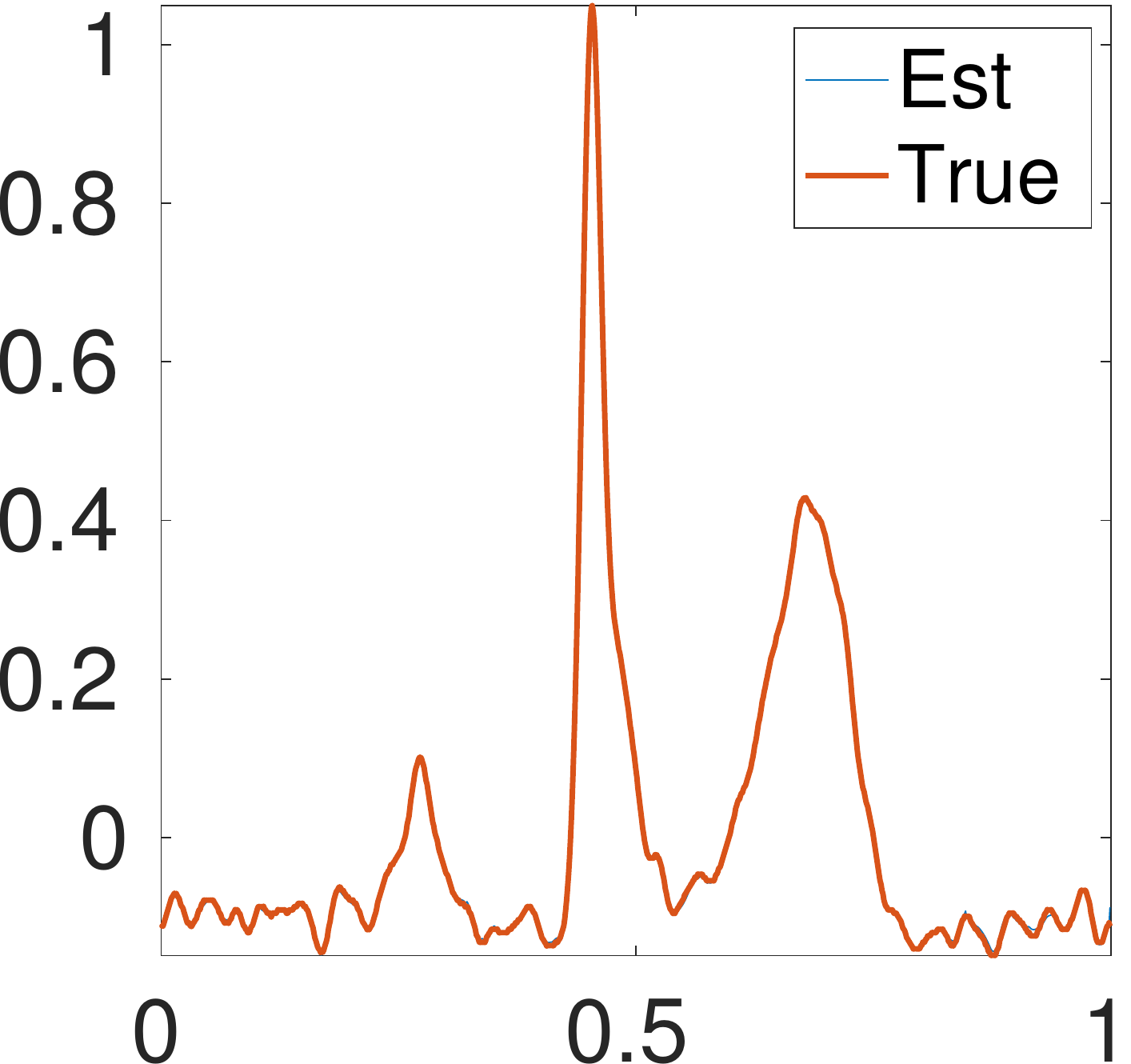}   &
      \includegraphics[width=1.15in]{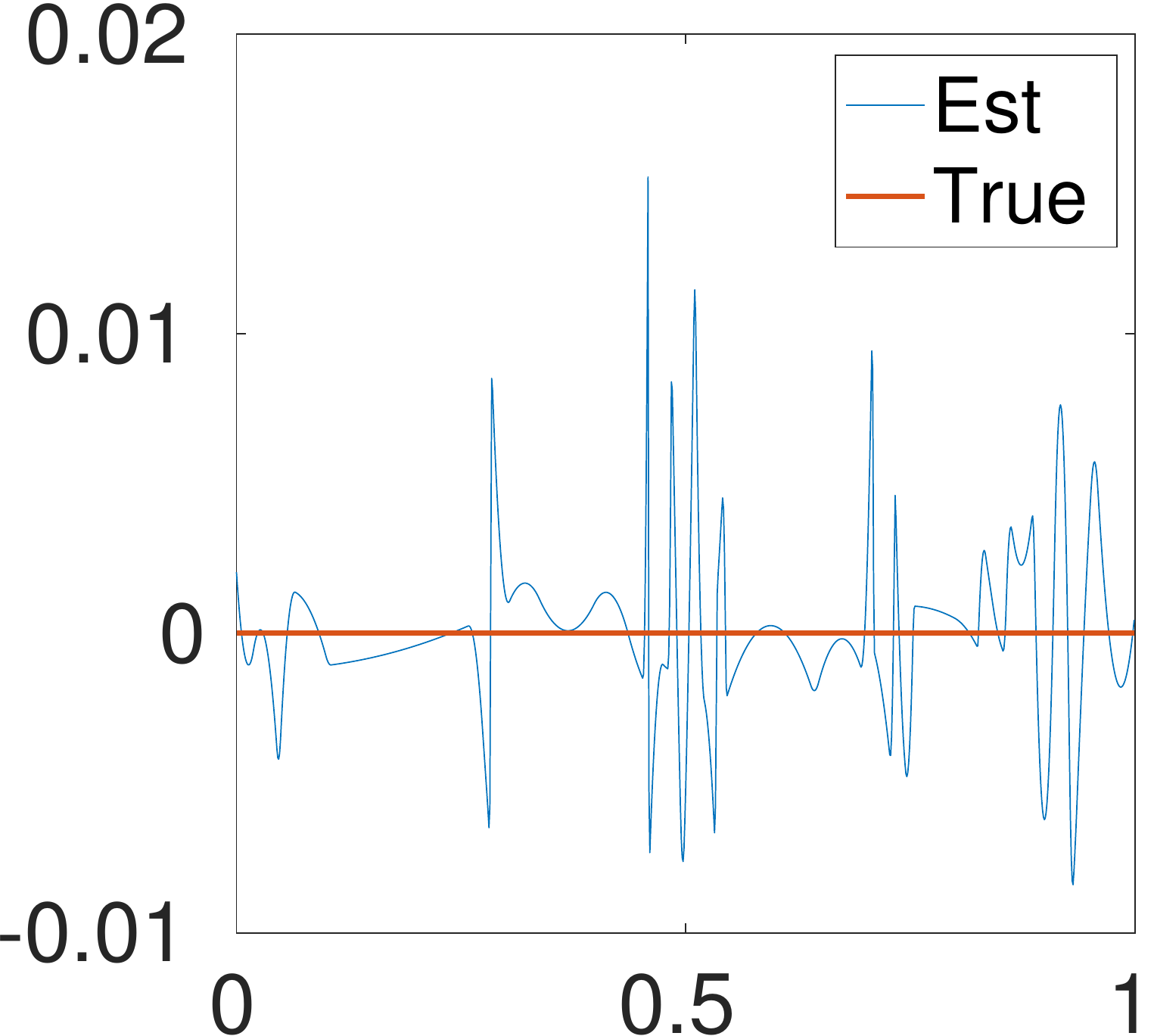}   
    \end{tabular}
  \end{center}
  \caption{Top row: estimated shape functions $a_{0,1}s_{c0,1}(2\pi t)$, $a_{1,1}s_{c1,1}(2\pi t)$, $a_{-1,1}s_{c-1,1}(2\pi t)$, $b_{1,1}s_{s1,1}(2\pi t)$, and $b_{-1,1}s_{s-1,1}(2\pi t)$ of $f_1(t)$ in \eqref{eq:f1_4}. Bottom row: estimated shape functions $a_{0,2}s_{c0,2}(2\pi t)$, $a_{1,2}s_{c1,2}(2\pi t)$, $a_{-1,2}s_{c-1,2}(2\pi t)$, $b_{1,2}s_{s1,2}(2\pi t)$, and $b_{-1,2}s_{s-1,2}(2\pi t)$ of $f_2(t)$ in \eqref{eq:f2_4}.}
\label{fig:4}
\end{figure}

\begin{figure}[ht!]
  \begin{center}
    \begin{tabular}{ccccc}
      \includegraphics[width=1.025in]{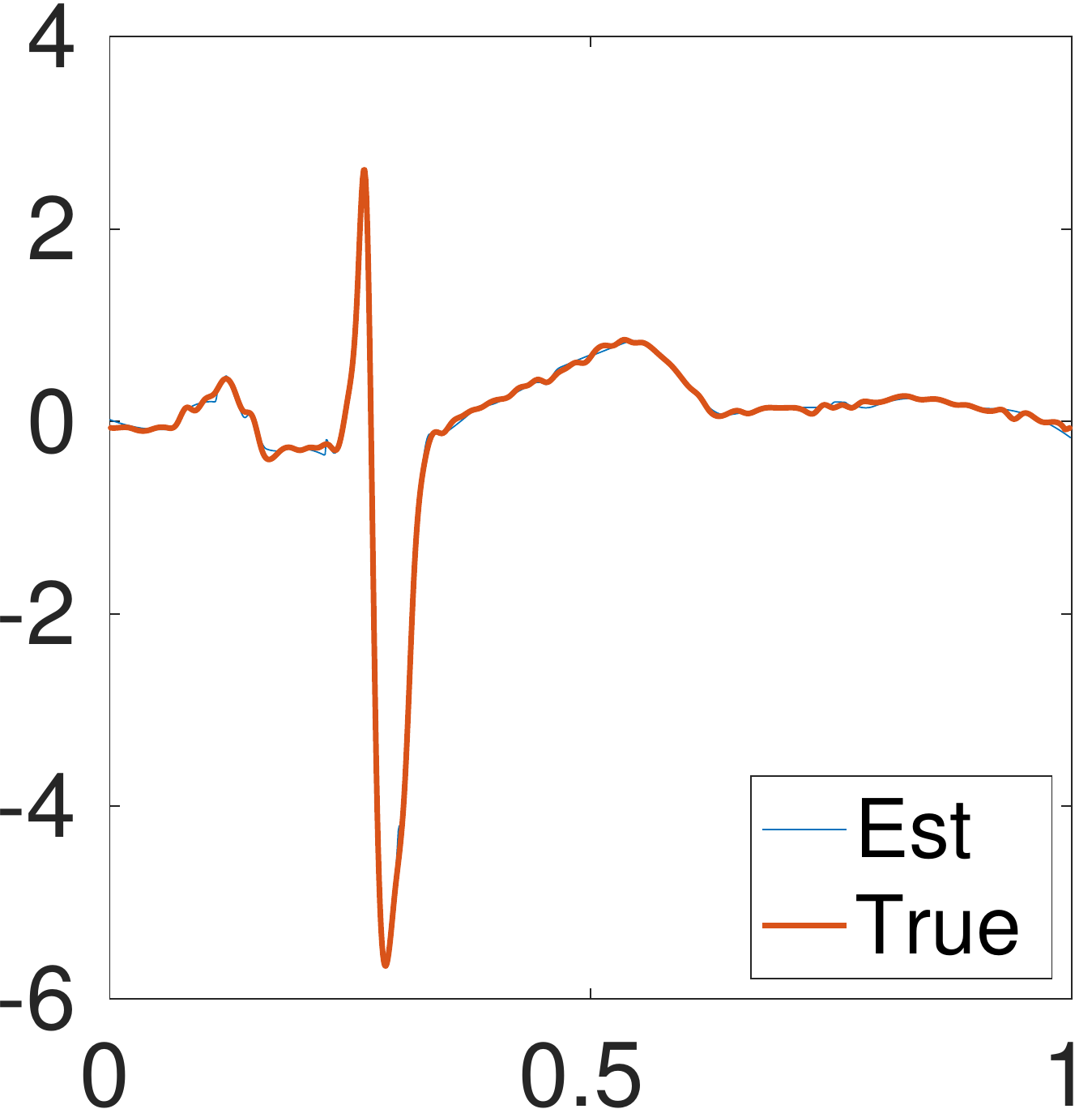}   &
      \includegraphics[width=1.1in]{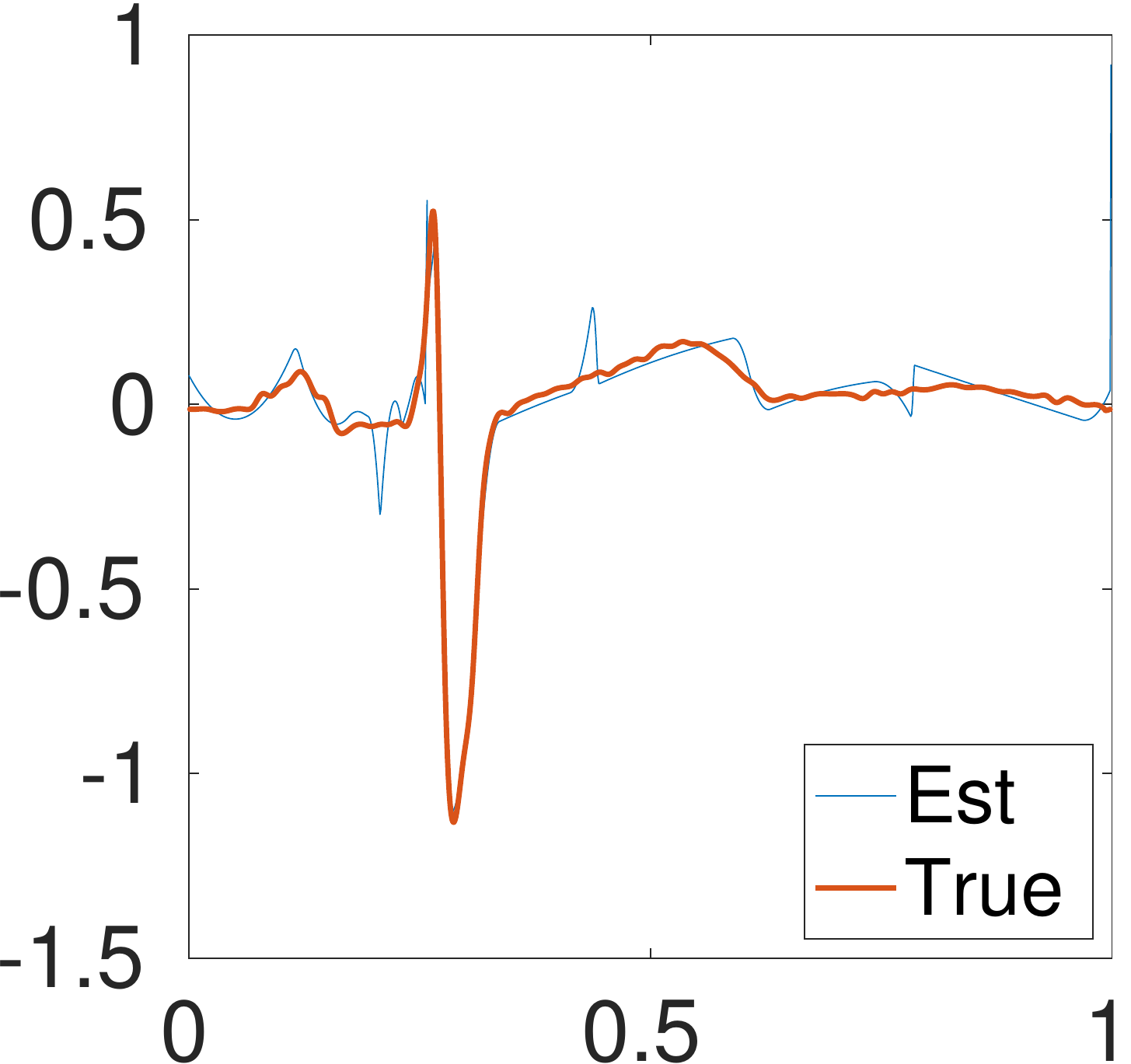}   &
      \includegraphics[width=1.1in]{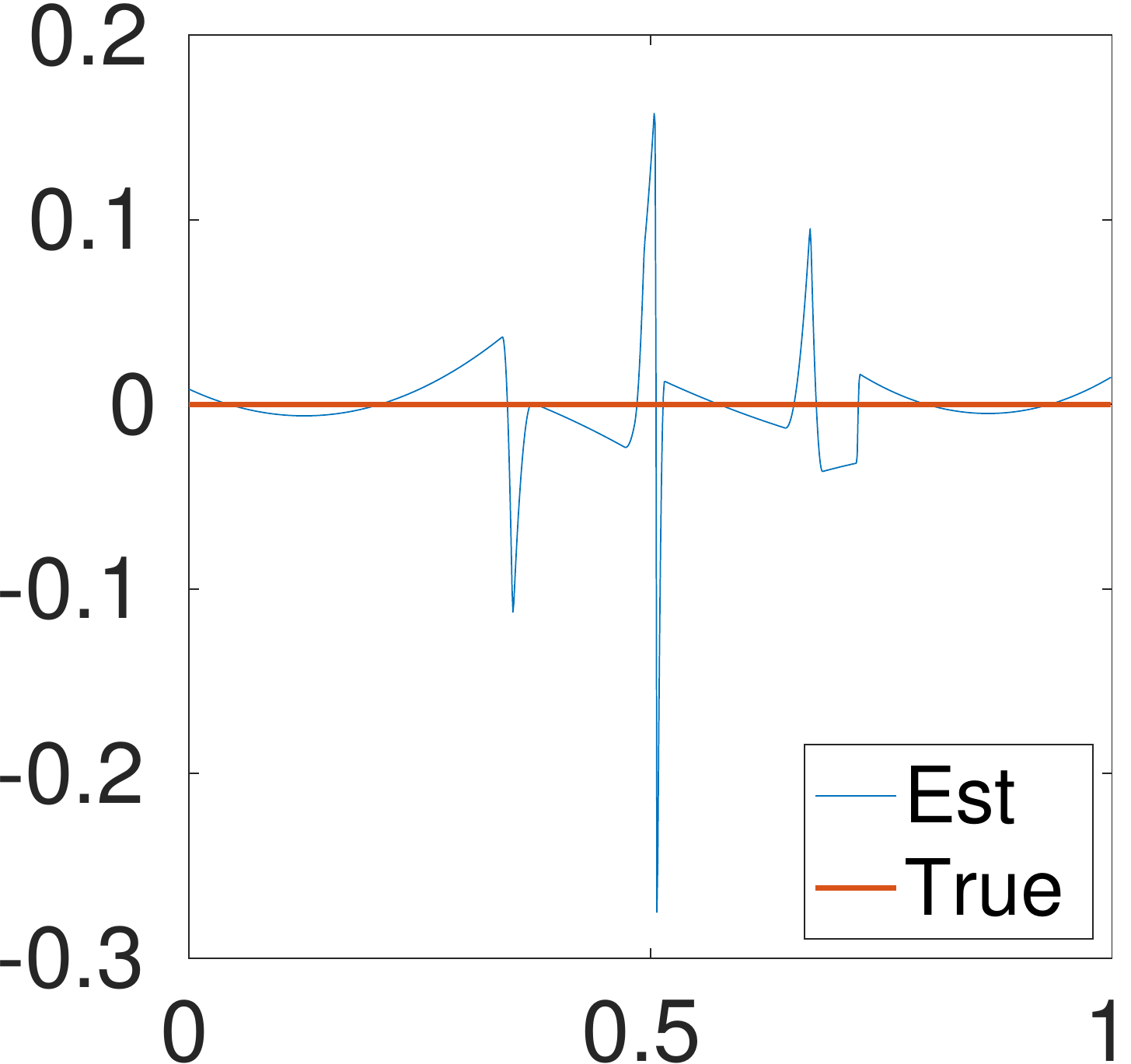}   &
      \includegraphics[width=1.1in]{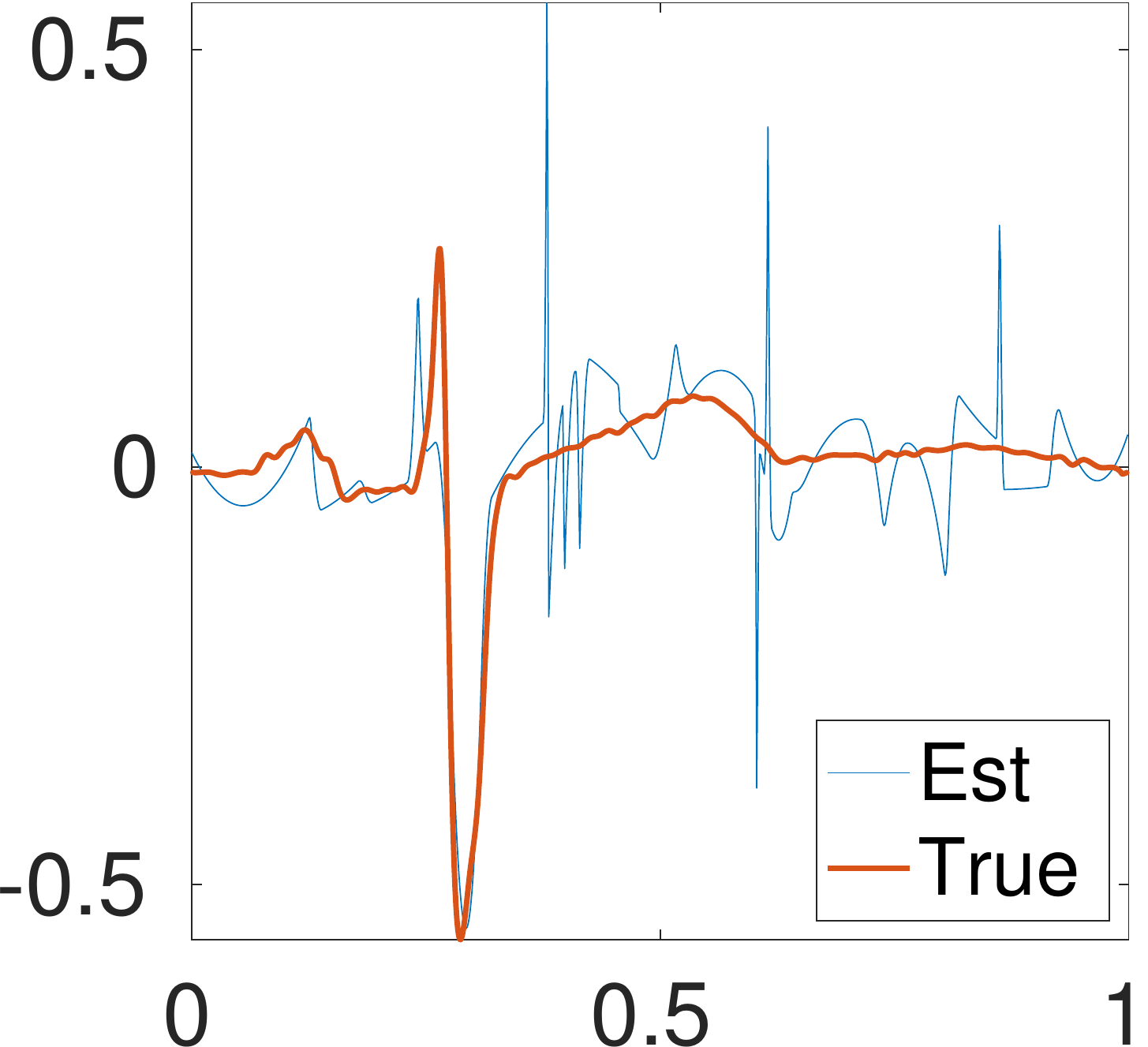}   &
      \includegraphics[width=1.15in]{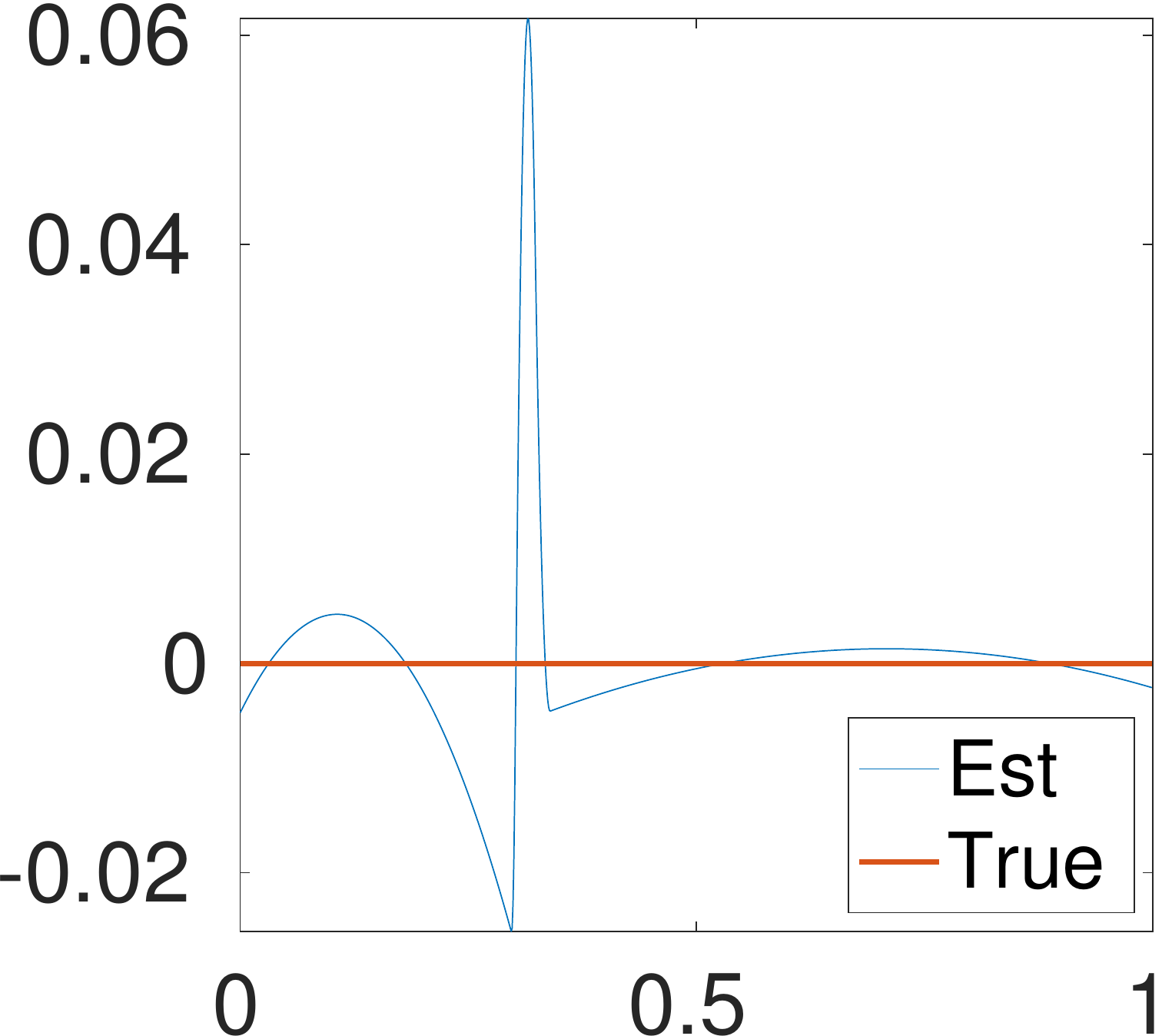}   \\
      \includegraphics[width=1in]{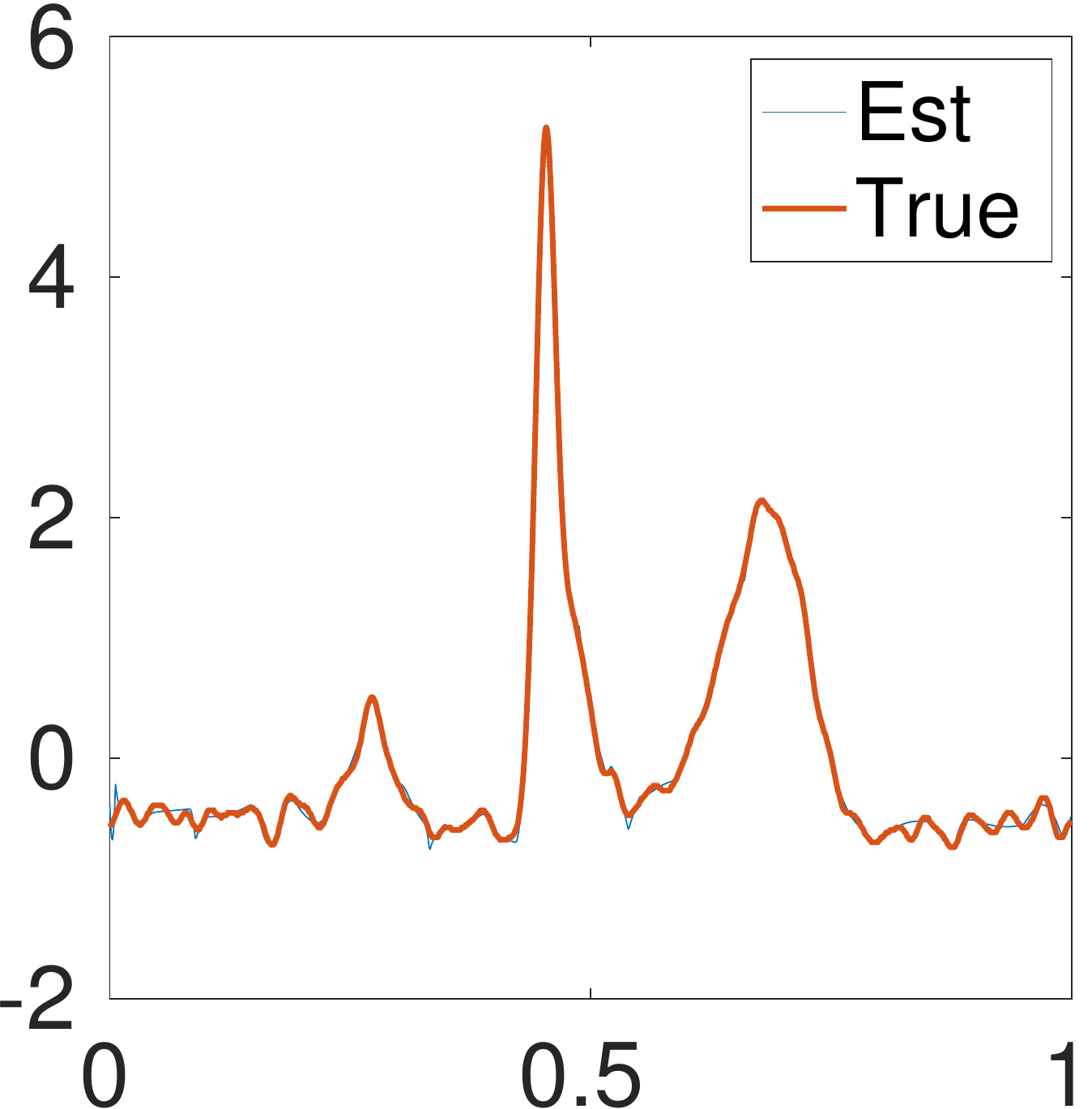}   &
      \includegraphics[width=1.1in]{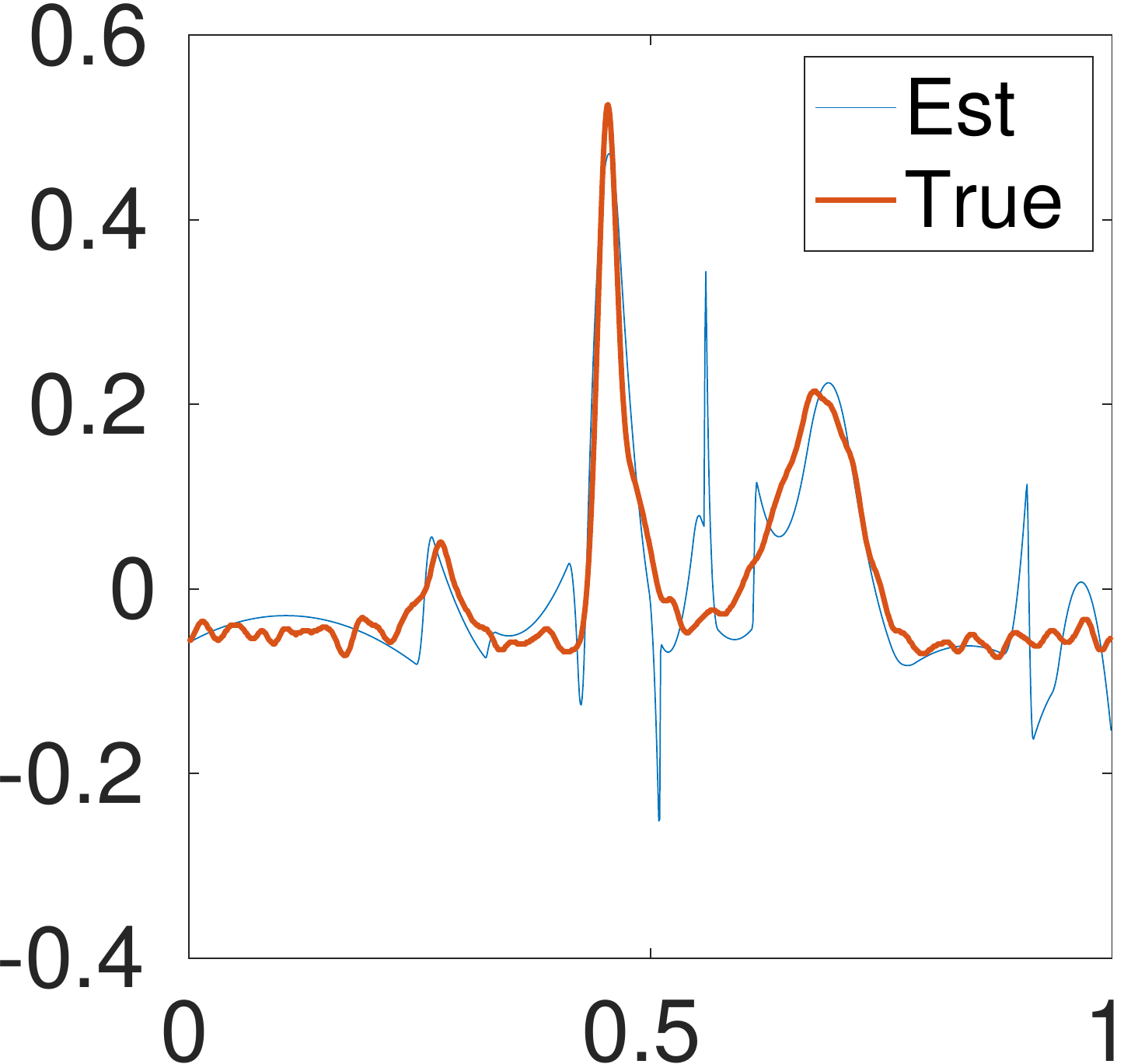}   &
      \includegraphics[width=1.1in]{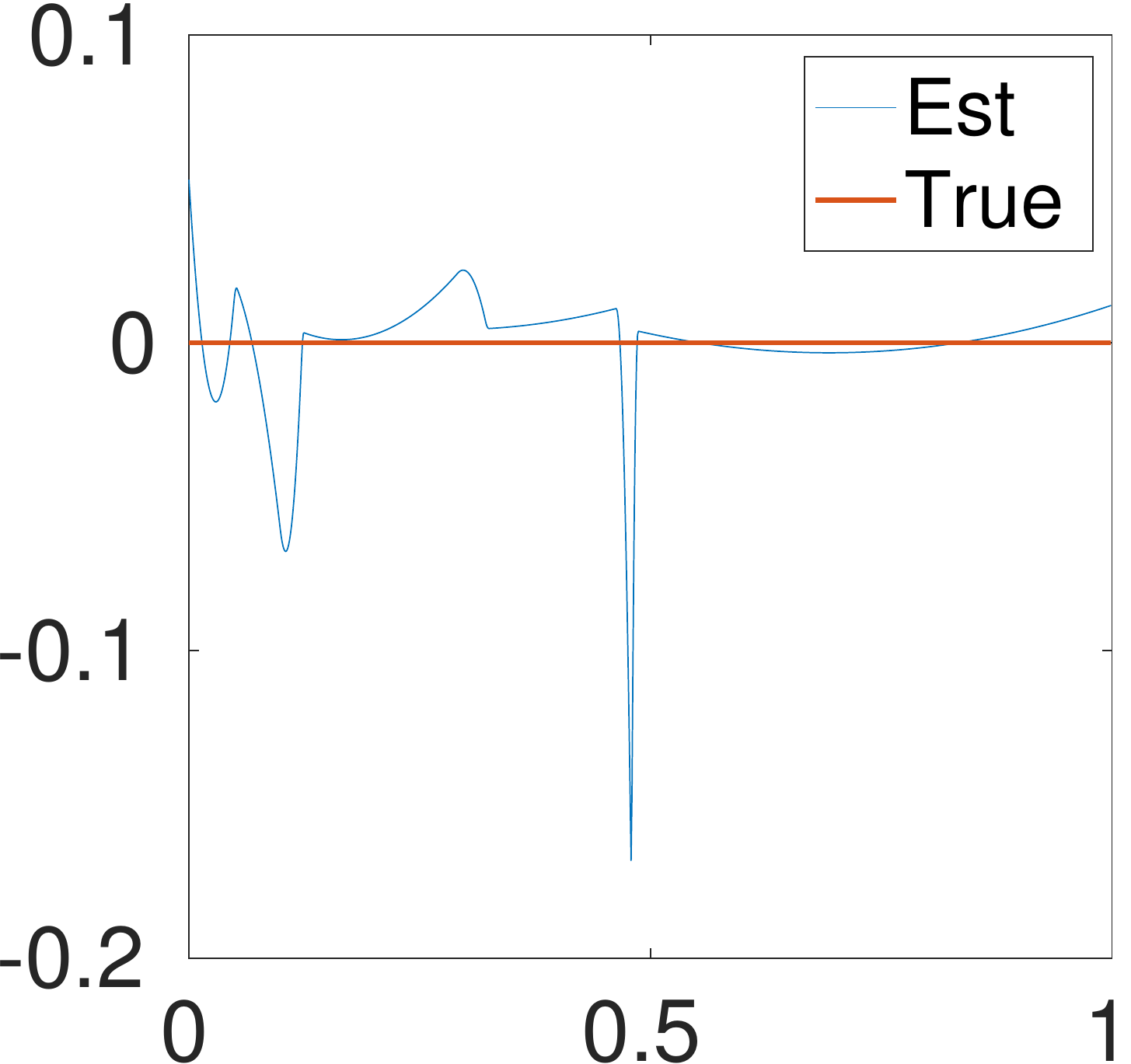}   &
      \includegraphics[width=1.1in]{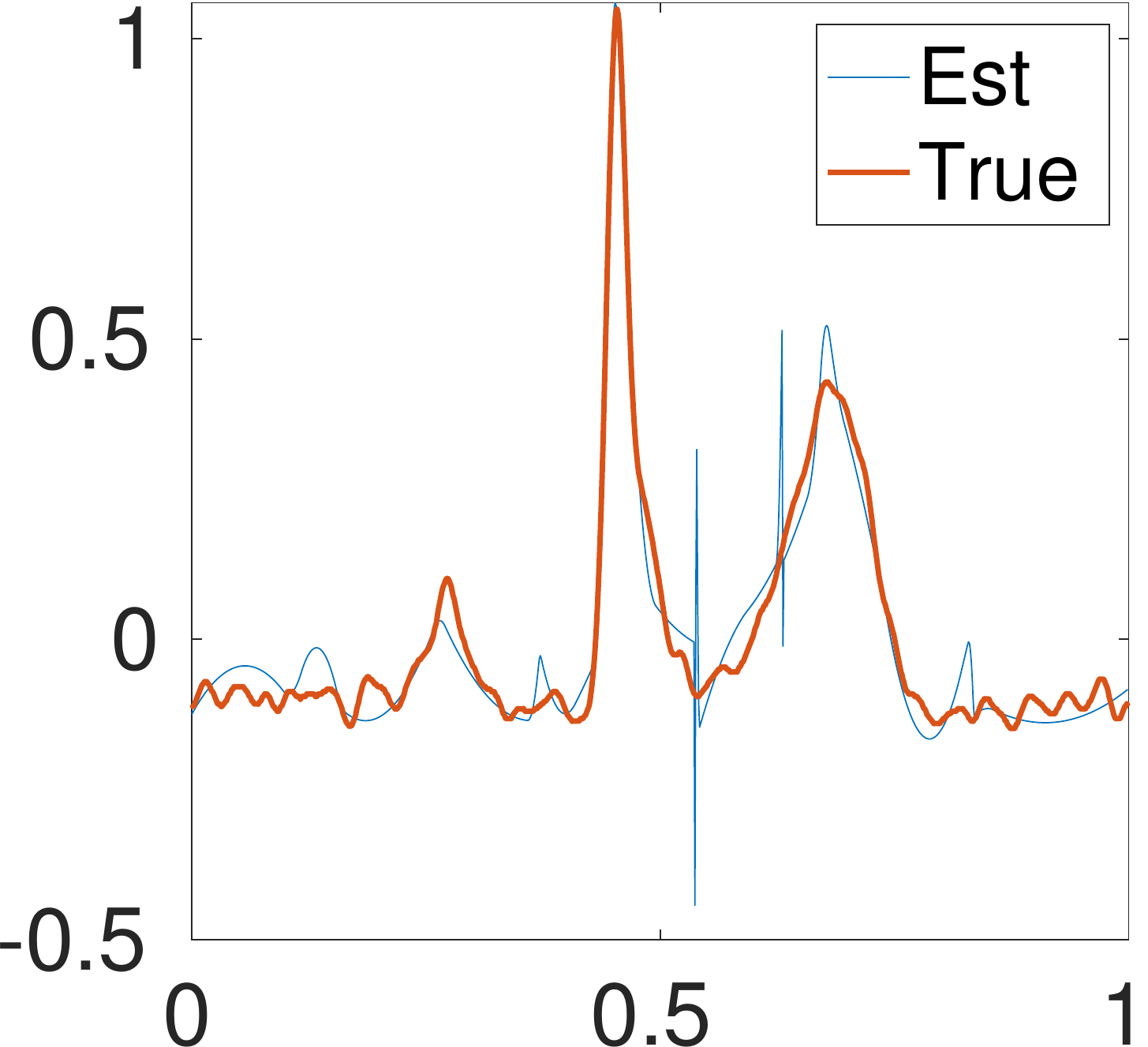}   &
      \includegraphics[width=1in]{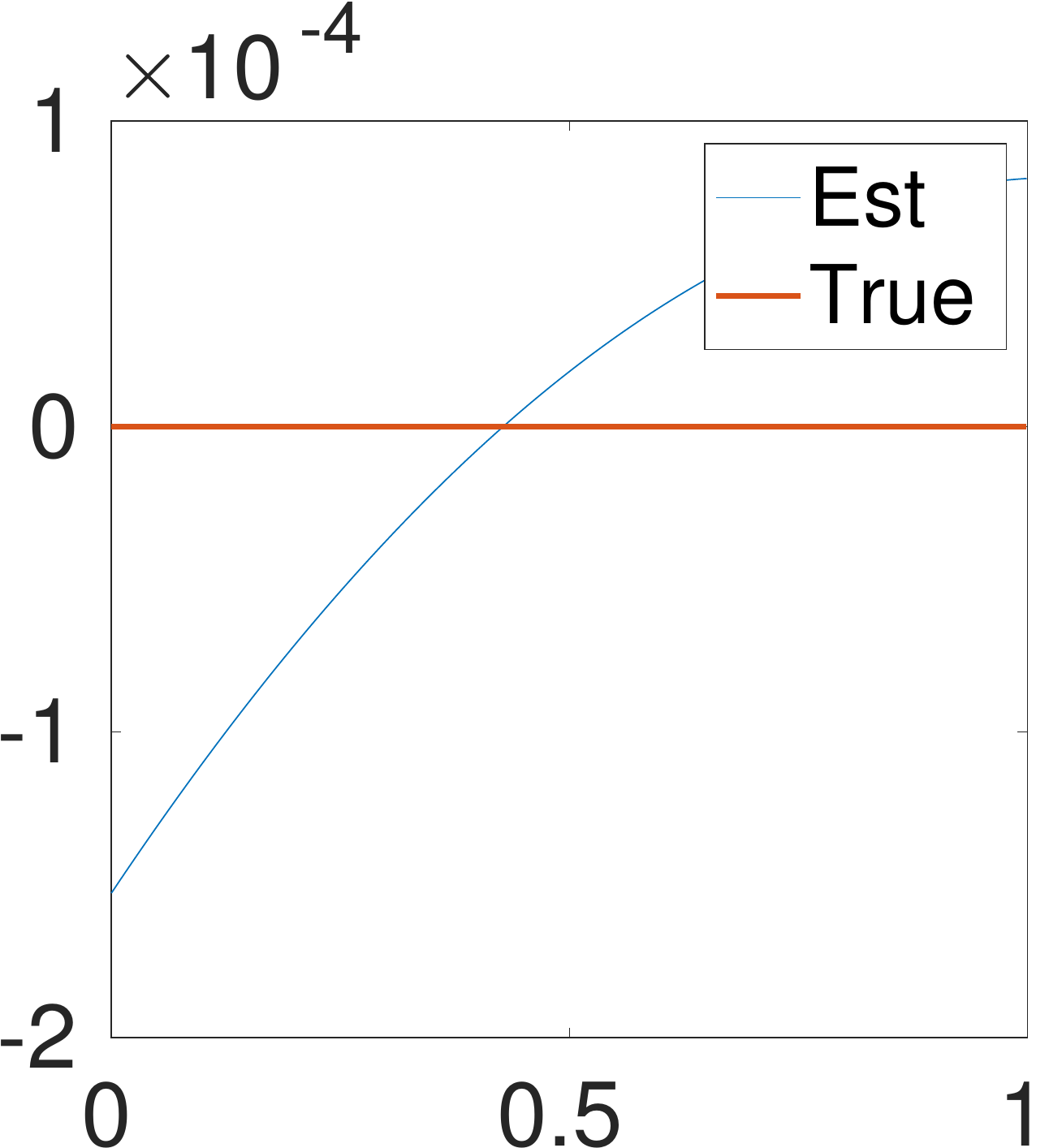}   
    \end{tabular}
  \end{center}
  \caption{Top row: estimated shape functions $a_{0,1}s_{c0,1}(2\pi t)$, $a_{1,1}s_{c1,1}(2\pi t)$, $a_{-1,1}s_{c-1,1}(2\pi t)$, $b_{1,1}s_{s1,1}(2\pi t)$, and $b_{-1,1}s_{s-1,1}(2\pi t)$ of $f_1(t)+ns$, where $f_1(t)$ is given in \eqref{eq:f1_4} and $ns$ is Gaussian random noise with variance $2.25$. Bottom row: estimated shape functions $a_{0,2}s_{c0,2}(2\pi t)$, $a_{1,2}s_{c1,2}(2\pi t)$, $a_{-1,2}s_{c-1,2}(2\pi t)$, $b_{1,2}s_{s1,2}(2\pi t)$, and $b_{-1,2}s_{s-1,2}(2\pi t)$ of $f_2(t)+ns$, where $f_2(t)$ is given in \eqref{eq:f2_4} and $ns$ is Gaussian random noise with variance $2.25$.}
\label{fig:4ns}
\end{figure}

%

\subsection{Multiresolution intrinsic mode functions in real data}

In this section, time series in real application are provided to support the model of multiresolution intrinsic mode functions (MIMFs). The first example is an ECG record from a normal subject and the second example is a motion-contaminated ECG record. The reader is referred to \url{https://www.physionet.org/physiobank/database/} for more details about the ECG data. We compute the band-limited multiresolution approximations of the first example and visualize them in Figure \ref{fig:10_1}, \ref{fig:10_2}, and \ref{fig:10_3}; the band-limited multiresolution approximations of the second example are plotted in Figure \ref{fig:11_1}, \ref{fig:11_2}, and \ref{fig:11_3}. Note that when the bandwidth of the multiresolution approximation increases, the approximation error decreases, and finer variation of the time series can be captured. These observations support the model of MIMF as a superposition of several oscillatory components. In particular, the results in Figure \ref{fig:10_2} shows that, as the bandwidth increases, the variation of the time-varying amplitude of the signal has been captured in the high frequency components of the MIMF. Figure \ref{fig:10_3} and \ref{fig:11_3} show the first five shape functions of these two examples, respectively; all shape functions vary a lot at different level of resolution. The actual time-varying shape of an ECG signals we see in the raw data is not exactly any single shape function in the shape function series; they are actually the results of all shape functions in the shape function series. This completes the validation of the proposed MIMF model.

\begin{figure}[ht!]
  \begin{center}
    \begin{tabular}{c}
      \includegraphics[width=6in]{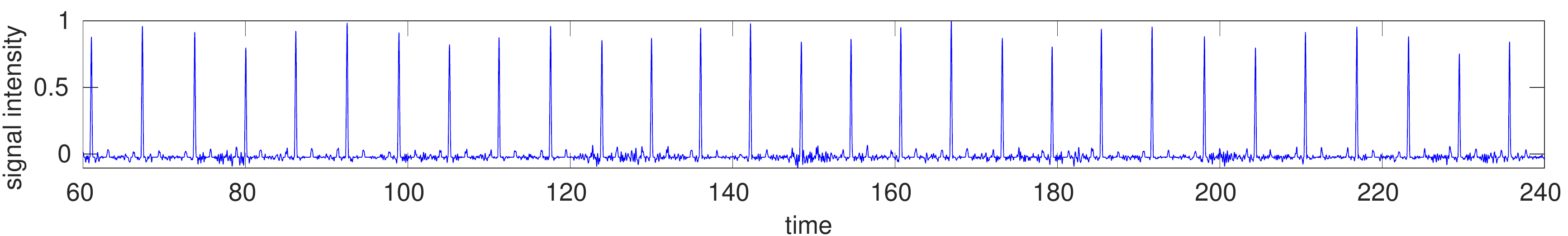}  \\
      An ECG record from a normal subject $f(t)$\\
      \includegraphics[width=6in]{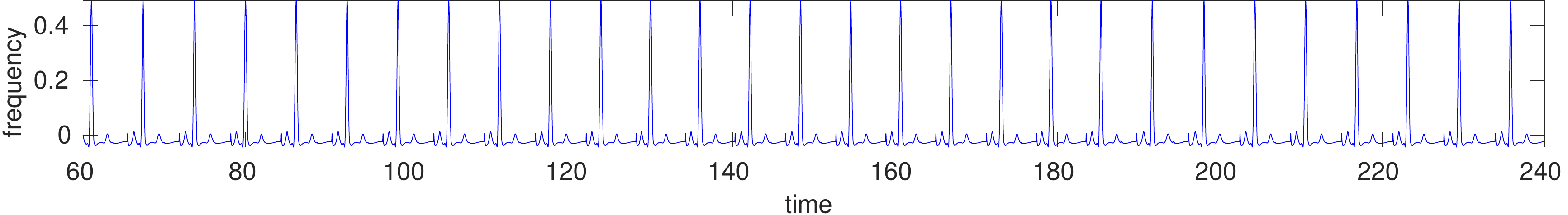}  \\
     $0$-banded multiresolution approximation $\mathcal{M}_0(f)(t)$\\
      \includegraphics[width=6in]{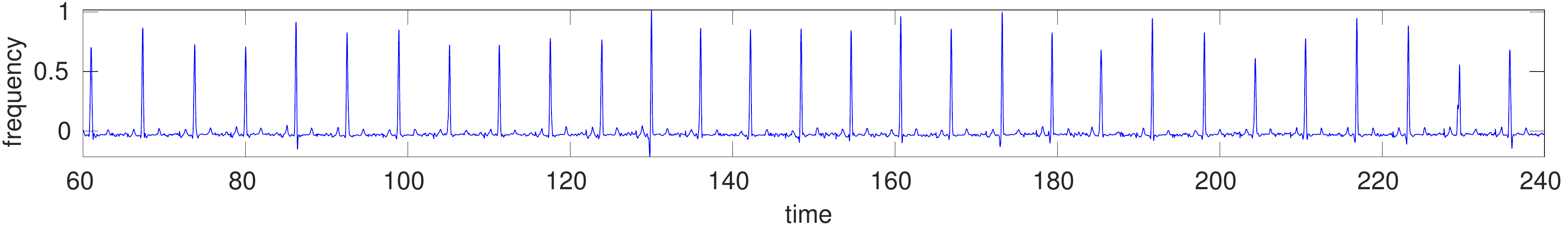}  \\
     $20$-banded multiresolution approximation $\mathcal{M}_{20}(f)(t)$\\
      \includegraphics[width=6in]{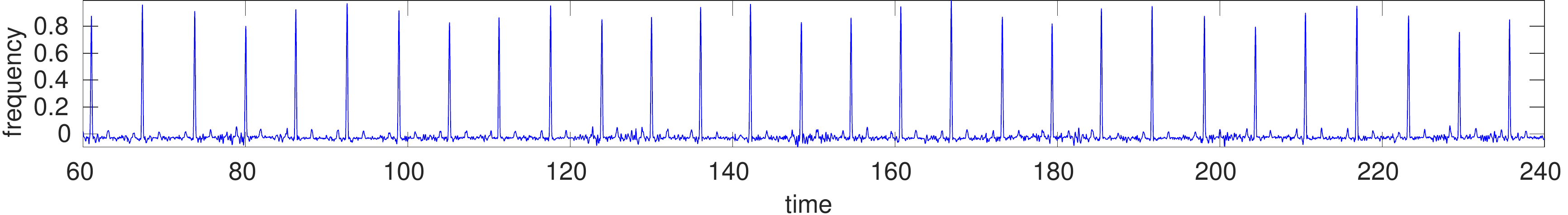}  \\
      $40$-banded multiresolution approximation $\mathcal{M}_{40}(f)(t)$
    \end{tabular}
  \end{center}
  \caption{Multiresolution approximations of an ECG record from a normal subject.}
\label{fig:10_1}
\end{figure}

\begin{figure}[ht!]
  \begin{center}
    \begin{tabular}{c}
      \includegraphics[width=6in]{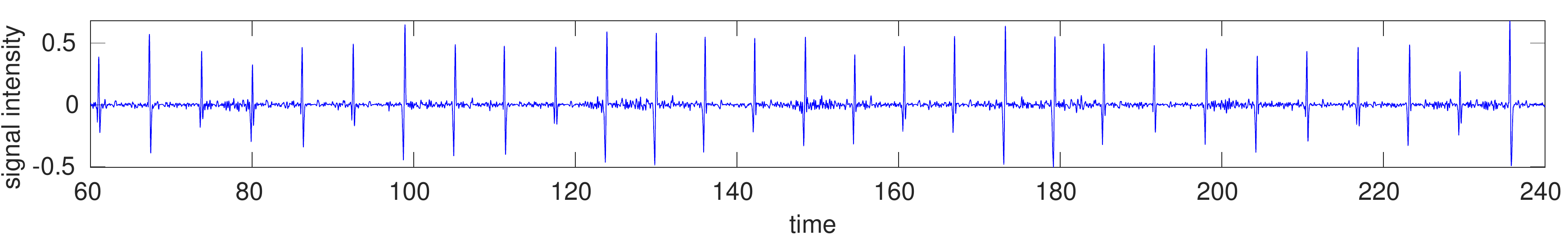}  \\
      $f(t)-\mathcal{M}_0(f)(t)$ \\
      \includegraphics[width=6in]{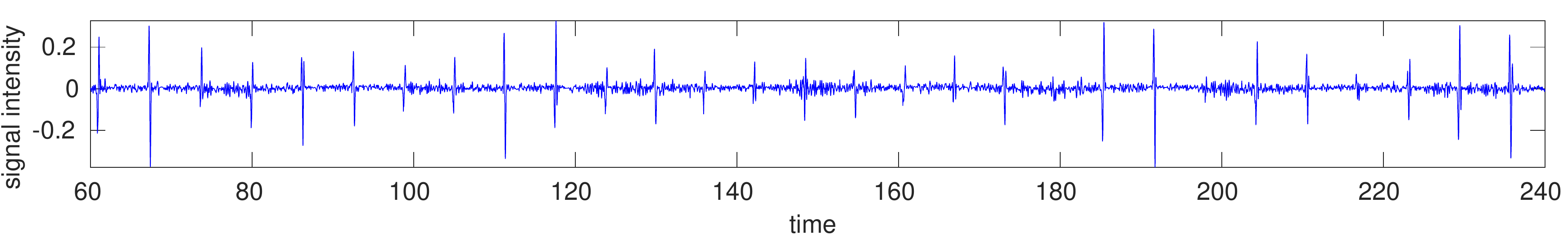}  \\
      $f(t)-\mathcal{M}_{20}(f)(t)$ \\
      \includegraphics[width=6in]{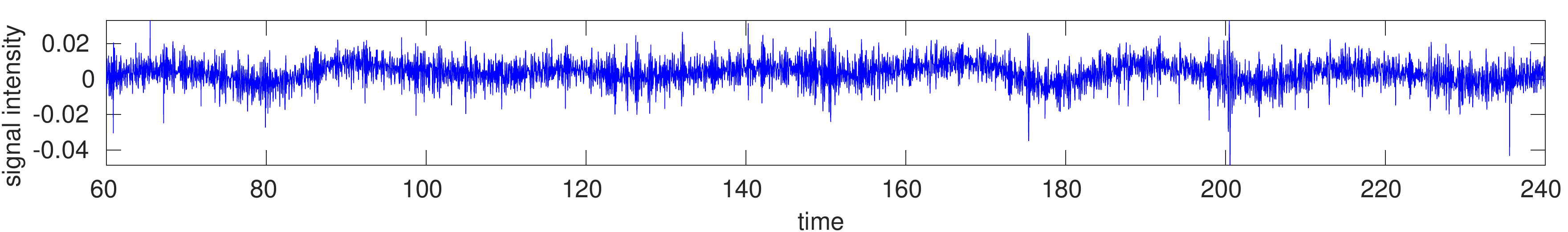}  \\
      $f(t)-\mathcal{M}_{40}(f)(t)$ 
    \end{tabular}
  \end{center}
  \caption{The residual of the multiresolution approximations of an ECG record from a normal subject in Figure \ref{fig:10_1}.}
\label{fig:10_2}
\end{figure}

\begin{figure}[ht!]
  \begin{center}
    \begin{tabular}{ccccc}
      \includegraphics[width=1.1in]{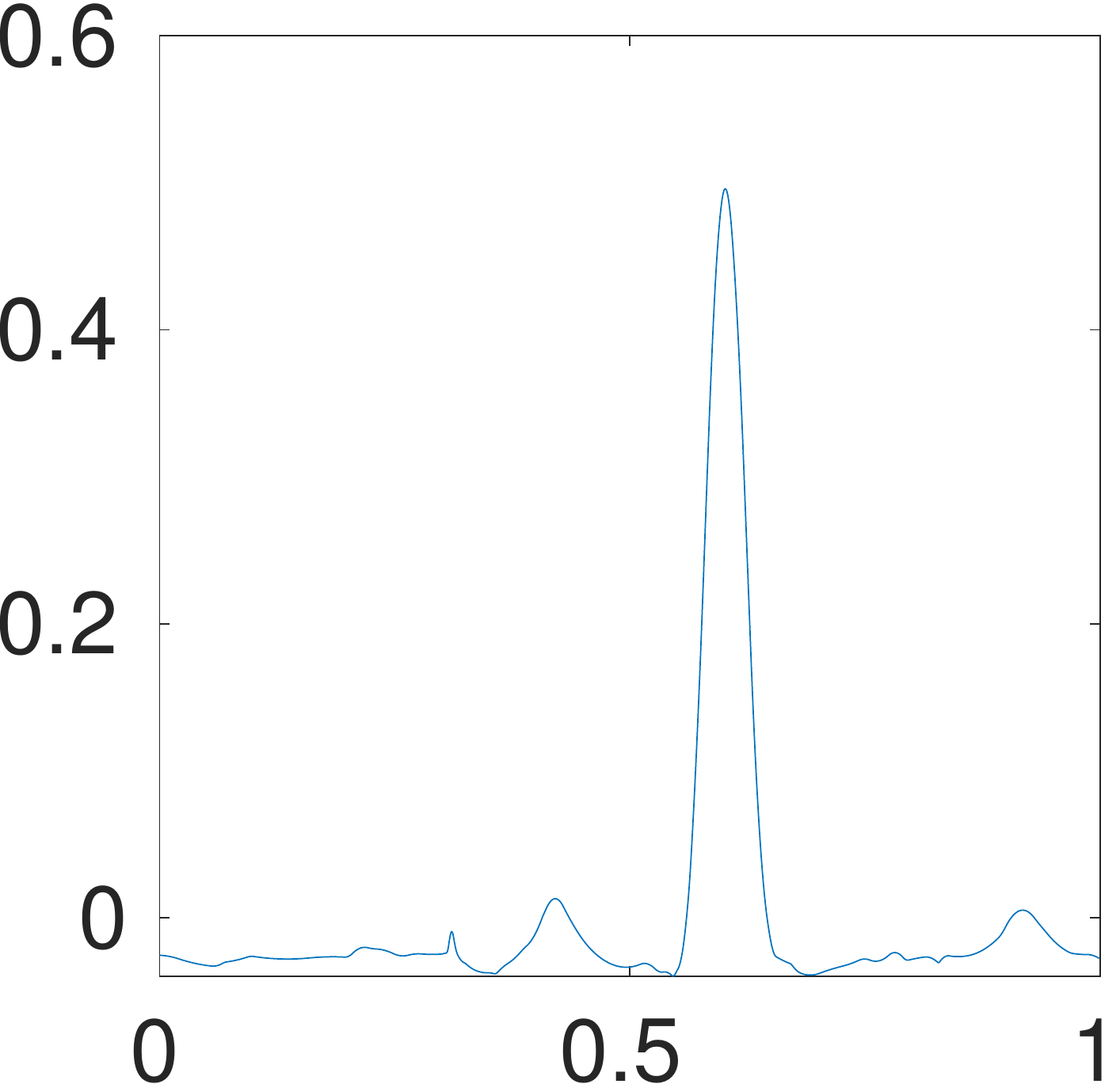}   &
      \includegraphics[width=1.1in]{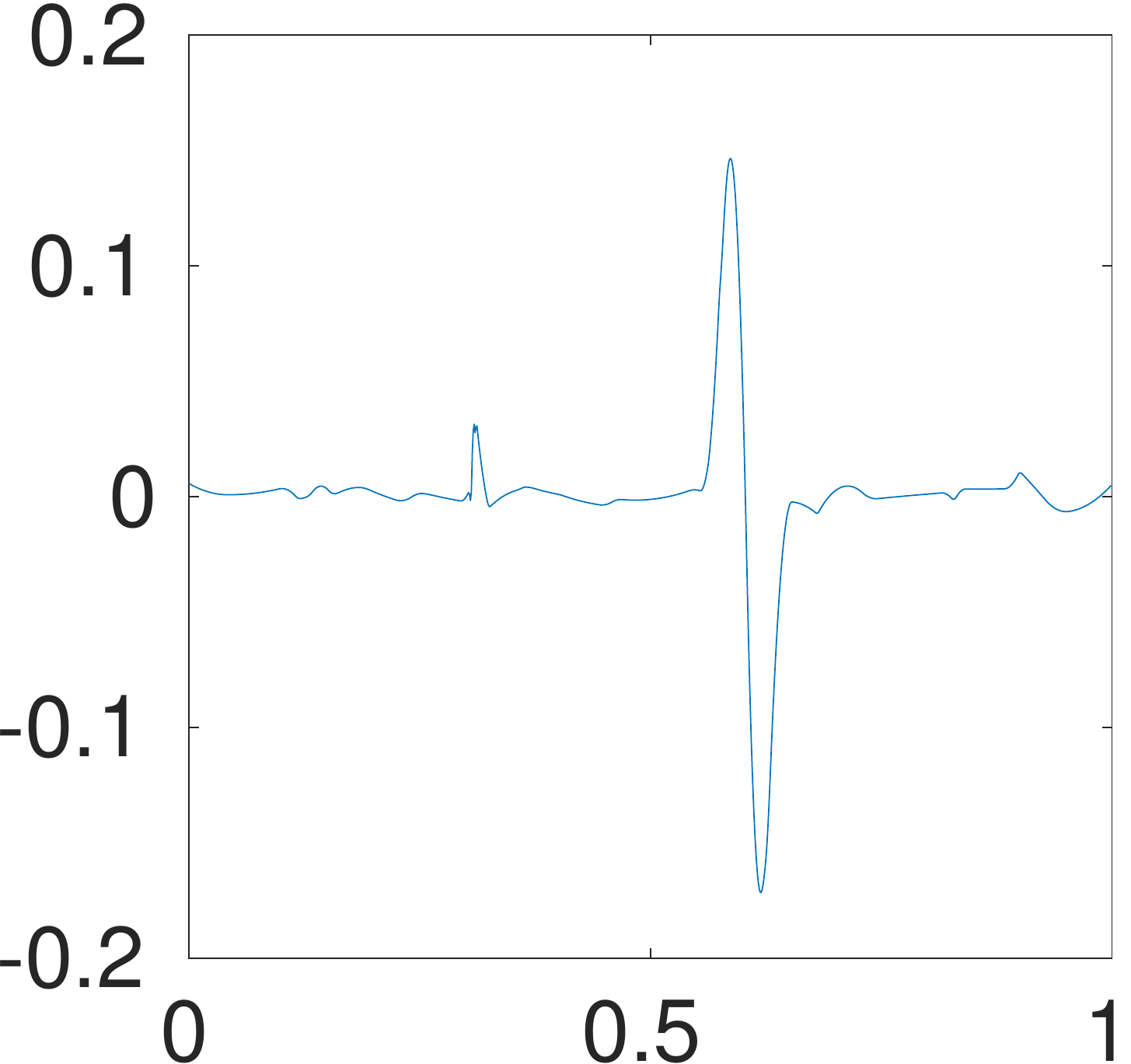}   &
      \includegraphics[width=1in]{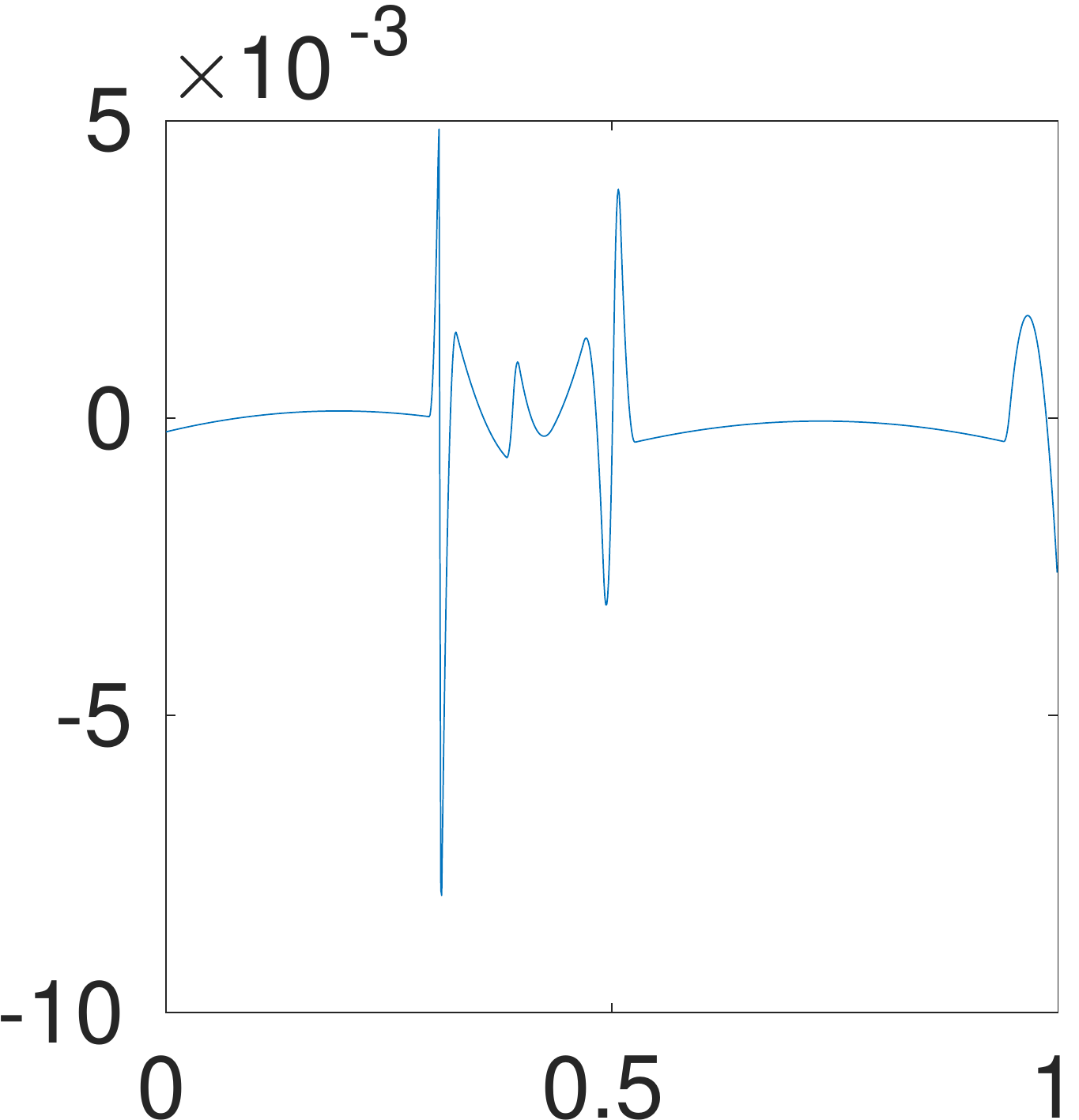}   &
      \includegraphics[width=1.1in]{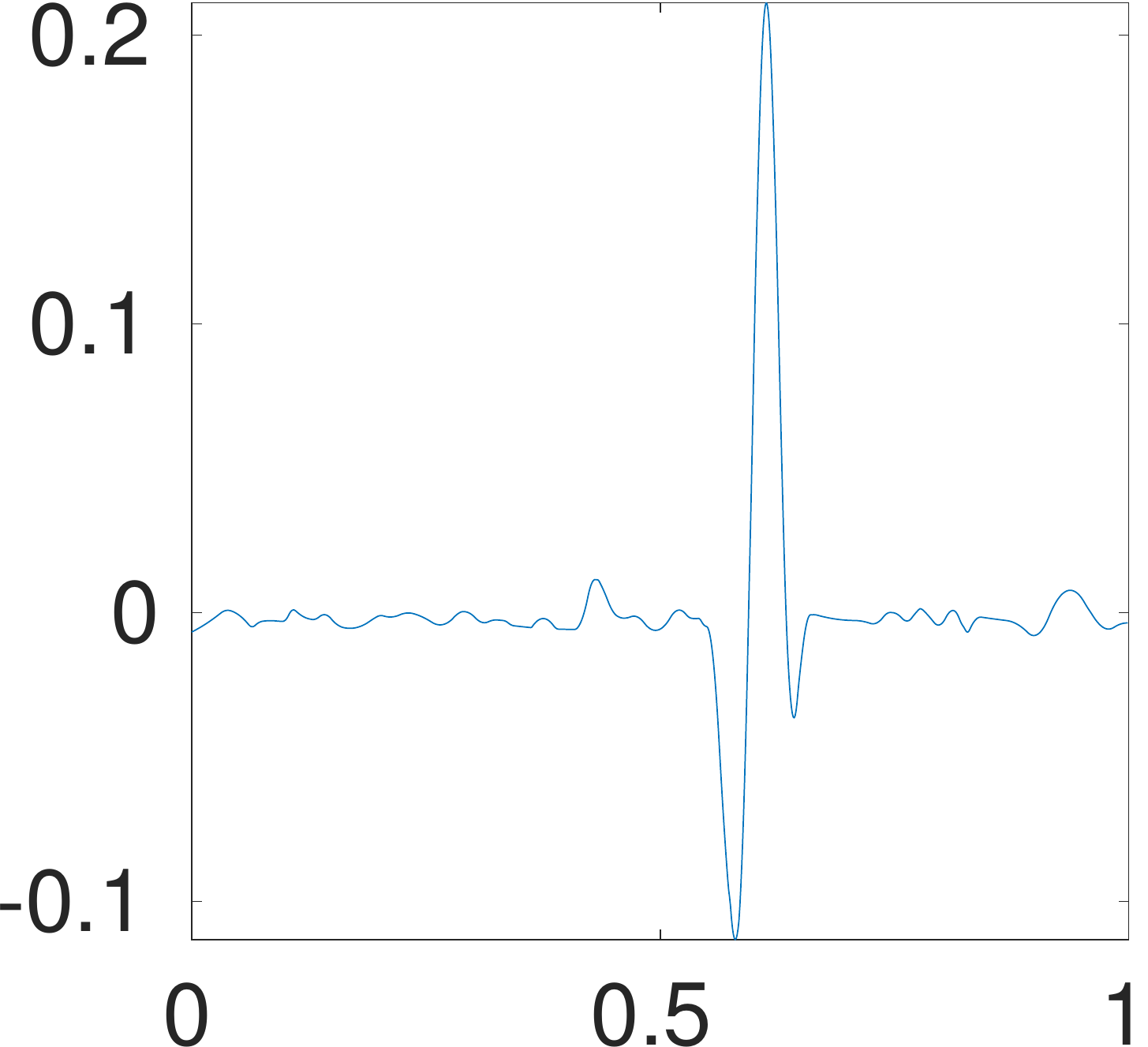}   &
      \includegraphics[width=1in]{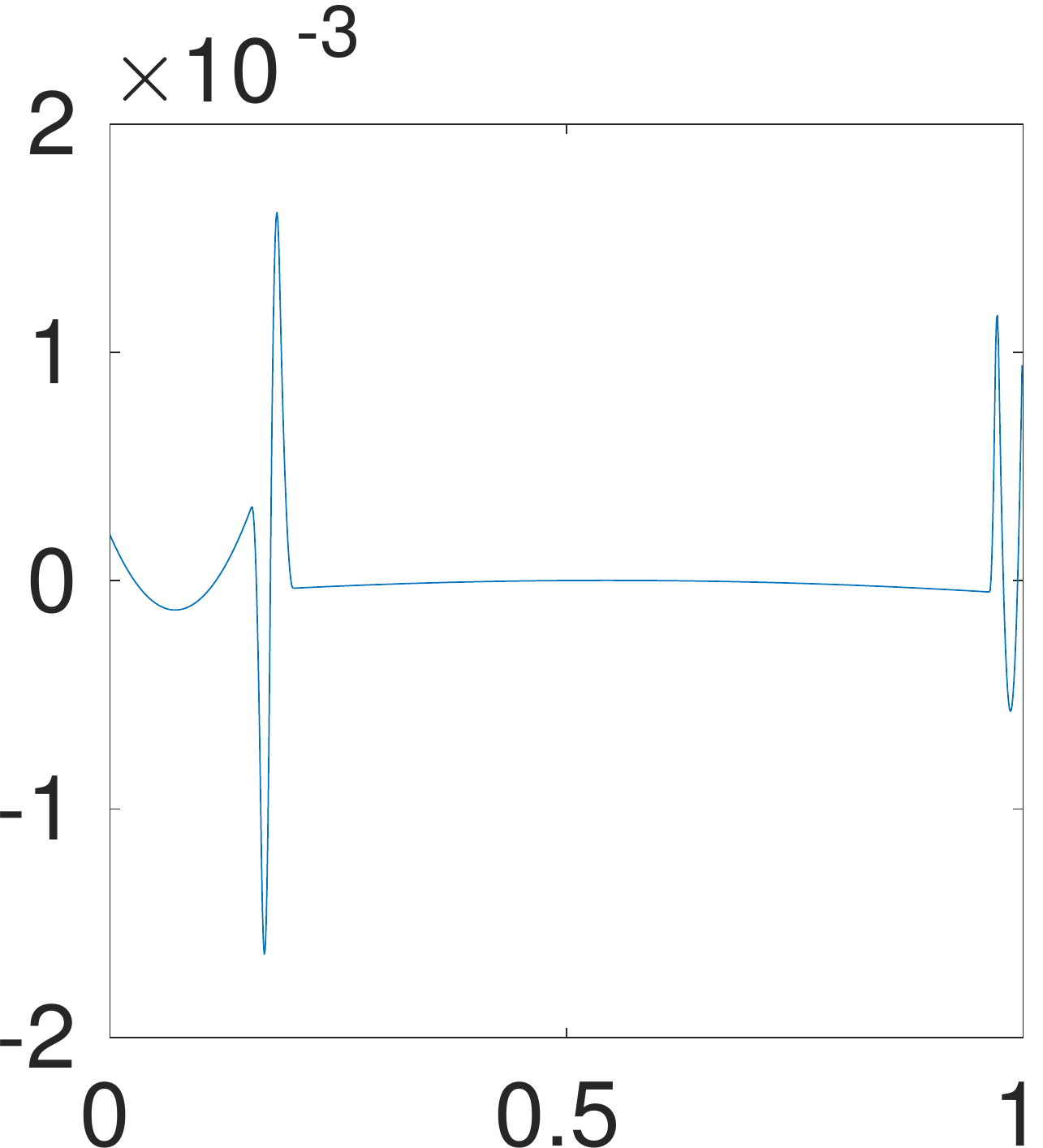}    
    \end{tabular}
  \end{center}
  \caption{Estimated shape functions $a_0 s_{c0}(t)$, $a_1 s_{c1}(t)$, $a_{-1} s_{c-1}(t)$, $b_1 s_{s1}(t)$, and $b_{-1} s_{s-1}(t)$ for the ECG signal in Figure \ref{fig:10_1}.}
\label{fig:10_3}
\end{figure}

\begin{figure}[ht!]
  \begin{center}
    \begin{tabular}{c}
      \includegraphics[width=6in]{Pictures/MMD_fig12_org.pdf}  \\
      A motion-contaminated ECG record $f(t)$\\
     \includegraphics[width=6in]{Pictures/MMD_fig12_comp_bw_0.pdf}  \\
     $0$-banded multiresolution approximation $\mathcal{M}_0(f)(t)$\\
      \includegraphics[width=6in]{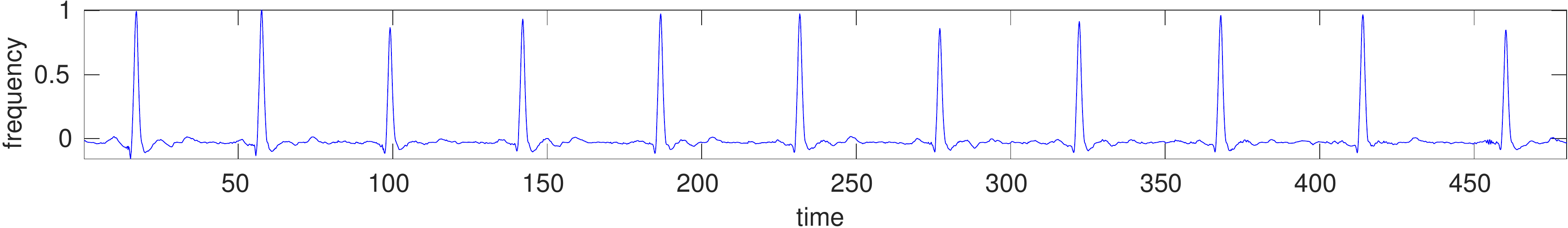}  \\
     $20$-banded multiresolution approximation $\mathcal{M}_{20}(f)(t)$\\
      \includegraphics[width=6in]{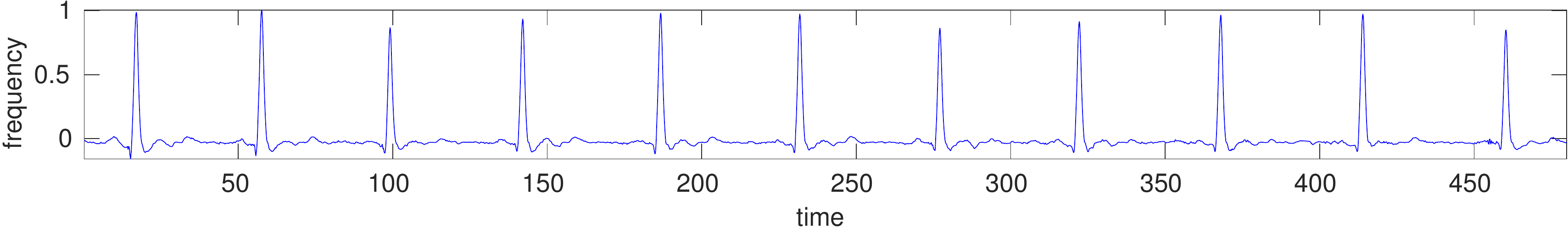}  \\
     $40$-banded multiresolution approximation $\mathcal{M}_{40}(f)(t)$
    \end{tabular}
  \end{center}
  \caption{Multiresolution approximations of a motion-contaminated ECG record.}
\label{fig:11_1}
\end{figure}

\begin{figure}[ht!]
  \begin{center}
    \begin{tabular}{c}
      \includegraphics[width=6in]{Pictures/MMD_fig12_comp_res_0.pdf}  \\
      $f(t)-\mathcal{M}_{0}(f)(t)$ \\
      \includegraphics[width=6in]{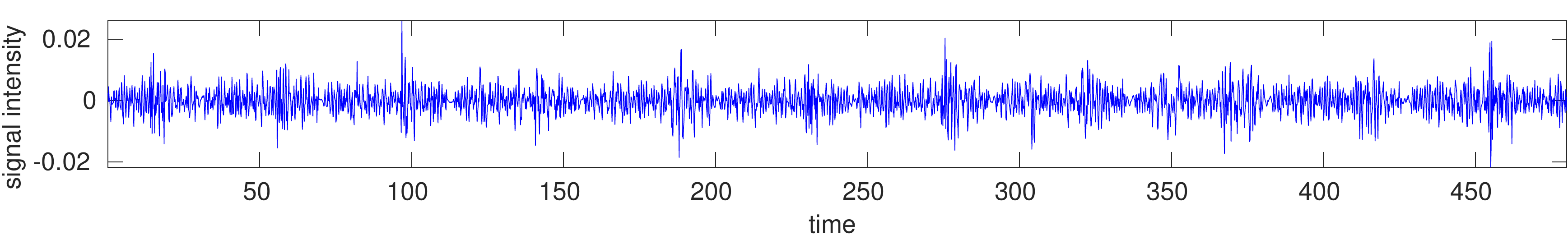}  \\
      $f(t)-\mathcal{M}_{20}(f)(t)$ \\
      \includegraphics[width=6in]{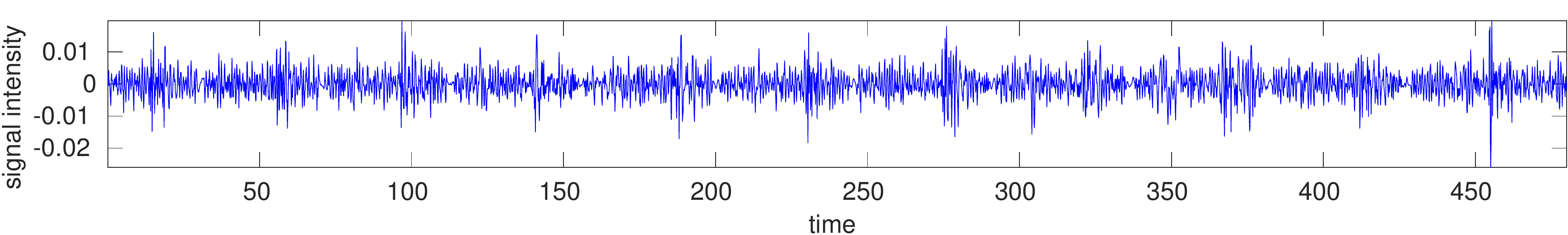}  \\
      $f(t)-\mathcal{M}_{40}(f)(t)$ 
    \end{tabular}
  \end{center}
  \caption{The residual of the multiresolution approximations of a motion-contaminated ECG record in Figure \ref{fig:11_1}.}
\label{fig:11_2}
\end{figure}

\begin{figure}[ht!]
  \begin{center}
    \begin{tabular}{ccccc}
      \includegraphics[width=1.1in]{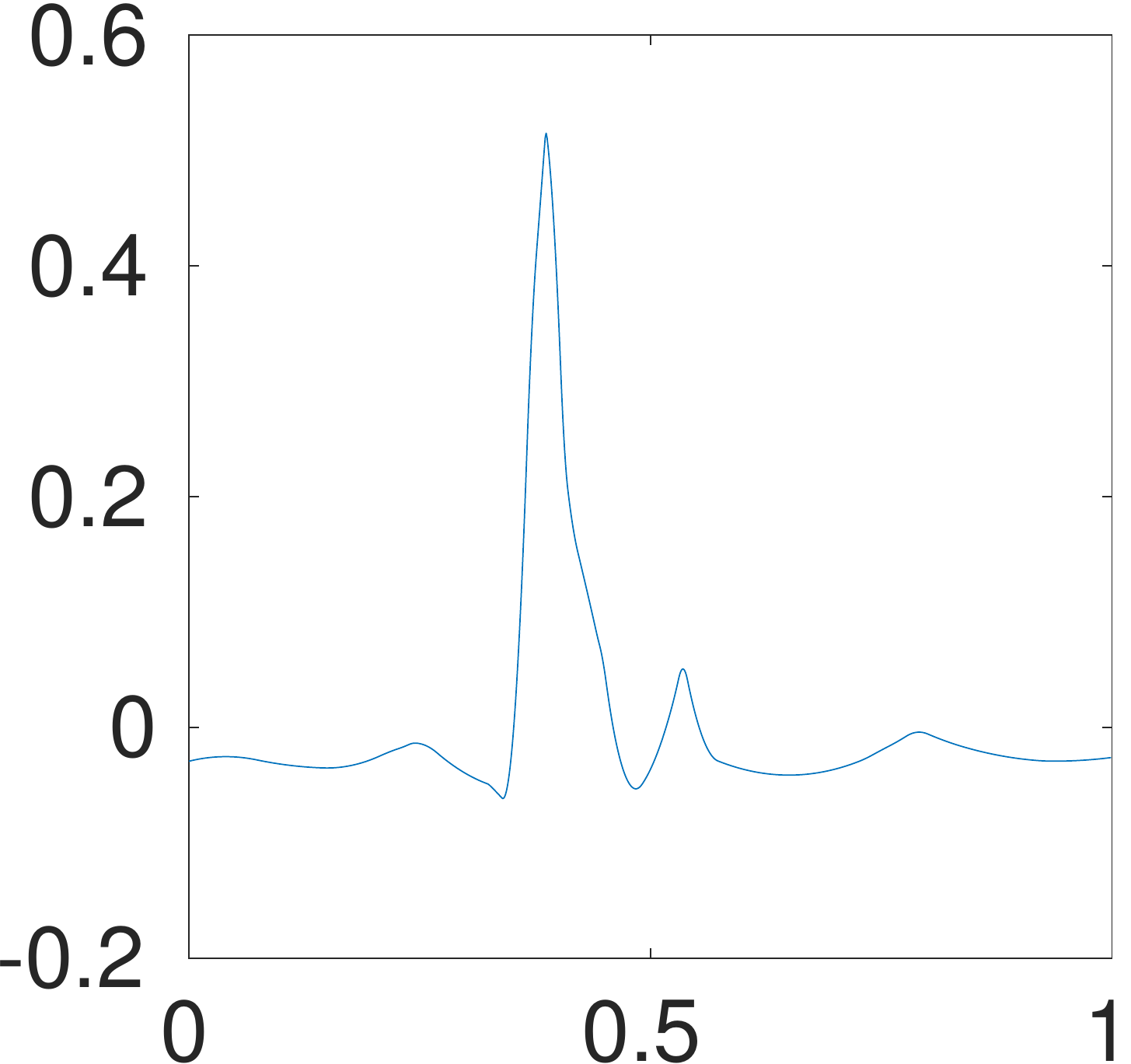}   &
      \includegraphics[width=1.1in]{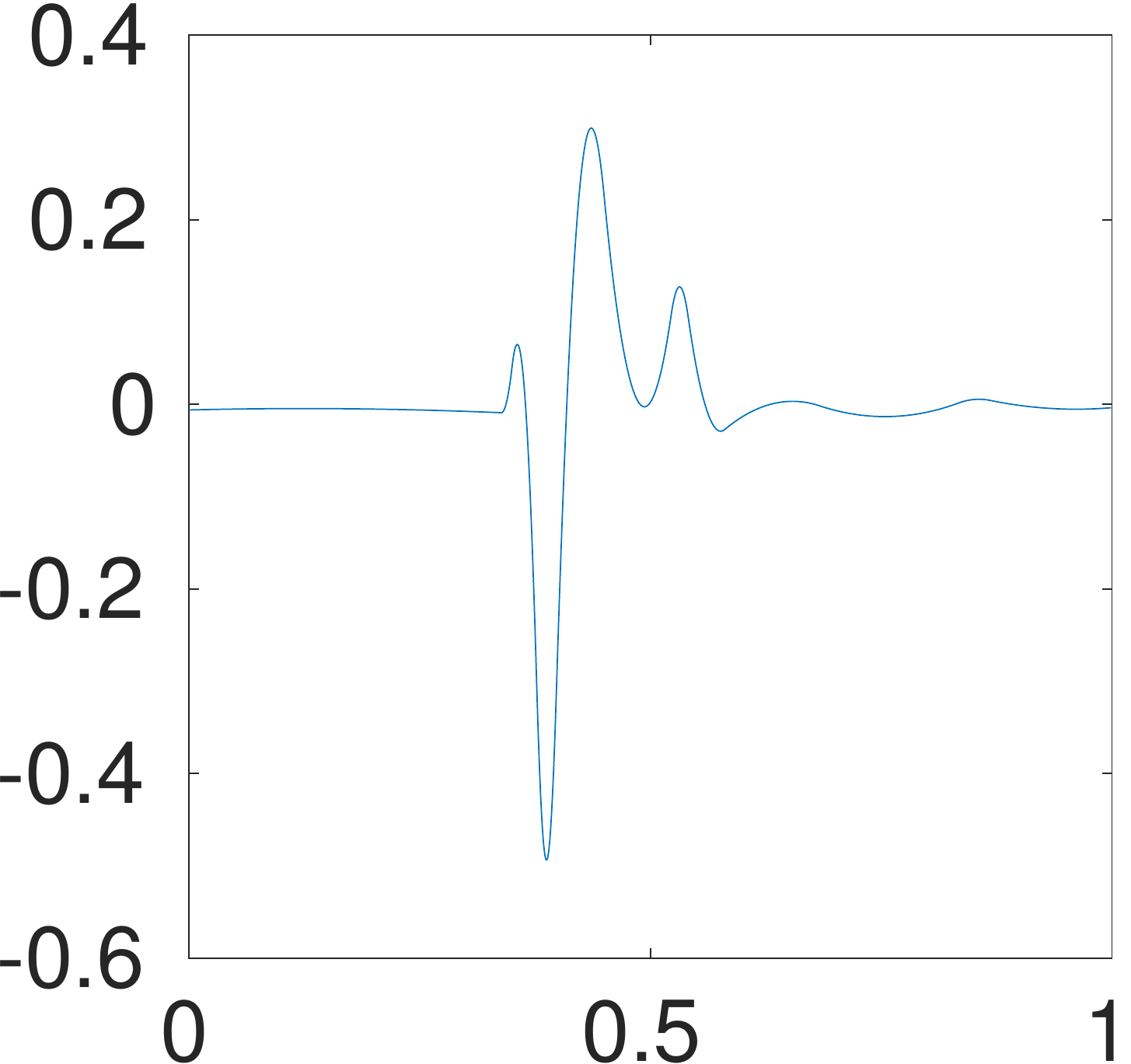}   &
      \includegraphics[width=1in]{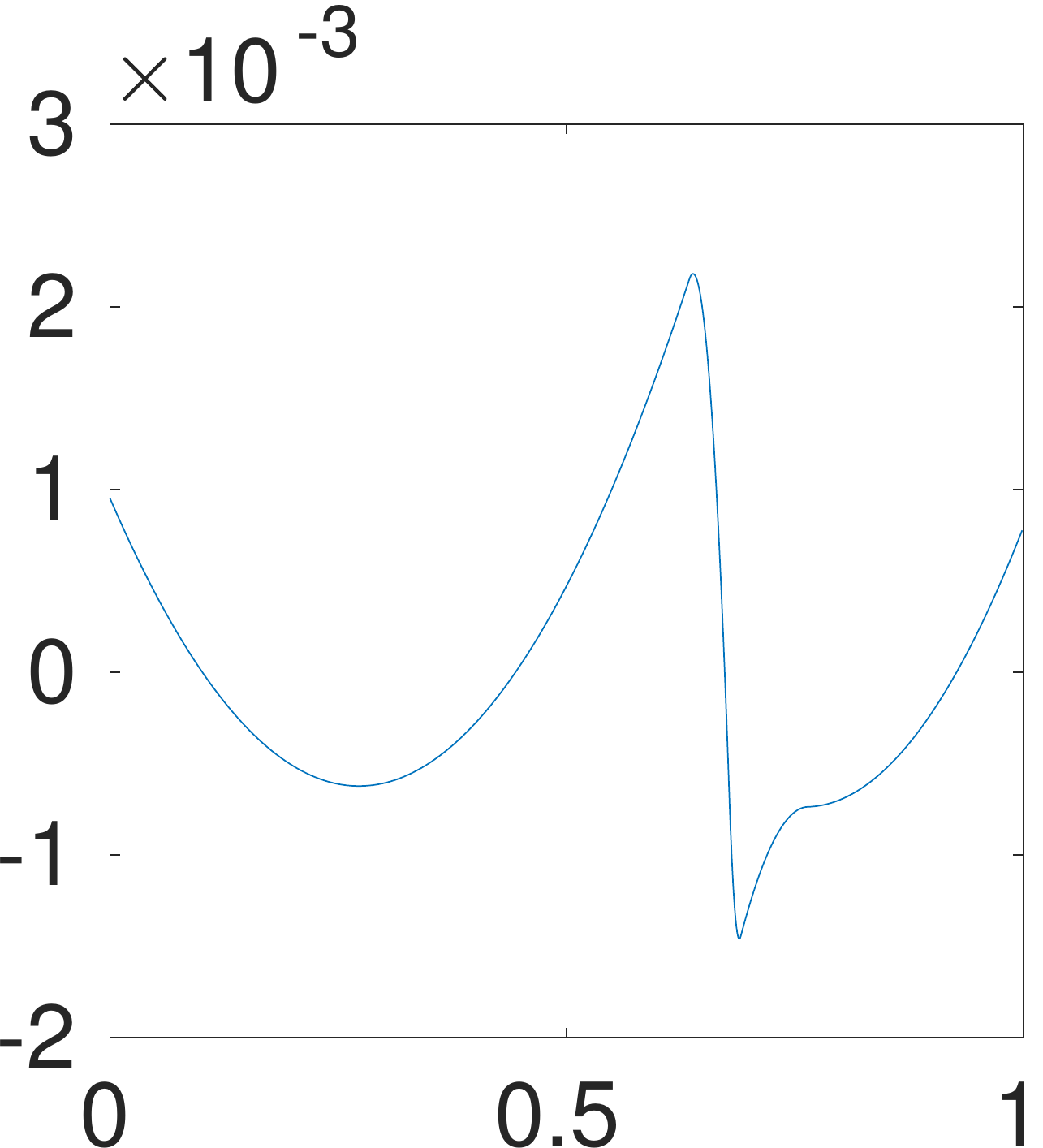}   &
      \includegraphics[width=1.1in]{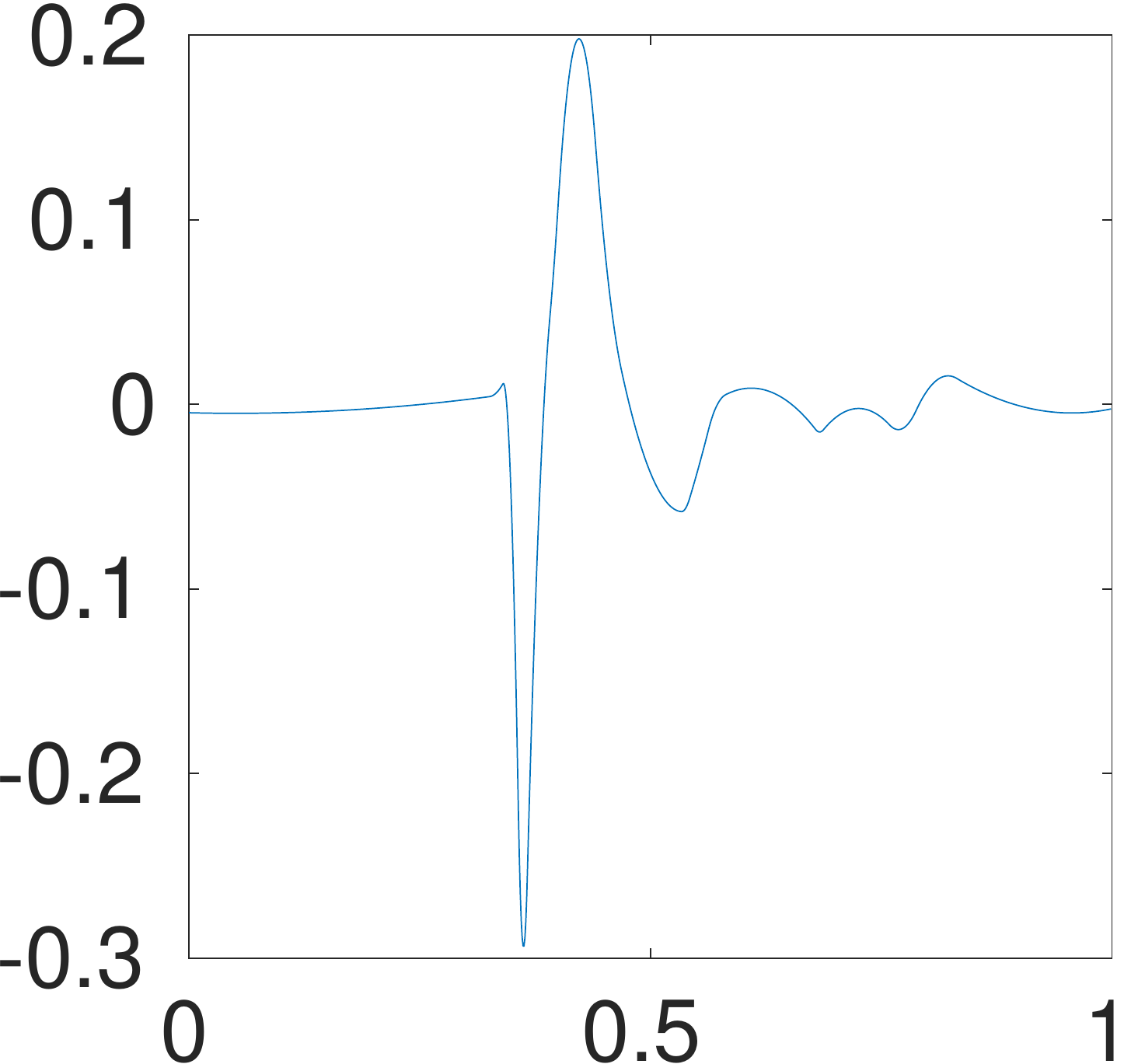}   &
      \includegraphics[width=1in]{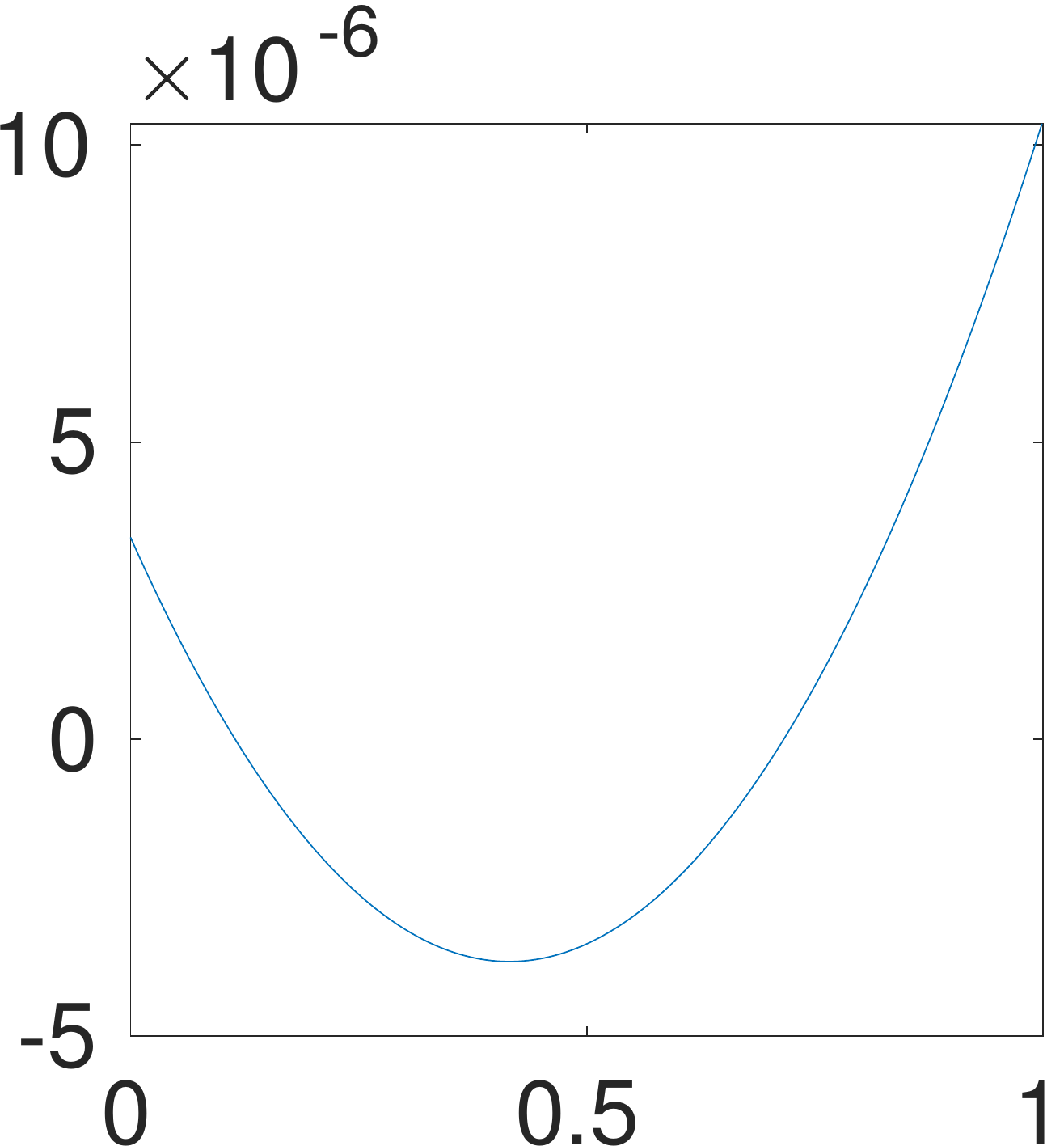}    
    \end{tabular}
  \end{center}
  \caption{Estimated shape functions $a_0 s_{c0}(t)$, $a_1 s_{c1}(t)$, $a_{-1} s_{c-1}(t)$, $b_1 s_{s1}(t)$, and $b_{-1} s_{s-1}(t)$ for the ECG signal in Figure \ref{fig:11_1}.}
\label{fig:11_3}
\end{figure}

\subsection{Real example of multiresolution mode decomposition}

In this section, we provide a real example to demonstrate the application of the multiresolution mode decomposition. This is an example of photoplethysmography (PPG)\footnote{From \url{http://www.capnobase.org}.} that contains the hemodynamical information as well as the respiration information. Hence, the PPG signal essentially contains two MIMFs. In this example, the instantaneous frequencies and phases are not known and they are estimated via the synchrosqueezed transform in \cite{1DSSWPT}. Figure \ref{fig:15_1} shows the estimated fundamental instantaneous frequencies of the respiratory and cardiac cycles. Inputing their corresponding instantaneous phases into the MMD algorithm, the PPG signal is separated into a respiratory MIMF and a cardiac MIMF as shown in Figure \ref{fig:15_2}; their leading multiresolution shape functions are shown in Figure \ref{fig:15_3}. 

The last two panels of Figure \ref{fig:15_2} shows that the PPG signal has been completely separated into two MIMFs; the residual signal only contains noise, a smooth trend, and some sharp changes that are not correlated to the oscillation in MIMFs. It is worth to exploring the application of MIMFs and MMD for fault detection in health data. For example, the detection of abnormal ECG waveforms is important to cardiac disease diagnosis \cite{lippincott2011ecg,hyper}; the abnormality can be identified by detecting the sharp changes in the residual signal that are not correlated to the normal oscillation pattern. The second panel shows that the MIMF model can characterize time-varying shape functions, while the third panel shows that the MIMF model can capture the time-varying amplitude functions. 

\begin{figure}[ht!]
  \begin{center}
    \begin{tabular}{c}
  \hspace{-0.5cm}  \includegraphics[width=6.1in]{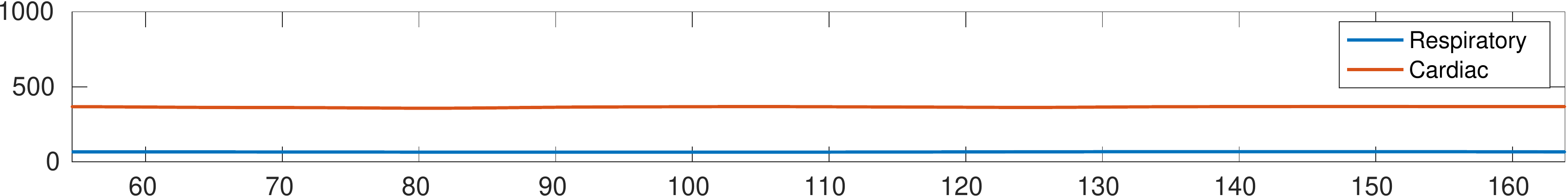} 
    \end{tabular}
  \end{center}
  \caption{Estimated fundamental instantaneous frequencies of the real PPG signal in the first panel of Figure \ref{fig:15_2} by the synchrosqueezed transform.}
\label{fig:15_1}
\end{figure}

\begin{figure}[ht!]
  \begin{center}
    \begin{tabular}{c}
      \includegraphics[width=6in]{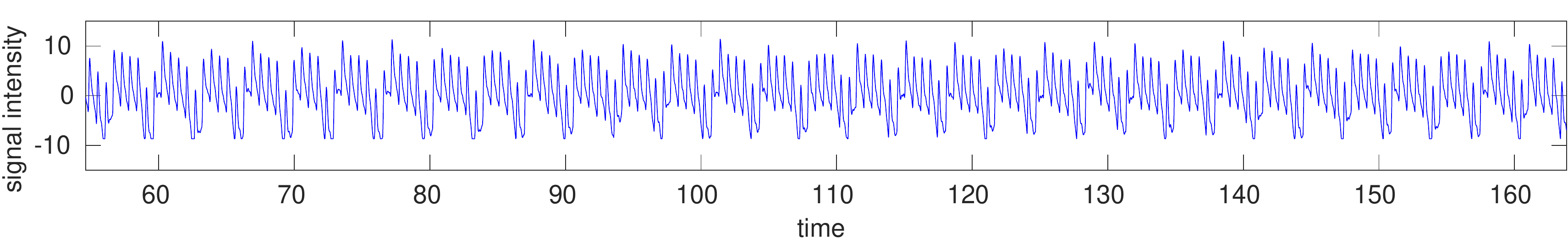}  \\
   \includegraphics[width=6in]{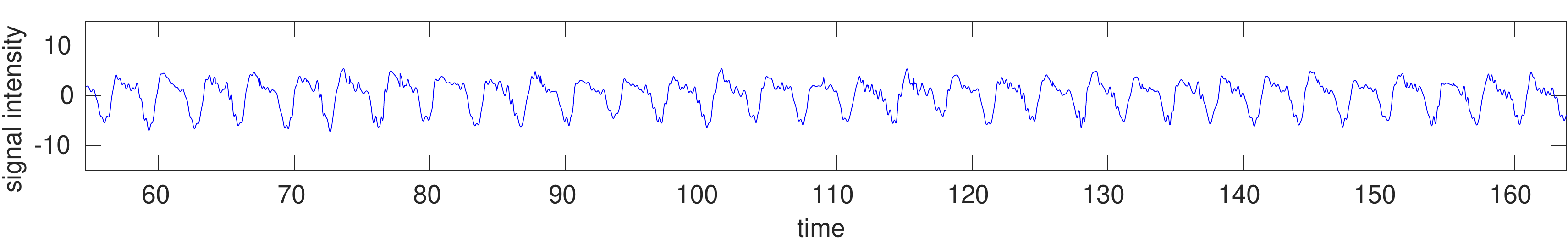}   \\
   \includegraphics[width=6in]{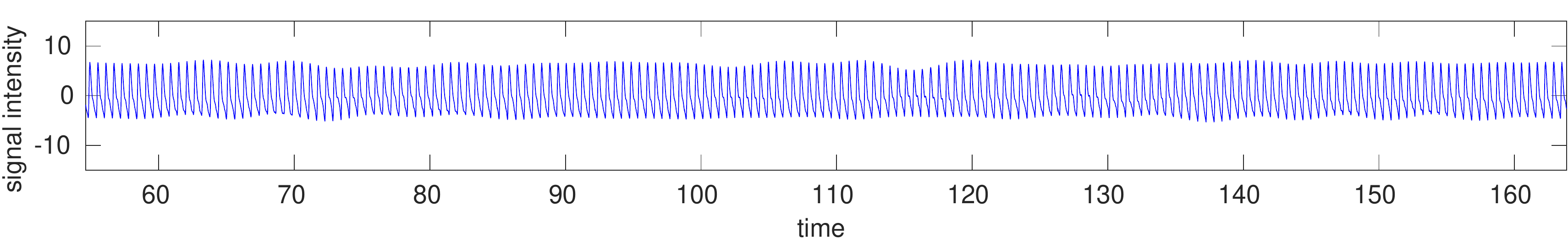}  \\
      \includegraphics[width=6in]{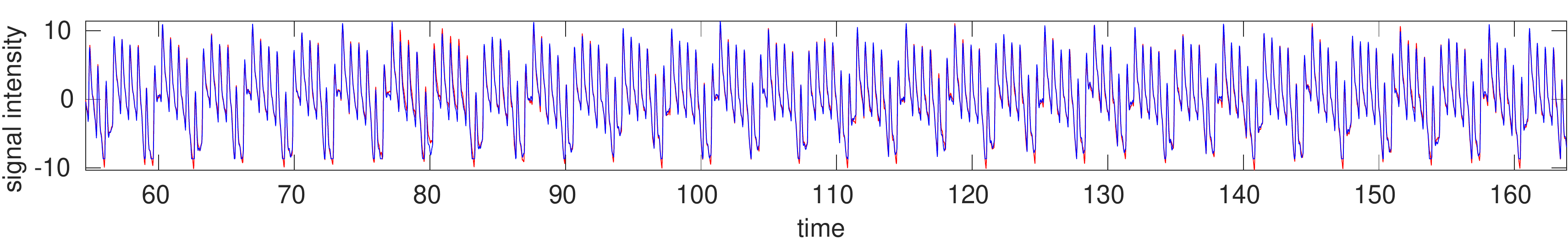}  \\
      \includegraphics[width=6in]{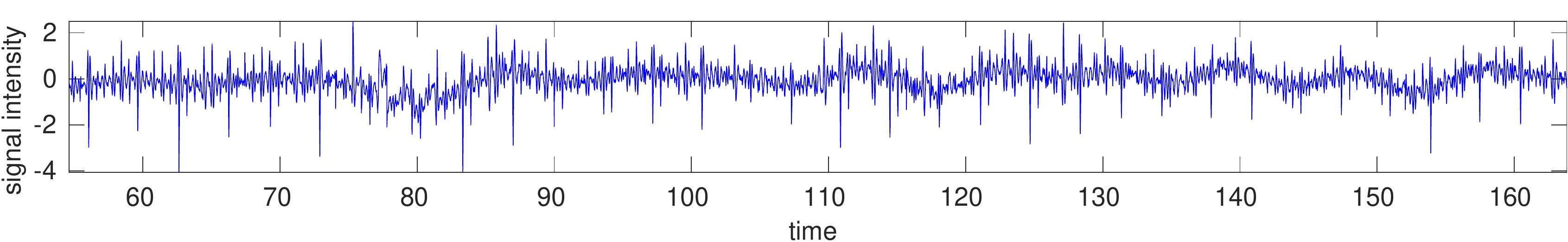}  
    \end{tabular}
  \end{center}
  \caption{First panel: the raw PPG signal $f(t)$. Second panel: the respiratory MIMF $f_1(t)$. Third panel: the cardiac MIMF $f_2(t)$. Fourth panel: the summation of the respiratory and cardiac MIMFs $f_1(t)+f_2(t)$ (red) compared to the raw PPG signal $f(t)$ (blue). The fifth panel: the residual signal $f(t)-f_1(t)-f_2(t)$.}
\label{fig:15_2}
\end{figure}

\begin{figure}[ht!]
  \begin{center}
    \begin{tabular}{ccccc}
      \includegraphics[width=0.975in]{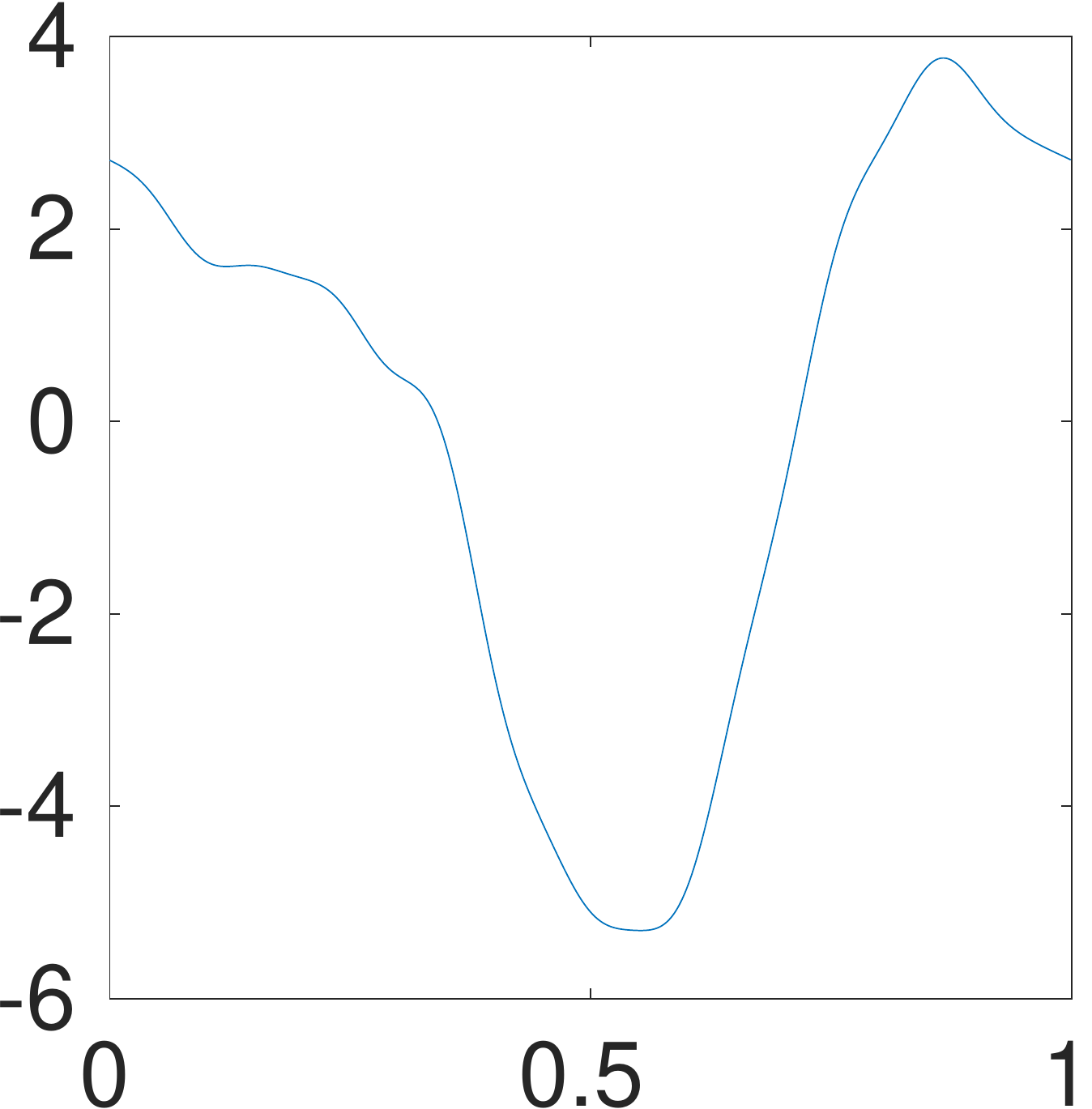}   &
      \includegraphics[width=1.05in]{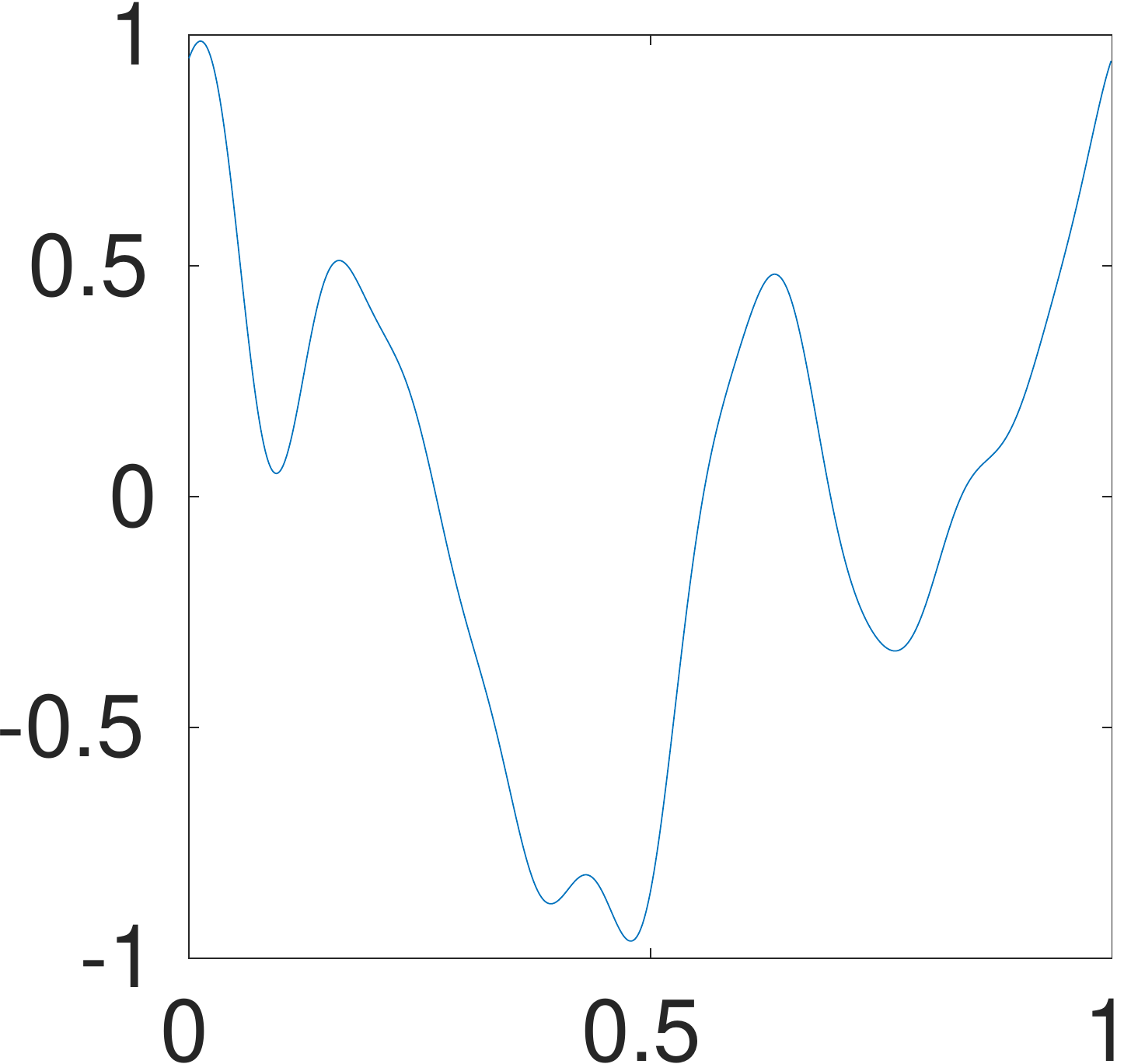}   &
      \includegraphics[width=1.1in]{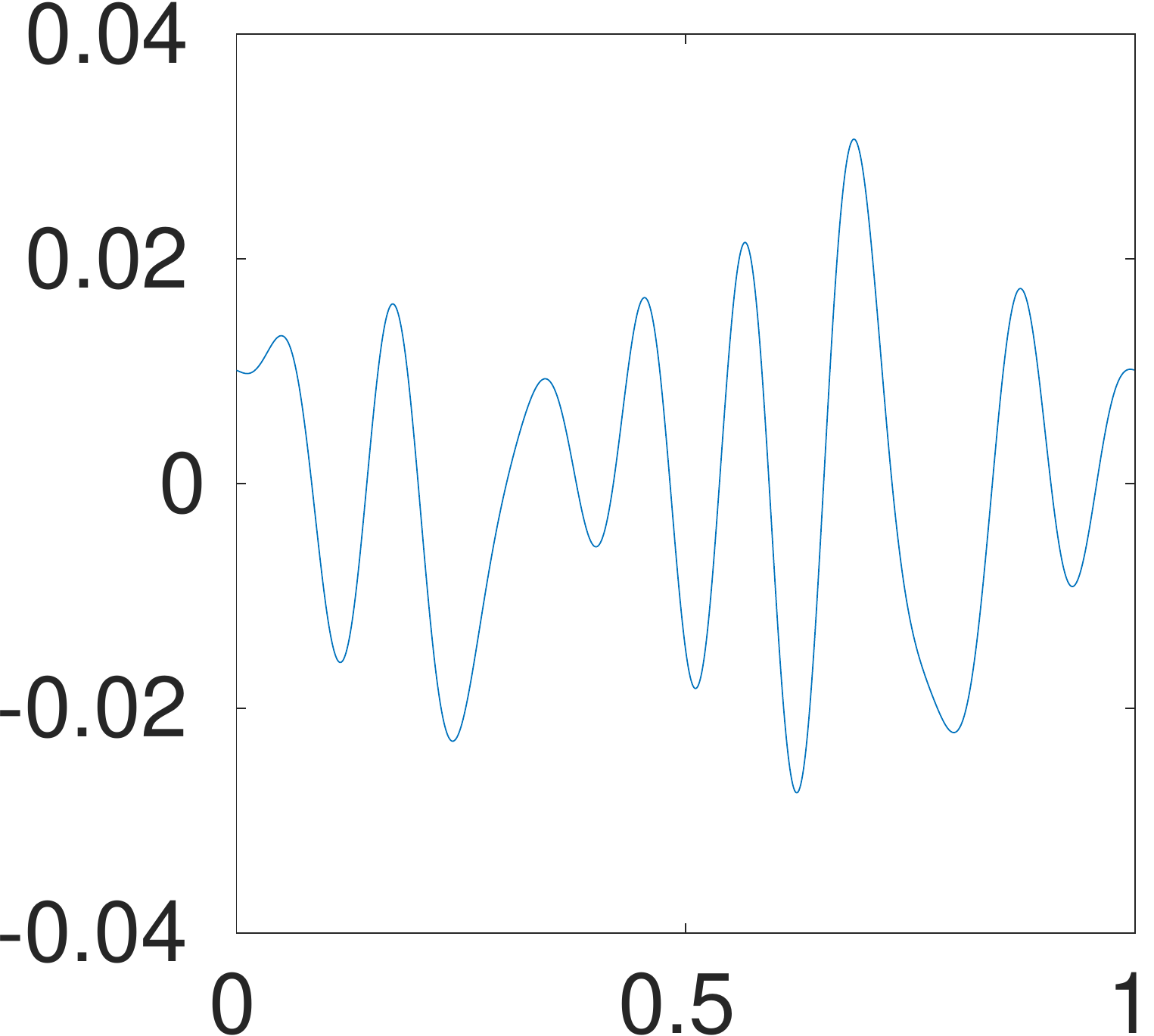}   &
      \includegraphics[width=1.05in]{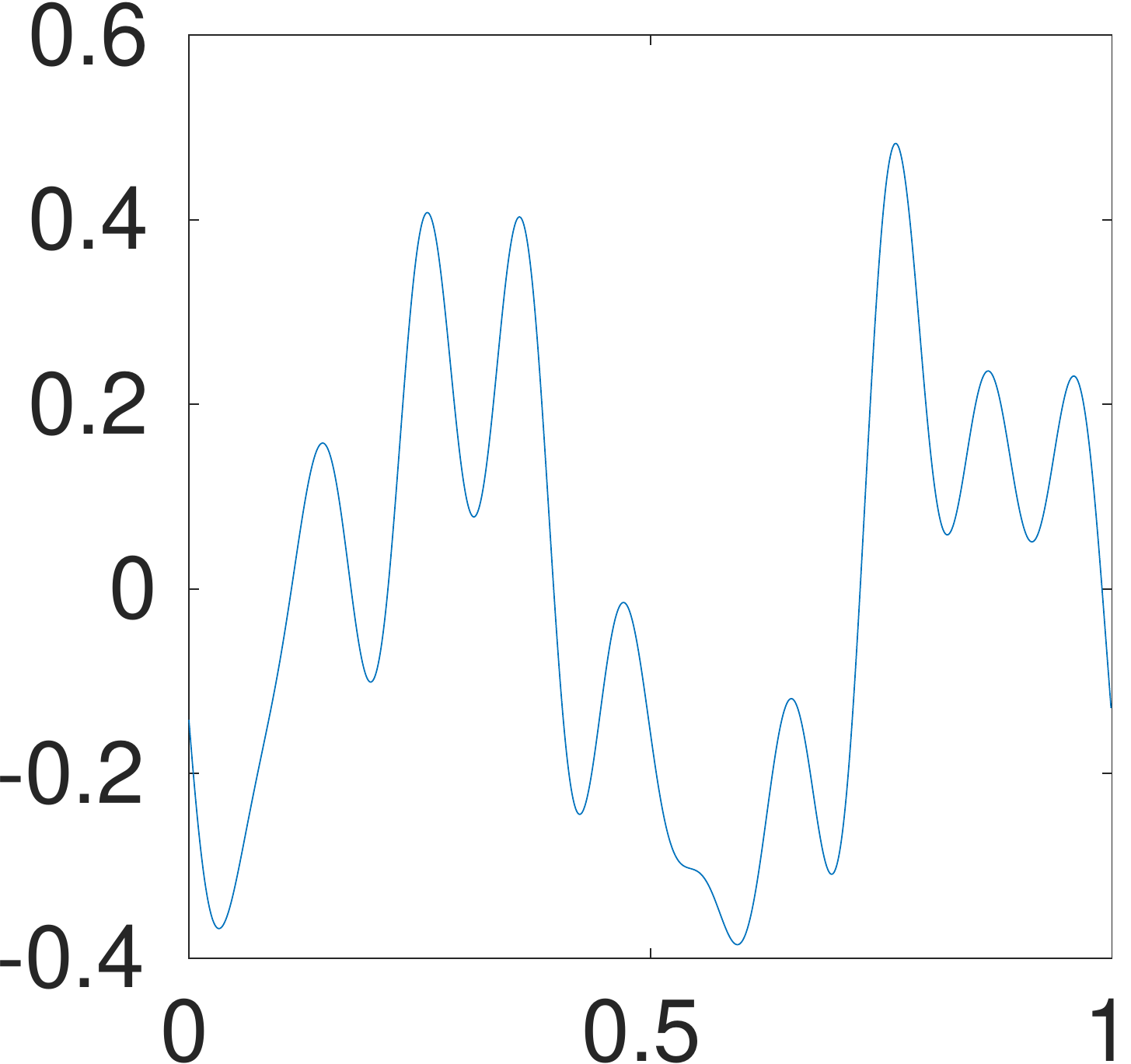}   &
      \includegraphics[width=1in]{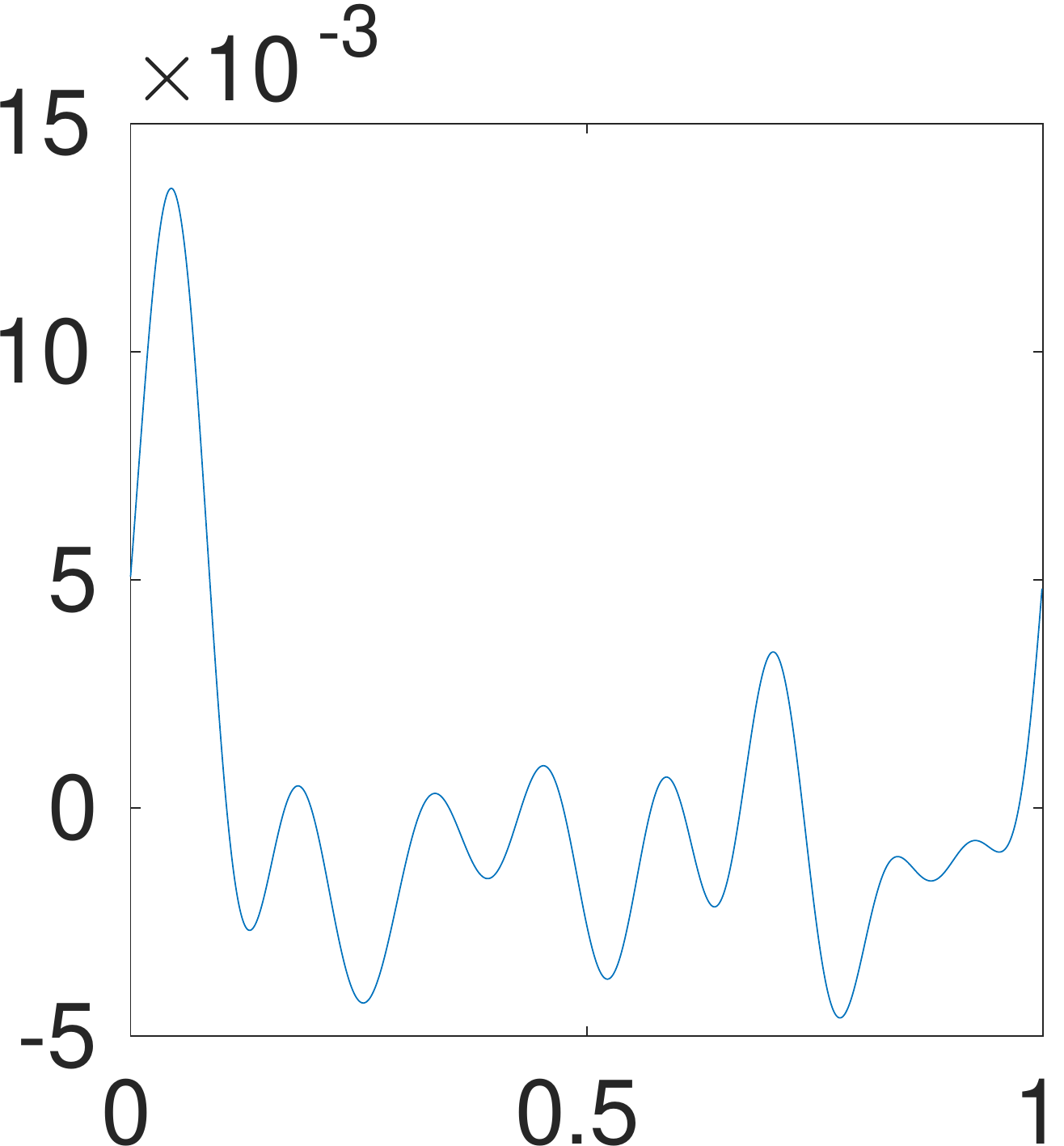}   \\
      \includegraphics[width=0.975in]{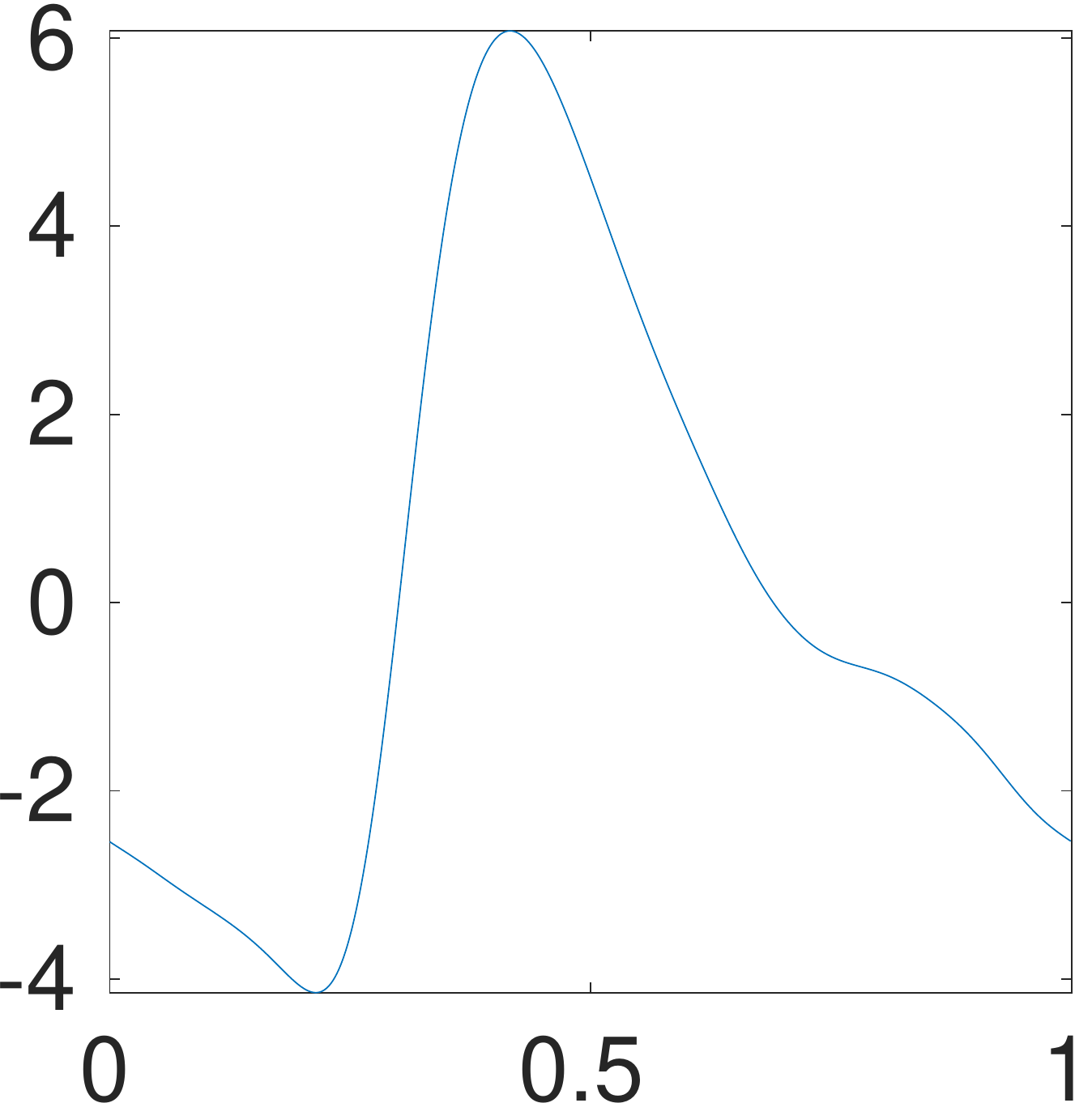}   &
      \includegraphics[width=1.05in]{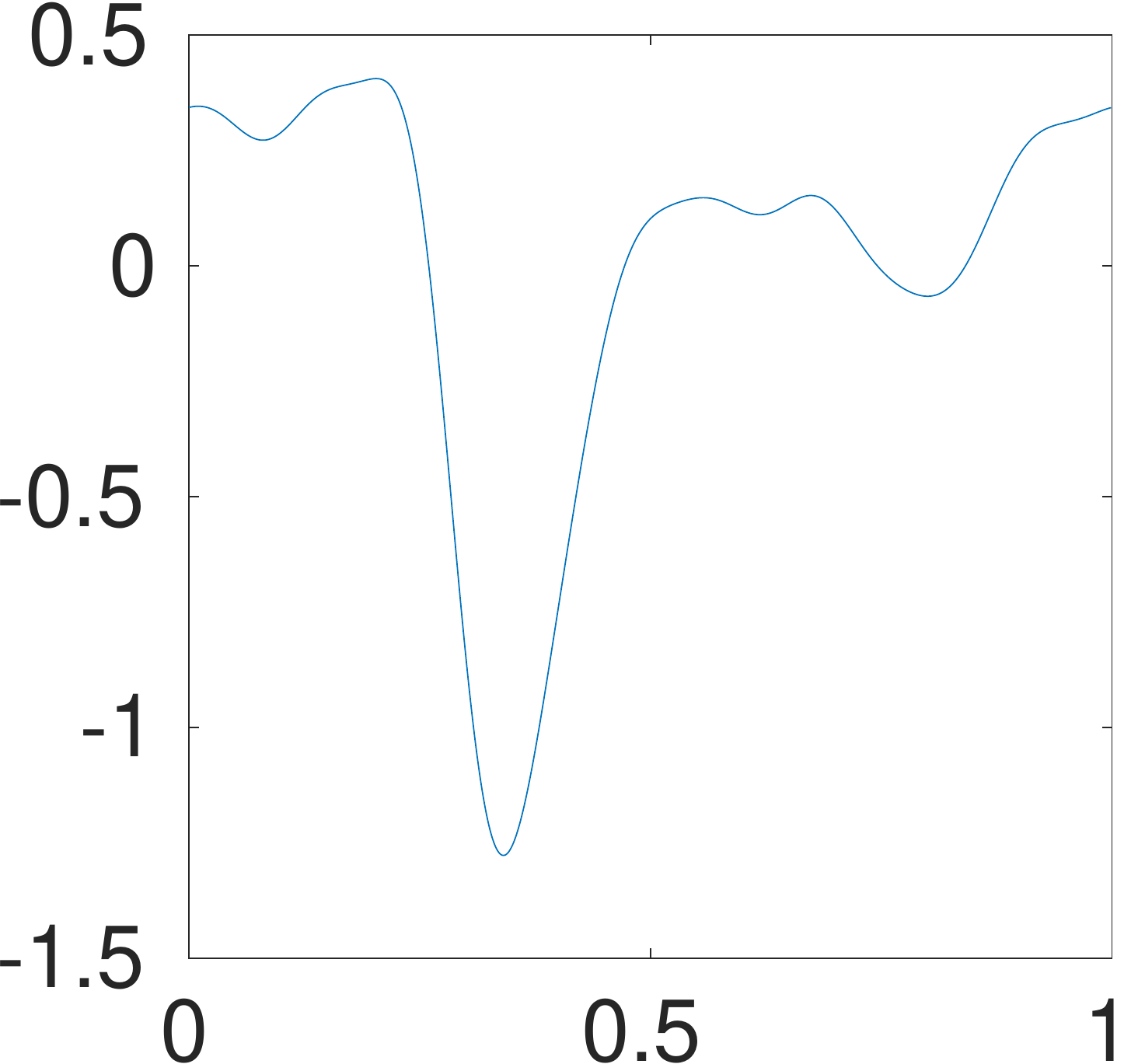}   &
      \includegraphics[width=1.15in]{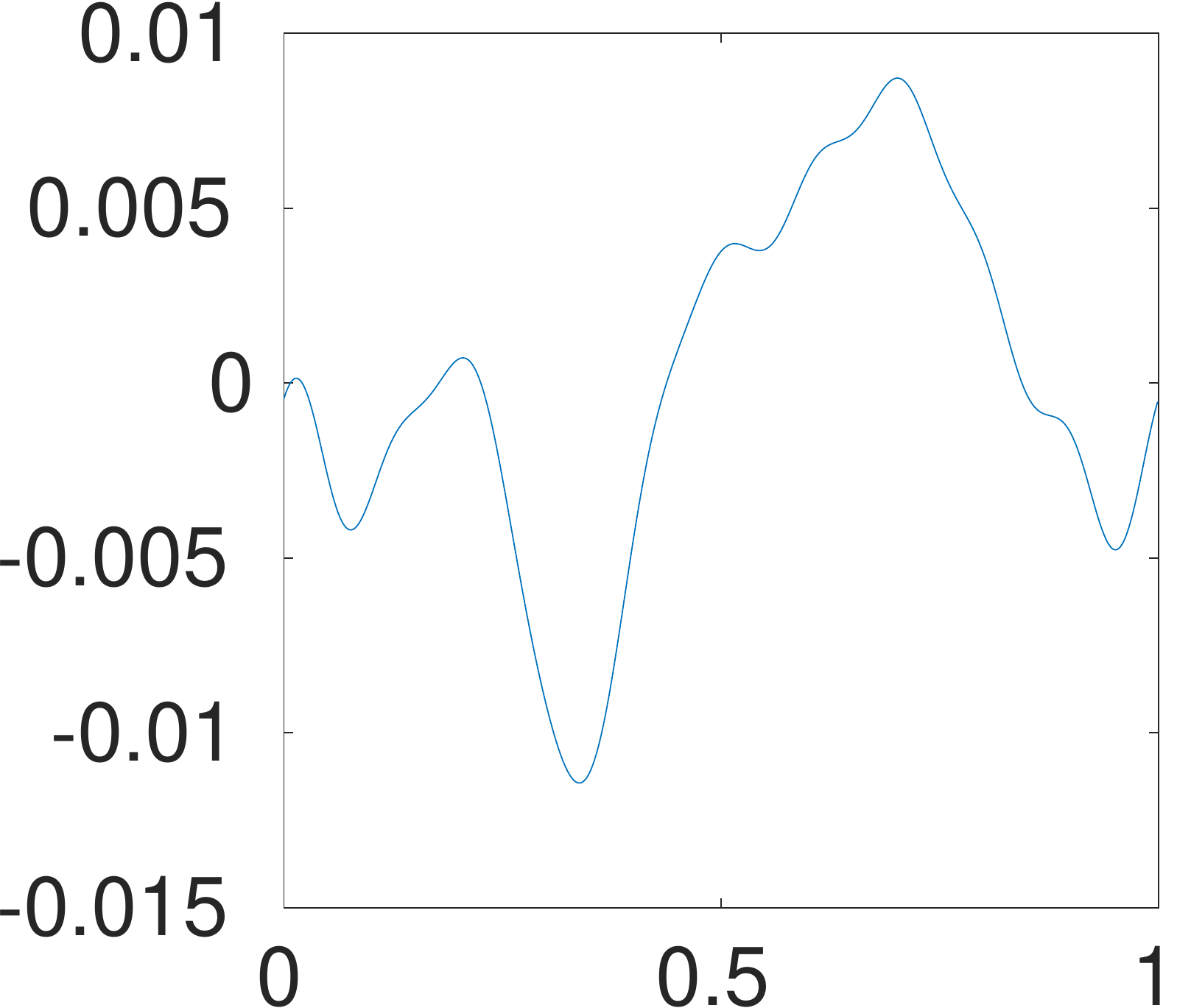}   &
      \includegraphics[width=1.05in]{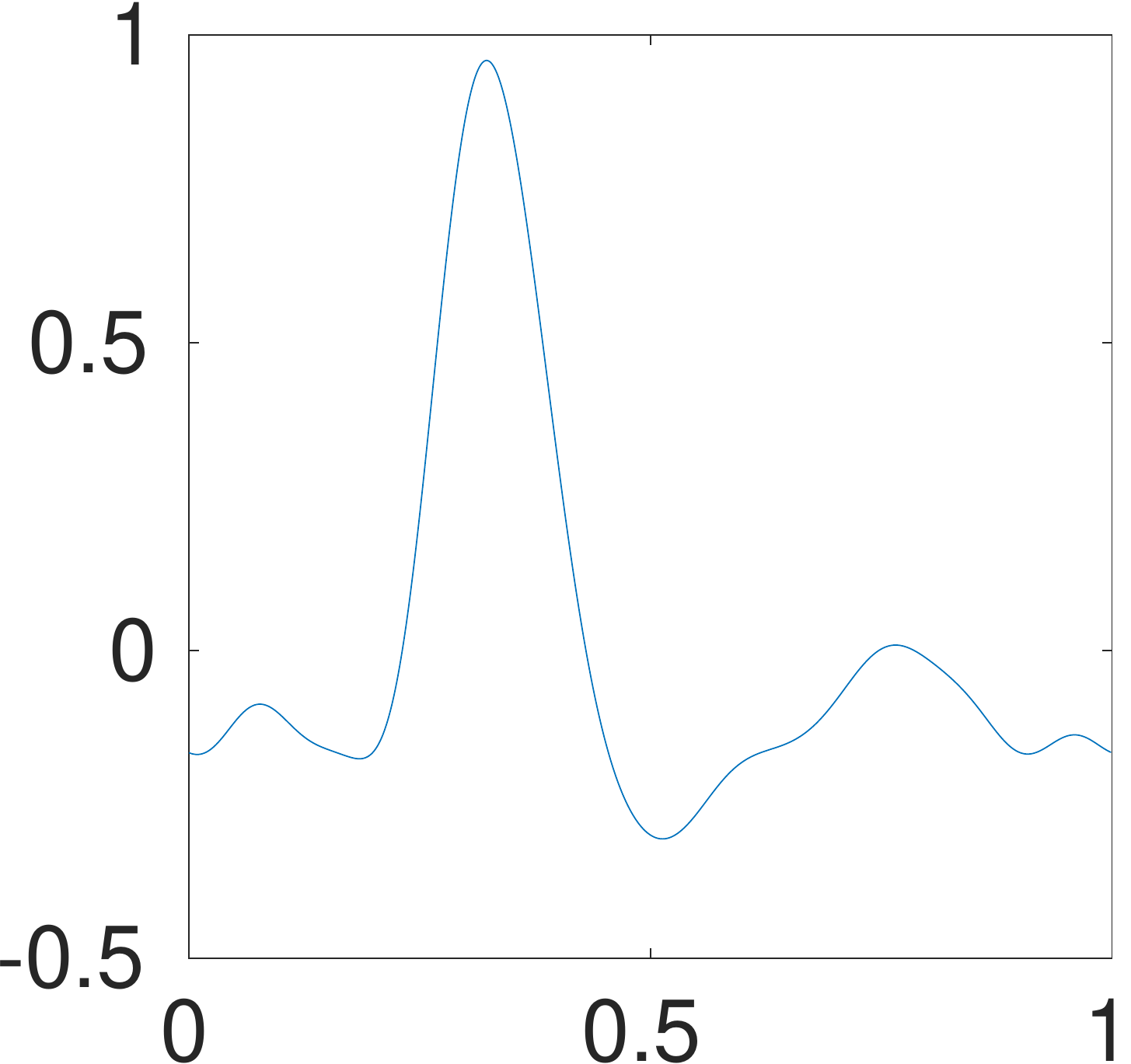}   &
      \includegraphics[width=0.975in]{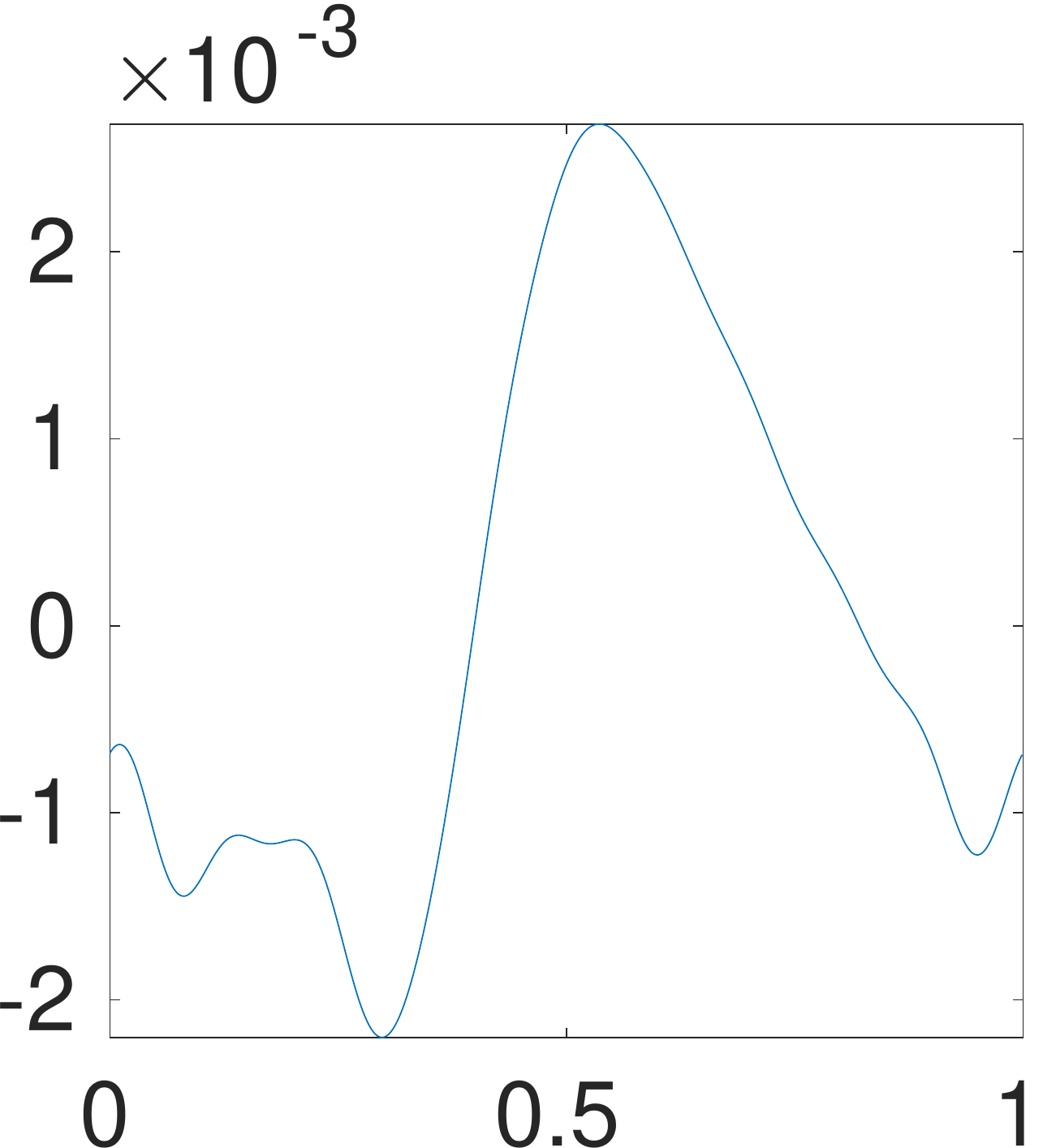}   
    \end{tabular}
  \end{center}
  \caption{Top row: estimated shape functions $a_{0,1}s_{c0,1}(2\pi t)$, $a_{1,1}s_{c1,1}(2\pi t)$, $a_{-1,1}s_{c-1,1}(2\pi t)$, $b_{1,1}s_{s1,1}(2\pi t)$, and $b_{-1,1}s_{s-1,1}(2\pi t)$ for the respiratory MIMF. Bottom row: estimated shape functions $a_{0,2}s_{c0,2}(2\pi t)$, $a_{1,2}s_{c1,2}(2\pi t)$, $a_{-1,2}s_{c-1,2}(2\pi t)$, $b_{1,2}s_{s1,2}(2\pi t)$, and $b_{-1,2}s_{s-1,2}(2\pi t)$ for the cardiac MIMF.}
\label{fig:15_3}
\end{figure}

\section{Conclusion}
\label{sec:con}

This paper proposed the multiresolution intrinsic mode function (MIMF) as a new model to simulate oscillatory time series with time-varying amplitudes and shapes. In the case of a superposition of several MIMFs, a novel multiresolution decomposition algorithm based on the idea of recursive diffeomorphism-based regression is proposed to separate the signal into individual MIMFs. The convergence and the robustness of recursive scheme has been theoretically and numerically proved. The application of MIMFs and MMD is not limited to decomposing signals into several components with well-differentiated instantaneous phase functions; they can also be used for time series denoising due to the robustness of the MMD algorithm; the multiresolution expansion coefficients and shape function series can also provide better features for adaptive time series analysis than traditional Fourier analysis and wavelet analysis. 

The MIMF model is inspired by the visual observation of nonlinear and non-stationary oscillatory time series with time-dependent amplitudes and shape functions. Theoretically it has been proved that the correlation of MIMF's with well-differentiated phase functions is asymptotically zero in the sense of recursive diffeomorphism-based regression. In other words, when the MIMF and MMD model is applied to analyze a time series, the resulting representation by MIMF's is unique. Numerically, it is shown that MIMF's match real oscillatory time series well, especially those from health data. To theoretically validate the MIMF model for oscillatory time series, a more fundamental but challenging approach is to understand the governing dynamics that generate the time series. For example, despite several attempts \cite{DAS2013490} over decades still it has not been successful to establish dynamical models faithfully describing real ECG signals -- not to mention the variability among different individuals. The MIMF model may serve as a prototype that separates the variability of complicated signals into simpler components, the generator of whose dynamics may be more accessible. 

Though adaptive time-series analysis is the main motivation of MIMF and MMD discussed in this paper, these models and algorithms can be naturally extended to higher dimensional spaces. Applications include atomic crystal images in physics \cite{Crystal,LuWirthYang:2016}, art investigation \cite{Canvas,Canvas2}, geology \cite{GeoReview,SSCT,977903}, imaging \cite{4685952}, etc. In higher dimensional spaces, the computational efficiency is a crucial issue. There are mainly two directions for future works for fast algorithms for MMD. A natural idea is to develop fast regression techniques to reduce the time for each iteration in the recursive scheme in MMD; a more challenging question is to develop new recursive schemes to update all the multiresolution expansion coefficients and shape function series simultaneously, instead of updating them one by one in each iteration in Algorithm \ref{alg:RDBR2} and \ref{alg:MMD}.

{\bf Acknowledgments.} 
H.Y. thanks Ingrid Daubechies for her fruitful discussion. 

\bibliographystyle{unsrt} 
\bibliography{ref}

\end{document}